\documentclass[a4paper,oneside, 11pt,reqno]{amsart}
\usepackage[
    a4paper,
    left=2.5cm,
    right=2.5cm,
    top=3.0cm,
    bottom=3.0cm
]{geometry}
\allowdisplaybreaks
\usepackage{amssymb,amsmath,mathtools,amsfonts,mathrsfs,calc,physics}
\usepackage{subfigure}
\usepackage{xcolor}
\usepackage[dvipsnames]{xcolor}
\usepackage{epsfig}
\usepackage[all]{xy}
\usepackage[utf8]{inputenc}
\usepackage{graphicx}
\usepackage{xcolor}
\usepackage{comment}
\usepackage{stmaryrd}
\usepackage{enumitem}
\usepackage{dynkin-diagrams}
\usepackage{hyperref} 
\usepackage{ragged2e}
\usepackage{vwcol}
\usepackage{paracol} 
\usepackage{lipsum}
\usepackage{xparse}
\usepackage[pro]{fontawesome5}
\usepackage{tcolorbox}
\usepackage[]{mdframed}
\newtcolorbox{mybox}[1]{colback=blue!3!white,colframe=blue!60!black,fonttitle=\bfseries, title=#1}
\usepackage{faktor}
\usepackage{footmisc} 
\usepackage{matlab-prettifier}
\hypersetup{colorlinks=true,linkcolor=blue,citecolor=RubineRed,urlcolor=blue}
\usepackage{anyfontsize}    
\usepackage[linesnumbered,ruled,noline]{algorithm2e} %\begin{algorithm}[H] 
\SetNlSty{textnormal}{}{:}   
\SetAlgoNlRelativeSize{-1}
\SetKwComment{Comment}{\textcolor{black}{$\triangleright$ }}{\textcolor{ForestGreen}{}}
\usepackage{matlab-prettifier}
\newcommand{\Rn}{\mathbb{R}^n}
\newcommand{\mc}[1]{\mathcal{#1}}
\newcommand{\mb}[1]{\mathbb{#1}}

\newtheorem*{claim}{Claim} 
\newenvironment{proofofclaim}%
{\begin{proof}}%
  {\end{proof}}
\newtheorem{theorem}{Theorem}[section] 
\newtheorem{lemma}[theorem]{Lemma}  
\newtheorem{corollary}[theorem]{Corollary}
\newtheorem{proposition}[theorem]{Proposition}

\theoremstyle{definition}
\newtheorem{example}[theorem]{Example}

\newtheorem{definition}[theorem]{Definition}
\newtheoremstyle{lemmastyle}% nombre
  {3pt}
  {3pt}
  {\itshape}
  {}
  {\normalfont\scshape}% 
  {.}
  {.5em}
  {}
\theoremstyle{lemmastyle}
\newenvironment{proofoflemma}%
{\begin{proof}}%
  {\end{proof}}
\title[Nonlocal Equations on Domains and Manifolds Without Symmetries]{Nonlocal Equations on Domains and Manifolds Without Symmetries}

\begin{document}

\maketitle

\begin{center}
\author{\textsc{Santiago Castellanos}\footnote{
Universidad de Santiago de Chile, Santiago de Chile.
Email: \email{santiago.castellanosg10@gmail.com}.\\
This article is based on the research conducted during the author's Ph.D. thesis \cite{Castellanos 1}.}}
\end{center}

\begin{abstract}
In this work, we establish existence and regularity results for solutions to nonlinear nonlocal equations of the form $f(\Delta)u = U(x,u)$ and $f(\Delta_{g})u = U(x,u)$ in Euclidean and Riemannian settings, respectively. We prove existence results across three distinct geometric scenarios: bounded domains in $\mathbb{R}^n$, perturbed equations on compact Riemannian manifolds, and noncompact Cartan-Hadamard manifolds under suitable geometric curvature conditions.
\vspace{0.2cm}

\justify
\end{abstract}

\section{Introduction}

In this article, we study generalized nonlocal equations, both linear and 
nonlinear, of the form
\begin{equation}\label{ecuacion no local}
    f(\Delta)u = U(x, u(x)),
\end{equation}
where $f$ is an arbitrary measurable real-valued function defined on 
$\mathbb{R}$, and $U$ is a 
prescribed nonlinearity. Our central objective is to prove that, under 
appropriate hypotheses on $f$ and $U$, equations of this form admit smooth 
solutions in general scenarios, entirely without the assumption of underlying symmetries. The investigation of nonlocal equations of type 
\eqref{ecuacion no local} is strongly motivated by their relevance in
contemporary mathematical physics, particularly in \emph{string theory} and 
\emph{nonlocal cosmological models} \cite{Kostelecky}. Among the most paradigmatic examples is the 
generalized \emph{bosonic string equation}
\begin{equation}
    \Delta e^{-c\Delta}\phi - U(x,\phi) = 0, \qquad c > 0.
\end{equation}
Another 
central example is the \emph{open $p$-adic string equation} for the scalar tachyon 
field,
\begin{equation}
    p^{a\partial_t^2}\phi = \phi^p, \qquad a > 0,
\end{equation}
studied rigorously in \cite{Vladimirov 2}. 
At the level of generality of \eqref{ecuacion no local}, neither pseudo-differential operators in the 
classical sense of H\"{o}rmander \cite{Hormander} nor the theory of classical symbols or conditions are necessary.
We seek a direct meaning for the operator 
$f(\Delta)$, or its equivalent on Riemannian manifolds; to this end, we represent the nonlocal part of the equation \eqref{ecuacion no local} via the Fourier transform, functional calculus, and spectral theory, following the approach initiated by G\'{o}rka, Prado, 
and Reyes in \cite{Gorka 6}. 
Once interpreted, we proceed to solve the nonlocal equation in different geometrical and topological settings without symmetry; namely, bounded Euclidean domains, compact Riemannian manifolds, and complete noncompact Cartan-Hadamard manifolds on which the operator $f(\Delta) + I$ (resp., $f(\Delta_{g}) + I$) is well defined and acts naturally.

In the Euclidean context for bounded domains $\Omega \subset \Rn$ and the space $\Rn$, the conditions imposed on $f: \mathbb{R} \to \mathbb{R}$ are the following, \\
$(P)$~~\text{The function} $s \mapsto f(-s^{2})$~\text{is nonnegative}.\\
$(E_{\beta})$ There exist real numbers $\beta, R, M > 0$ such that
$$M(1+|\xi|^{2})^{\frac{\beta}{2}} \leq f(-|\xi|^{2})~~\forall \xi~\text{such that}~ |\xi|> R.$$
Functions satisfying these conditions are called $\mathcal{G}^{\beta}$-symbols, and we denote the class of such functions by $\mathcal{G}^{\beta}$. 
In the Euclidean setting, the nonlocal operator $f(\Delta)$ is interpreted through the Fourier transform as
$
f(\Delta)u=\mathcal{F}^{-1}\bigl(f(-|\xi|^2)\mathcal{F}(u)(\xi)\bigr),
$
so that $f$ is evaluated at $-|\xi|^2\in(-\infty,0]$.
We work with Riemannian manifolds $(M,g)$, whether on compact connected manifolds without boundary, or complete noncompact Riemannian manifolds, the nonlocal equation we work with varies slightly with respect to \eqref{ecuacion no local}, since the nonlocal term is now $f(\Delta_{g})$, where $\Delta_{g} = -\text{div}_{g}(\nabla_{g})$ is the positive Laplace-Beltrami operator, and the conditions on $f$ are formulated in terms of the discrete spectrum of $\Delta_g$.

In this article, we study the existence of solutions in the absence of symmetry across three main contexts:
\renewcommand{\theenumi}{\Roman{enumi}}
\begin{enumerate}
    \setcounter{enumi}{0}
        \item For the case of a bounded domain $\Omega \subset \mathbb{R}^n$ with Lipschitz boundary, we first recall the ambient Euclidean space without symmetry: for a fixed symbol $f \in \mc{G}^{\beta}$, we consider the Hilbert space $\mc{H}^{\beta}(f)$ given by all measurable real-valued functions $g$ on $\mathbb{R}^n$ whose Fourier transform $\mathcal{F}(g)$ exists and satisfies
    \begin{equation}
   \int_{\mathbb{R}^n} \left[ 1 + f\left(-|\xi|^2\right) \right]^2 \left| \mathcal{F}(g)(\xi) \right|^2 \, d\xi < \infty.
    \end{equation}
This space satisfies $\mathcal{H}^{\beta}(f) \subseteq L^2(\mathbb{R}^n)$, and allows us to give an interpretation of $f(\Delta) + I$. 
On a bounded domain $\Omega \subset \mathbb{R}^n$ with Lipschitz boundary, the Fourier transform is applied to the zero extension $\widetilde{u}$ of functions $u$ defined on $\Omega$. In the case of a bounded domain $\Omega \subset \mathbb{R}^n$ with Lipschitz boundary, we use Fourier transform on extensions by zero of functions defined on $\Omega$. We define a subfamily of symbols $\mc{G}^{\beta,\gamma} \subset \mc{G}^{\beta}$ which allows us to control the nonlocal operator on $\Omega$. 
Generally for $f\in \mc{G}^{\beta, \gamma}$,  the nonlocal operator on $\Omega$ is defined by
$$
f(\Delta_\Omega)u
=\mc{R}_{\Omega}f(\Delta)\widetilde{u}=
\left(\mathcal{F}^{-1}\bigl(f(-|\xi|^2)\mathcal{F}(\widetilde{u})(\xi)\bigr)\right)\Big|_{\Omega}.
$$
Accordingly, for a fixed symbol $f\in\mathcal{G}^{\beta, \gamma}$, the associated Hilbert space on $\Omega$ is defined by
\begin{equation}
\mathcal{H}_{0}^{\beta/2}(\Omega;f)
=
\overline{\{\widetilde{\eta} \mid \eta \in C^{\infty}_{0}(\Omega)\}}^{\mc{H}^{\beta/2}(f)}.
\end{equation}
in which $\widetilde{\eta}$ is the zero-extension on $\mb{R}^{n}$ of $\eta$. 
This space is, intuitively, the subspace of $\mathcal{H}^{\beta/2}(f)$ consisting of functions supported in $\overline{\Omega}$ after zero extension.

\item For compact Riemannian manifolds $(M,g)$ without boundary, we use the fact that the spectrum of $\Delta_{g}$ is a discrete set of non-negative real numbers $0 \leq \lambda_1 <\lambda_2 \leq \dots \to \infty$, admitting an orthonormal basis $\{\omega_{n} \mid n \in \mathbb{N}\}$ of eigenfunctions in $L^{2}(M,g)$. Here, without the action of any isometric group, we work with the fractional Sobolev space
\begin{equation}
H^s(M,g) = \left\{ h=\sum_{i=1}^{\infty}\alpha_{i}\omega_{i} ~\Big|~ \sum_{i=1}^{\infty}|\alpha_{i}|^{2}(1+\lambda_{i})^{s} < \infty \right\}, \quad s \in \mathbb{R}_{>0},
\end{equation}
and define the space
    \begin{equation}\label{espacio de Sobolev en verson sumas para intro}
        \widetilde{H}_{f}(M,g) = \left\{h=\sum_{i=1}^{\infty} \alpha_{i}\omega_{i} ~\Big|~ \sum_{i=1}^{\infty}|\alpha_{i}|^{2}(1 + f(\lambda_{i}))^{2} <  \infty   \right\}.
    \end{equation}
    The growth conditions imposed on $f$ will guarantee that the embedding $\widetilde{H}_{f}(M,g) \hookrightarrow H^{2k}(M,g)$ is continuous.
    
    \item For noncompact Riemannian manifolds, we focus on Cartan--Hadamard $n$-manifolds. Under the assumption that the sectional curvature bounds diverge at infinity, $K_{p}(r) \to -\infty$ as $r \to \infty$, Donnelly's theorem \cite[t.1.1]{Donnelly} ensures that $\Delta_g$ retains a discrete spectrum. This property allows us to utilize the space $\widetilde{H}_{f}(M,g)$ defined in \eqref{espacio de Sobolev en verson sumas para intro}, provided an additional eigenfunction regularity condition is met to recover the necessary compact embeddings.
\end{enumerate}

\justify

We summarize our main theorems below. First, we study nonlocal equations on bounded domains 
$\Omega \subset \mathbb{R}^n$ for symbols $f \in \mc{G}^{\beta,\gamma}\subset \mc{G}^{\beta}$.
The geometry of the domain allows us to introduce the space $\mc{H}^{\beta/2}_{0}(\Omega;f)$ to 
exploit the compact Rellich--Kondrachov embedding on $\Omega$ of \cite[t.1.22]{Rosenberg}. We obtain the following:

\begin{theorem}[\ref{teorema en Rn, para dominios acotados}]
Let $\Omega\subset\mathbb{R}^n$ be an open, bounded domain with Lipschitz boundary. Let $f\in\mathcal{G}^{\beta,\gamma}$ with $\gamma\geq\beta>n$, 
$n>1$. Let $U:\Omega\times\mathbb{R}\to\mathbb{R}$ be a Carathéodory non-linearity which is a $C^{1}$ in the second variable, and assume there exist constants $\alpha>1$, $\mu\ge 1$, $C>0$, and a function $h\in L^{2}(\Omega)$ such that
\begin{equation}\label{desigualdad de caso acotado intro}
|U(x,y)+y| \le C(|h(x)|+|y|^{\alpha}), \qquad \left|\frac{\partial}{\partial y}\bigl(U(x,y)+y\bigr)\right| \le C(1+|y|^{\mu}).
\end{equation}
Then there exist $0 < \tau < 1$ and $0 < \rho_{\tau} < 1$ such that whenever $\norm{h}_{L^{2}(\Omega)} < \rho_{\tau}$, there exists a solution $u \in \mc{D}(L_{\Omega})$ of the non-linear equation $f(\Delta_{\Omega})u-U(\cdot,u)=0$.
\end{theorem}
The second result, which also relies on classical Rellich-Kondrachov embeddings, generalizes the localized Euclidean 
result of \cite[t.3.3]{Gorka 6} to compact Riemannian manifolds.
In this case, equation \eqref{ecuacion no local}
is perturbed by a term $\delta\psi L(\phi)$, where 
$L: L^2(M,g) \to L^2(M,g)$ is a bounded linear operator and 
$\psi \in C^\infty(M)$. We establish the following

\begin{theorem}[\ref{teorema del caso perturbado para Riemannian manifolds}]
Let $(M,g)$ be a compact, connected and oriented Riemannian $n$-manifold without boundary. Let $U : M \times \mathbb{R} \to \mathbb{R}$
be a Carathéodory non-linearity which is a  $C^{1}$ in the second variable, $f$ a 
continuous, nonnegative real-valued function satisfying condition (\ref{codicion de elipticidad de una fucnion real valuada para Riemannian manifolds}). Let $L : L^{2}(M,g) \to L^2(M,g)$ be a linear and
bounded operator.
%be a linearbounded operator and $C^{1}$
Suppose there exist constants $\alpha > 1$, $\beta \ge 1$, $C > 0$, and a $\delta > 0$ sufficiently small, 
a function $\psi \in C^{\infty}(M)$, $\chi \in L^{\infty}(M,g)$ 
and a function $h\in L^2(M,g)$, such that the following holds
\begin{equation}\label{desigualdad numero 1 del caso pertubado intro}
    |U(x,y) + y| \le C\,\Bigl(|h(x)| + |y|^\alpha\Bigr),
\end{equation}
\begin{equation}\label{desigualdad numero 2 del caso pertubado intro}
    \left|\frac{\partial}{\partial y}\Bigl(U(x,y) + y\Bigr)\right| \le C\,\Bigl(1+ |\chi(x)| + |y|^\beta\Bigr).
\end{equation}
Then, if $\norm{h}_{L^{2}(M,g)}$ is sufficiently small, there exists a solution $\phi\in \widetilde{H}_f(M,g)$ of the non-linear equation
\begin{equation}\label{ecuacion caso perturbado Riemanniano intro}
    f(\Delta_{g})\phi - \delta\psi L(\phi) - U(x,\phi)= 0
\end{equation}
\end{theorem}

The third result extends the foregoing theory to a class of noncompact 
Riemannian manifolds. A Cartan-Hadamard $n$-manifold is a complete, oriented, 
simply connected Riemannian manifold with non-positive (or negative) sectional curvature $K$; 
such manifolds are diffeomorphic to $\mathbb{R}^n$ by the Cartan--Hadamard 
theorem \cite{DoCarmo}.
Under the condition $K_{p}(r) \to -\infty$ as $r \to \infty$, where
\begin{equation}
K_{p}(r) = \sup_{x \in M} \left\{ K(x, \pi) \mid d_{g}(p,x) \geq r, \, \pi = \text{span}\{u,v\} \subset T_x M \text{ is a 2-plane} \right\}.
\end{equation}
Donnelly's theorem 
\cite[t.1.1]{Donnelly} guarantees that $\Delta_g$ has discrete spectrum. 
An additional eigenfunction regularity condition, see \eqref{condicion sombrero}, on 
the growth of $\|\nabla^k\omega_j\|_{L^\infty(M)}$ is required to recover a compact embedding. We obtain the following

\begin{theorem}[\ref{teorema finnal para el caso no compacto}]
     Let $(M,g)$ be a Cartan–Hadamard $n$-manifold 
      satisfying $K_{p}(r) \to -\infty$ as $r \to \infty$ for any $p \in M$, and assume condition \eqref{condicion sombrero} in the case $\ell=0$ for some real $s>0$. 
      Suppose that $f$ satisfies \eqref{codicion de elipticidad de una fucnion real valuada para Riemannian manifolds},
      that $U: M \times \mathbb{R} \to \mathbb{R}$ is a Carathéodory non-linearity which is a $C^{1}$ in the second variable, and that there exist 
      constants $\alpha > 1$, $\beta \geq 1$, $C > 0$, and a 
      function $h \in L^2(M,g)$ such that the following inequalities hold
\begin{equation}\label{desigualdad numero 1 del caso completo intro}
|U(x, y) + y| \leq C \left(|h(x)| + |y|^\alpha \right), \quad 
    \left| \frac{\partial}{\partial y} \left( U(x, y) + y \right) \right| \leq C ( 1 + |y|^\beta ).
\end{equation}
Then, there exists $\varepsilon \in (0,1]$ and $0 < \rho_{\varepsilon} < 1$ such that whenever $\norm{h}_{L^2(M)} < \rho_{\varepsilon}$, there exists a solution $\psi \in \widetilde{H}_f(M,g)$ to the non-linear equation $f(\Delta_g)\psi - U(x, \psi) = 0.$
\end{theorem}
The article is structured as follows: in Section \hyperref[Preliminaries and main results]{(II)}, we present the necessary preliminaries regarding nonlocal theory in the Euclidean case, discussing essential functional-analytic tools, reviewing foundational results from \cite{Gorka 6}, and establishing the precise interpretation of the operator $f(\Delta)$. In Section \hyperref[Nonlocal equations on Bounded Domains 1]{(III)}, we address the nonlocal theory in the context of bounded domains $\Omega \subset \mathbb{R}^{n}$, introducing the space $\mc{H}^{\beta/2}_{0}(\Omega;f)$ and culminating in the proof of Theorem \ref{teorema en Rn, para dominios acotados}. 

In Section \hyperref[Nonlocal Equations on Riemannian Manifolds]{(IV)}, we develop our study of nonlocal theory on compact Riemannian manifolds. We discuss the rigorous representation of the operator $f(\Delta_{g})$ via spectral theory and establish Theorem \ref{teorema del caso perturbado para Riemannian manifolds}, which proves the existence of solutions for perturbed nonlocal equations on compact manifolds. Finally, in Section \hyperref[seccion de Cartan-Hadamard equations]{(V)}, we extend this theory to the noncompact setting, focusing specifically on Cartan--Hadamard manifolds under the asymptotic curvature condition $K_p(r) \to -\infty$ and the eigenfunction control condition \eqref{condicion sombrero}. This section includes the proof of Theorem \ref{teorema finnal para el caso no compacto} and concludes with the detailed study of a concrete example of a Cartan--Hadamard manifold that satisfies all the hypotheses of our main theorem.

\vspace{1cm}

\section{Preliminaries}\label{Preliminaries and main results}

We introduce appropriate function spaces to interpret the nonlocal term that appears in Equation (\ref{ecuacion no local}). Specifically, our goal in this section is to give meaning to the formal operator
\begin{equation}\label{operador no local}
    L = f(\Delta) + \text{Id}
\end{equation}
where $f$ is a measurable function, as an operator defined on (a subspace of) 
$L^2(\mathbb{R}^n)$. We will call $f$ the "symbol"; however, it need not belong to Hörmander's class of symbols\cite{Hormander}, since the framework developed here is more general and self-contained. Our approach is based on a direct functional calculus via the Fourier transform, which does not require $f$ to be smooth nor to satisfy the polynomial growth estimates characteristic of Hörmander's theory. Instead, it will suffice that $f$ satisfies the following conditions, \\
$(P)$~~\text{The function} $s \mapsto f(-s^{2})$~\text{is nonnegative}.\\
$(E_{\beta})$ There exist real numbers $\beta, R, M > 0$ such that
\begin{equation*}\label{elliptic condition over f}
M(1+|\xi|^{2})^{\frac{\beta}{2}} \leq f(-|\xi|^{2}), \quad \forall \xi \in \mb{R}^{n}~\text{with}~ |\xi|> R.
\end{equation*}
Functions satisfying these conditions are called $\mathcal{G}^{\beta}$-symbols. The condition $(E_{\beta})$ is the ellipticity condition for a pseudo-differential operator. (see \cite{Gorka 6})

\begin{definition}
    Fix $f \in \mathcal{G}^{\beta}$. The space $\mathcal{H}^{\beta}(f)$ is given by all real-valued functions $g$ on $\mathbb{R}^n$ such that $g$ is measurable, the Fourier transform $\mathcal{F}(g)$ exists, and
    \begin{equation}
     \int_{\mathbb{R}^n} \left[ 1 + f\left(-|\xi|^2\right) \right]^2 \left| \mathcal{F}(g)(\xi) \right|^2 \, d\xi < \infty.
    \end{equation}
\end{definition}

\justify
This space satisfies $\mathcal{H}^{\beta}(f) \subseteq L^2(\mathbb{R}^n).$ The reason for defining this space is explained by the following computation. If we assume that $f$ is a real-analytic function, one can consider $f$ as a power series centered at $x=0$, which we may evaluate at $\Delta$, giving the formal expression
\begin{equation*}
f(\Delta)u=\sum_{n=0}^{\infty}\frac{f^{(n)}(0)}{n!}\Delta^{n}u 
\end{equation*}
Applying the Fourier transform to this equation yields
\begin{align*} 
\mathcal{F}(f(\Delta)u)(\xi)&=\sum_{n=0}^{\infty}\frac{f^{(n)}(0)}{n!}\mathcal{F}(\Delta^{n}u)= \sum_{n=0}^{\infty}\frac{f^{(n)}(0)}{n!}(-|\xi|^{2})^{n}\mathcal{F}(u) = f(-|\xi|^{2})\mathcal{F}(u)(\xi). 
\end{align*}
Therefore we define the operator $f(\Delta)u$ by
\begin{equation}\label{forma de operador f}
    f(\Delta)u = \mathcal{F}^{-1}(f(-|\xi|^2)\mathcal{F}(u)(\xi)).
\end{equation}
If $u \in \mathcal{H}^{\beta}(f)$, then the function $f(-|\xi|^2)\mathcal{F}(u) \in L^2(\mathbb{R}^n)$, and $f(\Delta)u$ is well defined, so that $f(\Delta)$ is a linear operator from $\mathcal{H}^{\beta}(f)$ to $L^2(\mathbb{R}^n)$.

\begin{definition}\label{espacio Hb permite la interpretacion de L}
    Let $f$ be a $\mathcal{G}^{\beta}$-symbol. The operator $L = f(\Delta) + \text{I}$ acts on functions $u \in \mathcal{H}^{\beta}(f)$ as follows
    \begin{equation}\label{operador no local L}
        Lu = \mathcal{F}^{-1} \left( \mathcal{F}(u) + f(-|\xi|^2)\mathcal{F}(u) \right)
    \end{equation}
Since $u \in \mathcal{H}^{\beta}(f)$, we have $Lu \in L^2(\mathbb{R}^n)$.
\end{definition}

A function used repeatedly is
\begin{equation}\label{funcion K}
 \widehat{K}(\xi)=\cfrac{1}{1+f(-|\xi|^{2})}   
\end{equation}
In \cite[l.2.1]{Gorka 6} it is proved that for $f$ a $\mathcal{G}^{\beta}$-symbol with $\beta > n/2$, one has $\widehat{K} \in L^{2}(\mathbb{R}^{n})$. This implies there exists a function $K_{f}$ such that $\mathcal{F}(K_{f})=\widehat{K}$.

\begin{proposition}\label{espacio Hillbert Hb}
A measurable function $g$ belongs to the space $\mathcal{H}^{\beta}(f)$, with $\beta > n/2$, if and only if $g = K_f * \tilde{g}$ for some $\tilde{g} \in L^2(\mathbb{R}^n)$. Moreover, if we endow $\mathcal{H}^{\beta}(f)$ with the inner product
\begin{equation}
    \langle g_1, g_2 \rangle = \int_{\mathbb{R}^n} \left( 1 + f(-|\xi|^2) \right)^2 \mathcal{F}(g_1)(\xi) \overline{\mathcal{F}(g_2)(\xi)} \, d\xi
,
\end{equation}
then $\mathcal{H}^{\beta}(f)$ is a Hilbert space and the map $\mathcal{K} : L^2(\mathbb{R}^n) \to \mathcal{H}^{\beta}(f)$ given by $ \mathcal{K}(g) = K_f * g$\label{isomorfismo isometrico entre los espacios Hb}
, is an isometric isomorphism (a bijective linear isometry).
\end{proposition}

\begin{lemma}\label{Lemma de contencion epscios Hb}
   \cite[l.2.2]{Gorka 6}. Let $f$ be a $\mathcal{G}^{\beta}$-symbol. We have:\\
i) For each $s \in \mathbb{R}$ with $s \leq \beta$, the embedding 
\begin{equation}
  \mathcal{H}^{\beta}(f) \hookrightarrow H^{s}(\mathbb{R}^n)
\end{equation}
holds. In other words,
$
\| u \|_{H^{s}(\mathbb{R}^n)} \leq C(s, \beta) \| u \|_{\mathcal{H}^{\beta}(f)},
$
for some constant $C(s, \beta) > 0$.\\
ii) For each integer $k \geq 1$ such that $\frac{n}{2} + k < s \leq \beta$, one has $\mathcal{H}^{\beta}(f) \hookrightarrow C^{k}(\mathbb{R}^n)$.\\
iii) For each $0 < \alpha < 1$ such that $\frac{n}{2} + \alpha = s \leq \beta$, the embedding $\mathcal{H}^{\beta}(f) \hookrightarrow C^{\alpha}(\mathbb{R}^n)$
holds, where $C^{\alpha}(\mathbb{R}^n)$ is the space of H\"{o}lder continuous functions of order $\alpha$.
\end{lemma}

\subsection{The Linear Equation}

\begin{definition}
    A function $u : \mathbb{R}^n \to \mathbb{R}$ is \emph{real-analytic} if it is infinitely differentiable and can be represented locally by a convergent power series. Thus, for each point $z \in \mathbb{R}^n$ there exists an open neighborhood $U$ of $z$ such that
\begin{equation}
    u(x) = \sum_{\alpha} \frac{ \partial^\alpha u(z)}{\alpha!} (x - z)^\alpha,
\end{equation}
the sum runs over all multi-indices. The space of real-analytic functions is denoted $C^\omega(\mathbb{R}^n)$. 
\end{definition}

\begin{theorem}\label{teorema caso lineal}\cite[t.3.1]{Gorka 6}. 
    Let $f$ be a $\mathcal{G}^{\beta}$-symbol. For each $g \in L^2(\mathbb{R}^n)$ there exists a unique solution $u_g \in \mathcal{H}^\beta(f)$ to the linear problem
\begin{equation}\label{ecuacion lineal nolocal}
    L u = g.
\end{equation}
Moreover, the following identity holds
\begin{equation}
  \|u_g\|_{\mathcal{H}^\beta(f)} = \|g\|_{L^2(\mathbb{R}^n)}.  
\end{equation}
\end{theorem}

\begin{theorem}\label{funcion real-analitica}
    Suppose $f \in \bigcap_{\beta > 0}\mathcal{G}^{\beta}$ and that there exists $c>0$ such that
    \begin{equation}
        \lim_{m \to \infty} \frac{c^{m}}{m!}\int_{0}^{\infty}\frac{r^{2m+n-1}}{[1+f(-r^{2})]^{2}}dr=0.
    \end{equation}
Then the solution $u$ of the linear problem $L u = g$ with $g \in L^2(\mathbb{R}^n)$ satisfies $u \in C^\omega(\mathbb{R}^n)$.
\end{theorem}

\newpage 

\subsection{The Nonlinear Equation}

\justify

In this section we study the nonlinear equation
\begin{equation}\label{ecuacion basica no-lineal}
    f(\Delta)u - U(\cdot , u)=0
\end{equation}
Before analyzing this equation in the general case, let us consider a nonlocal term of the form $U(x,y)=y + \varphi(x)V(x,y)$, where $x \in \mathbb{R}^{n}$ and $y \in \mathbb{R}$, and $\varphi \in C^{\infty}_{0}(\mathbb{R}^{n})$. This kind of equation is called the \emph{localized case},

\begin{theorem}\label{caso localizado}\cite[t.3.3]{Gorka 6}.
    Suppose $f \in \mathcal{G}^\beta$ with $\beta > \frac{n}{2}$, $n > 1$, and consider the function $U_\delta$, with $\delta > 0$, given by:
\begin{equation}\label{primera ecuacion no lineal}
U_\delta(x, y) = -y + \delta \varphi(x)V(x, y),
\end{equation}
where $\varphi \in C_0^\infty(\mathbb{R}^n)$ and $V$ is a $C^1$-function on $\mathbb{R}^n \times \mathbb{R}$. Assume there exist a constant $\alpha \geq 1$, a constant $C > 0$ and a function $h \in L^2(\mathbb{R}^n)$ such that the following inequalities hold:
\begin{equation}\label{primera desigualdad caso localizado}
|V(x, y)| + \left|\frac{\partial}{\partial x_i} V(x, y)\right| \leq C(h(x) + |y|^\alpha), \quad i = 1, \ldots, n,
\end{equation}
\begin{equation}\label{segunda desigualdad caso localizado}
\left|\frac{\partial}{\partial y} V(x, y)\right| \leq C(1 + |y|^\alpha).
\end{equation}
Then, for $\delta$ sufficiently small, there exists a solution $u \in \mathcal{H}^\beta(f)$ of the problem (\ref{ecuacion basica no-lineal}) with the nonlinearity $U$ given by (\ref{primera ecuacion no lineal}).
\end{theorem}

\begin{corollary}
    Suppose $U$ is given by (\ref{primera ecuacion no lineal}), the hypotheses of Theorem \ref{funcion real-analitica} are satisfied, and the growth assumptions of Theorem \ref{caso localizado} hold. Then there exists a $C^\omega$ solution to the problem (\ref{ecuacion basica no-lineal}).
\end{corollary}

Now we work with a nonlinearity $U : \mathbb{R}^{n} \times  \mathbb{R} \to \mathbb{C}$ which is \emph{spherically symmetric} in $x$, i.e., it satisfies $U(|x|,y)=U(x,y)$.

\begin{definition}
\begin{equation}
    L^2_r(\mathbb{R}^n) = \overline{C_{0,r}^\infty(\mathbb{R}^n) }^{L^2(\mathbb{R}^n)}~~,~~
H^s_r(\mathbb{R}^n) = \overline{C_{0,r}^\infty(\mathbb{R}^n) }^{H^s(\mathbb{R}^n)},
\end{equation}  
where $C_{0,r}^\infty(\mathbb{R}^n)$ denotes the space of smooth radial compactly supported functions on $\mathbb{R}^n$. Denote by $\mathcal{H}^\beta_r(f)$ the closed subspace of radial functions in $\mathcal{H}^\beta(f)$.
\end{definition}

\begin{theorem}\label{problema lineal para el caso cilindrico}
    If $g \in L^2_r(\mathbb{R}^n)$, then the unique solution $u \in \mathcal{H}^\beta(f)$ of the problem \eqref{ecuacion lineal nolocal} is a radial function.  
Moreover, the identity holds:
\begin{equation}
\|u\|_{\mathcal{H}^\beta_r(f)} = \|g\|_{L^2_r(\mathbb{R}^n)}.
\end{equation}
\end{theorem}

\begin{theorem}\label{Caso Radial}\cite[t.3.5]{Gorka 6}.
    Suppose $\alpha > 1$ is a constant, $f \in \mathcal{G}^\beta$ for some $\beta > \frac{n}{2} \left(\frac{\alpha - 1}{\alpha}\right)$, and $U$ is spherically symmetric in $x$.  
Assume moreover that there exist functions $h \in L^2(\mathbb{R}^n)$ and $g \in L^{\frac{2\alpha}{\alpha -1} }(\mathbb{R}^n)$ such that the following two inequalities hold:
\begin{equation}\label{primera desigualdad caso radial}
|U(x, y) + y| \leq C(|h(x)| + |y|^\alpha),
\end{equation}
\begin{equation}\label{segunda desigualdad caso radial}
    \left|\frac{\partial}{\partial y} (U(x, y) + y)\right| \leq C(|g(x)| + |y|^{\alpha - 1})
\end{equation}
for some constant $C > 0$. Then there exist $0 < \varepsilon < 1$ and $0 < \rho_{\varepsilon} < 1$ such that, whenever $\|h\|_{L^2(\mathbb{R}^n)} < \rho_{\varepsilon}$, there exists a radial solution $u \in \mathcal{H}^\beta_r(f)$ to the equation
\begin{equation}
f(\Delta)u - U(\cdot, u) = 0
\end{equation}
with $\|u\|_{L^{2\alpha}_r(\mathbb{R}^n)} \leq \varepsilon$.
\end{theorem}

\begin{corollary}
    Suppose $U$ is spherically symmetric in $x$, the hypotheses of Theorem \ref{funcion real-analitica} are satisfied, and the growth assumptions of Theorem \ref{Caso Radial} hold. Then there exists a $C^\omega$ solution to the problem \ref{ecuacion basica no-lineal}.
\end{corollary}

\newpage

\section{Nonlocal Equations on Bounded Domains in \texorpdfstring{$\mathbb{R}^{n}$}{Rⁿ}}\label{Nonlocal equations on Bounded Domains 1}
Let $(\mathbb{R}^n, \overline{g})$ be the Riemannian manifold equipped with the standard Euclidean metric; this is a complete, noncompact Riemannian manifold. The Laplace-Beltrami operator for a smooth function $f \in C^\infty(\mathbb{R}^n)$ on $(\mathbb{R}^n, \overline{g})$ defined in (\ref{Operador de Laplace-Beltrami}) satisfies
\begin{equation*}
    \Delta_{\overline{g}} f = \overline{g}^{ij}\nabla_{i}\nabla_{j} f =\overline{g}^{ij}(\partial_i \partial_j f - \Gamma^k_{ij} \partial_k f)
\end{equation*}
where $\overline{g}_{ij} = \delta_{ij}$ are the components of the Euclidean metric tensor in Cartesian coordinates. This immediately implies that $\overline{g}^{ij} = \delta^{ij}$, $|\overline{g}| = 1$, and that the Christoffel symbols $\Gamma^k_{ij}$ vanish because the metric is constant in cartesian coordinates. Therefore,
\begin{equation*}
\Delta_{\overline{g}} f = \delta^{ij} \partial_i \partial_j f = \Delta f.
\end{equation*}
When seeking solutions on a subset $\Omega\subset\mathbb{R}^n$ with properties specified below, viewed as the submanifold $(\Omega,\overline{g}|_{\Omega})$, we focus on equations of the form $f(\Delta)u - U(x,u) = 0$. 

Here, we our approach is to use zero-extensions, rather than considering, for instance the internal spectral realization $f(\Delta_D)$ associated with the standard Dirichlet Laplacian.
The set $\Omega \subset \mathbb{R}^{n}$ is assumed to be an open, bounded domain with Lipschitz boundary.
We recall that, for a bounded domain $\Omega$, the Rellich Kondrachov theorem says that $t > s \ge 0$ with $t,s$ real numbers, the embedding 
\begin{equation}
  H^{t}(\Omega) \hookrightarrow H^{s}(\Omega)  
\end{equation}
is compact \cite[t.1.22]{Rosenberg}.
Let us establish the correct identification of the nonlocal operator $f(\Delta)$ on the bounded domain $\Omega$ with Lipschitz boundary. We preserve the definition of the symbol space $\mathcal{G}^{\beta}$ as functions $f: \mb{R} \to \mb{R}$ satisfying conditions $(P)$ and $(E_\beta)$.
However, for reasons explained below, we will use in bounded domains a subfamily of symbols $\mc{G}^{\beta,\gamma} \subset \mc{G}^{\beta}$, which satisfy the following additional conditions:
($C_{\gamma}$) There exist real constants $C>0$ and $\gamma\geq \beta$ such that
$$M(1+|\xi|^{2})^{\frac{\beta}{2}} \leq f(-\vert{}\xi\vert{}^2) \le C(1+\vert{}\xi\vert{}^2)^{\frac{\gamma}{2}}, \quad \forall \xi \in \mb{R}^{n}, |\xi|> R.$$
(I) The function $\Phi: \mb{R}^{n} \to \mb{R}$ defined by $\Phi(\xi)=f(-|\xi|^{2})$ is locally integrable on $\mb{R}^{n}$.
\justify
This last condition ensures that, for test functions $\psi \in C^{\infty}_{0}(\mb{R}^{n})$, we have  
$$\int_{\mb{R}^{n}}(1+f(-|\xi|^{2}))|\mc{F}(\psi)(\xi)|^{2}d\xi < \infty.$$
We then define the subfamily of symbols $\mc{G}^{\beta,\gamma}$ as 
\begin{equation}
\mc{G}^{\beta,\gamma}=\{f \in \mc{G}^{\beta} \mid (C_{\gamma}), (I)~\text{holds}\}.
\end{equation}
In order to rigorously define the nonlocal operator $f(\Delta)$ on $\Omega$, we introduce the zero-extension operator as 
\begin{align*}
\mathcal{E}_\Omega: L^2(\Omega) &\rightarrow L^2(\mathbb{R}^n), \quad 
u \mapsto \mathcal{E}_\Omega u = \widetilde{u}(x).
\end{align*}
The last expression is  $\widetilde{u}(x)=u(x)$ with $x \in \Omega$ and $\widetilde{u}(x)=0$ for $x \in \mathbb{R}^n\setminus\Omega$.
Clearly, $\mathcal{E}_\Omega$ is a linear isometric embedding, so $\|\widetilde{u}\|_{L^2(\mathbb{R}^n)} = \|u\|_{L^2(\Omega)}$. We also define the restriction operator
\begin{equation*}
\mathcal{R}_\Omega:L^2(\mathbb{R}^n)\longrightarrow L^2(\Omega), \quad \mathcal{R}_\Omega v = v|_\Omega.
\end{equation*}

\begin{definition}
Fix $f \in \mc{G}^{\beta,\gamma}$. The space $\mc{H}^{\beta/2}(f)$ is given by all real-valued functions $g$ on  $\mathbb{R}^{n}$ such that $g$ is measurable, the Fourier transform $\mc{F}(g)$ exists, and
\begin{equation*}
    \int_{\mb{R}^{n}}(1+f(-|\xi|^{2}))|\mc{F}(g)(\xi)|^{2}\,d\xi < \infty.
\end{equation*}
 If we endow $\mc{H}^{\beta/2}(f)$ with the inner product
 \begin{equation}
   \langle u, v \rangle_{\mc{H}^{\beta/2}}
=
\int_{\mathbb{R}^n} (1+f(-|\xi|^2)) \mathcal{F}(u)(\xi) \overline{\mathcal{F}(v)(\xi) }\,d\xi,  
 \end{equation}
 then $\mc{H}^{\beta/2}(f)$ is a Hilbert space.
\end{definition}

\justify

We introduce the following space for an $f \in \mc{G}^{\beta,\gamma}$:
\begin{equation}\label{Espacio de energia Hb}
\mathcal{H}_{0}^{\beta/2}(\Omega;f)
=
\overline{\{\widetilde{\eta} \mid \eta \in C^{\infty}_{0}(\Omega)\}}^{\mc{H}^{\beta/2}(f)}.
\end{equation}
Here $\widetilde{\eta}$ is the zero-extension to $\mb{R}^{n}$.
Note that for every $\psi \in C_0^\infty(\Omega)$, one has $\widetilde{\psi} \in C_0^\infty(\mb{R}^n) \subset \mc{S}(\mb{R}^n)$, so conditions $(I)$ and $(C_\gamma)$ guarantee that $\widetilde{\psi} \in \mc{H}^{\beta/2}(f)$ which implies that the former space is well defined. 
This space is a Hilbert space when endowed with the inner product
\begin{equation}\label{producto de espacio de energia}
\langle u, v \rangle_{\mathcal{H}_{0}^{\beta/2}(\Omega)}=\langle \widetilde{u}, \widetilde{v} \rangle_{\mathcal{H}^{\beta/2}}
=
\int_{\mathbb{R}^n} (1+f(-|\xi|^2)) \mathcal{F}(\widetilde{u})(\xi) \overline{\mathcal{F}(\widetilde{v})(\xi) }\,d\xi.
\end{equation}
Using the nonlocal operator $L= f(\Delta) + I$ in $\mb{R}^{n}$ defined on page \hyperref[espacio Hb permite la interpretacion de L]{\textcolor{blue}{21}}, we define the domain 
operator $L_\Omega$ acting on functions supported on $\overline{\Omega}$ as:
\begin{equation}\label{operador en Omega}
L_\Omega u = \mc{R}_\Omega \circ L \circ \mc{E}_\Omega u = \mc{R}_\Omega \mc{F}^{-1} \Bigl((1+f(-|\xi|^2)) \mc{F}(\widetilde{u})(\xi) \Bigr)
\end{equation}
for every $u \in \mc{D}(L_{\Omega})$. 
Following the same procedure, we define $f(\Delta_{\Omega})u$ as 
\begin{equation}\label{interpretacion del simbolo en el caso acotado}
f(\Delta_{\Omega})u=\mc{R}_{\Omega}f(\Delta)\widetilde{u}=(f(\Delta)\widetilde{u})|_{x \in \Omega}.
\end{equation}
Then we have that $L_{\Omega}u=f(\Delta_{\Omega})u + u$ for every $u \in \mc{D}(L_{\Omega})$. This domain is defined as follows 
\begin{equation}\label{maximal domain of LOmega}
\mc{D}(L_\Omega)
=
\left\{
u\in \mc{H}^{\beta/2}_{0}(\Omega;f) \mid \mathcal{R}_\Omega \mathcal{F}^{-1} \Bigl((1+f(-|\xi|^2)) \mathcal{F}(\widetilde{u})(\xi) \Bigr) \in L^2(\Omega)
\right\}.
\end{equation}
%This domain $\mc{D}(L_{\Omega})$ is a Hilbert space under the graph inner product $$\langle u, v \rangle_{\mathcal{D}(L_\Omega)} := \langle u, v \rangle_{L^2(\Omega)} + \langle L_\Omega u, L_\Omega v \rangle_{L^2(\Omega)}, \quad \forall u, v \in \mathcal{D}(L_\Omega),$$ which induces a norm, defined as $\|u\|_{\mathcal{D}(L_\Omega)} = \left( \|u\|_{L^2(\Omega)}^2 + \|L_\Omega u\|_{L^2(\Omega)}^2\right)^{1/2}$, see for instance \cite{Brezis}.

\justify
Furthermore, we can define the higher-regularity space inherited from the reviewed in the previous section
\begin{equation}\label{espacio mas grande de Hb}
\mathcal{H}_{0}^{\beta}(\Omega;f)
=
\left\{
u\in L^2(\Omega) : \int_{\mathbb{R}^n} (1+f(-|\xi|^2))^2 \big|\mathcal{F}(\widetilde{u})(\xi)\big|^2 \,d\xi <\infty
\right\}.
\end{equation}
We have the following continuous chain of embeddings: 
\begin{equation}\label{cadena de embbeddings caso acotado}
\mathcal{H}_{0}^{\beta}(\Omega;f) \hookrightarrow 
  \mc{D}(L_{\Omega}) \hookrightarrow
  \mathcal{H}_{0}^{\beta/2}(\Omega;f) \hookrightarrow L^2(\Omega).
\end{equation}
\justify
Since
$f(-|\xi|^2) \geq 0$ for every $f \in \mathcal{G}^{\beta,\gamma}$, it immediately follows that $1+f(-|\xi|^2) \geq 1$. Therefore, for any $u \in \mathcal{H}_{0}^{\beta/2}(\Omega;f)$ we have $$\|u\|_{L^2(\Omega)}^2 = \|\widetilde{u}\|_{L^2(\mathbb{R}^n)}^2 \leq \|u\|_{\mathcal{H}_{0}^{\beta/2}(\Omega;f)}^2.$$

\justify

Having introduced the spaces of interest on which we will work, we focus on proving that there exists a unique solution for the linear problem, that is, the problem
\begin{equation}\label{ecuacion lineal acotada}
       (L):= \begin{cases}
            f(\Delta_{\Omega})u + u = g & \text{in } \Omega, \\
            u = 0 & \text{in } \mathbb{R}^n \setminus \Omega.
        \end{cases}
    \end{equation}
The solution $u_{g}$ of the linear problem must belong to some space $\mc{D}$ such that $\mc{F}^{-1}(\mc{F}(\widetilde{u_{g}}) + f(-|\xi|^{2})\mc{F}(\widetilde{u_{g}})) \in L^{2}(\mb{R}^{n})$, which directly implies that the nonlocal symbol $f(\Delta)\widetilde{u_{g}} \in L^{2}(\mb{R}^{n})$, so that $f(\Delta_{\Omega})$ is a well-defined linear operator from $\mc{D}$ to $L^{2}(\mb{R}^{n})$.\\
Suppose that $u_{g} \in \mc{D}(L_{\Omega})$ is the solution of the linear problem for function $g \in L^{2}(\Omega)$, that is, $L_{\Omega}u_{g}=g$; then for any element $v \in \mc{H}^{\beta/2}_{0}(\Omega;f)$, we have, formally,
\begin{equation*}
  L_{\Omega}u_{g}=g \iff (L_{\Omega}u_{g})v=gv \implies \int_{\Omega}(L_{\Omega}u_{g})(x)v(x)dx = \int_{\Omega}g(x)\overline{v(x)}dx,
\end{equation*}
and so, using Parseval's identity on $\mb{R}^{n}$,
\begin{align*}
    \int_{\Omega}(L_{\Omega}u_{g})v\,dx &=\int_{\mathbb{R}^{n}}(L\widetilde{u}_{g})\widetilde{v}\,dx
    %=\int_{\mb{R}^{n}}\mc{F}(L\widetilde{u_{g}})(\xi)\mc{F}(\widetilde{v})(\xi)d\xi\\
=\int_{\mb{R}^{n}} \mc{F}^{-1}[(1+f(-|\xi|^2)) \mc{F}(\widetilde{u_{g}})(\xi)](\xi)\overline{\mc{F}^{-1}(\mc{F}(\widetilde{v})(\xi))}\,d\xi\\
&=\int_{\mathbb{R}^{n}}(1+f(-|\xi|^{2}))\mc{F}(\widetilde{u_{g}})(\xi)\overline{\mc{F}(\widetilde{v})(\xi)}\,d\xi=\langle  u_{g} , v  \rangle_{\mc{H}^{\beta/2}_{0}(\Omega)}.
\end{align*}
Therefore, if we prove that there exists a unique element $u_{g} \in \mc{H}^{\beta/2}_{0}(\Omega;f)$ satisfying the equality
\begin{equation*}
    \langle  u_{g} , v  \rangle_{\mc{H}^{\beta/2}_{0}(\Omega)} = \int_{\Omega}g(x)\overline{v(x)}\,dx
\end{equation*}
we would have a weak solution for \eqref{ecuacion lineal acotada}. This motivates us to use the following result.

\begin{theorem}[The first representation theorem]\cite[p.322]{Kato}.\label{Teorema de Kato} 
Let $t[u, v]$ be a densely defined, closed, sectorial sesquilinear form in a Hilbert space $\mathcal{H}$. There exists an $m$-sectorial operator $T$ such that:\\
(i) $D(T) \subset D(t)$ and
$$t[u, v] = \langle Tu, v \rangle_{\mc{H}}$$
for every $u \in D(T)$ and $v \in D(t)$;\\
(ii) $D(T)$ is a core of $t$;\\
(iii)\label{condicion iii del teorema de Kato} if $u \in D(t)$, $w \in \mathcal{H}$ and
$$t[u, v] = \langle w, v\rangle_{\mc{H}}$$
holds for every $v$ belonging to a core of $t$, then $u \in D(T)$ and $Tu = w$. 
The $m$-sectorial operator $T$ is uniquely determined by the condition (i).\end{theorem}

According to Kato's first representation theorem, every densely defined, closed, and sectorial sesquilinear form $t[u,v]$ on $\mathcal{H}$, meaning its numerical range is confined within a sector away from the negative real axis is uniquely represented by an $m$-sectorial operator $T$ with $D(T) \subset D(t)$, such that:
$$t[u,v] = \langle Tu,v \rangle_{\mathcal{H}}, \qquad \forall u \in D(T),\ v \in D(t),$$
where $D(T)$ is a core of $t$. If $t$ is symmetric and coercive, $T$ is self-adjoint and bounded below. In this framework, Condition (ii) guarantees that $D(T)$ provides a sufficiently rich structure to approximate any state in the energy space $\mathcal{H}^{\beta/2}_{0}(\Omega;f)$.

\begin{lemma}\label{sistema lineal-solucion fuerte}
     Let $\Omega\subset\mathbb{R}^n$ be a bounded Lipschitz domain and $f\in\mathcal{G}^{\beta,\gamma}$. Let $L_\Omega$ be the operator defined in \eqref{operador en Omega} with domain $\mc{D}(L_\Omega)$. Then for every $g\in L^2(\Omega)$, we have the following results:\\
i)\label{lema solucion debil acotada}~There exists a unique weak solution $u_g \in \mathcal{H}_{0}^{\beta/2}(\Omega;f)$ to the equation $L_\Omega u = g$, meaning that
\begin{equation}\label{condicion debil caso acotado}
\langle u_g, v \rangle_{\mathcal{H}_{0}^{\beta/2}(\Omega)} = \int_{\Omega} g(x) \overline{v(x)} \,dx , \quad \forall v \in \mathcal{H}_{0}^{\beta/2}(\Omega;f).
\end{equation}
Furthermore, the following estimate holds $\|u_g\|_{\mathcal{H}_{0}^{\beta/2}(\Omega)} \leq \|g\|_{L^2(\Omega)}.$\label{energy-estimate}\\
ii)~There exists a unique function $u_g\in \mc{D}(L_\Omega)$ such that 
\begin{equation}\label{sistema lineal fuerte}
   L_\Omega u_g = g, \quad \text{in } L^2(\Omega) 
\end{equation}
Moreover, the following estimate holds  $\|u_g\|_{L^2(\Omega)} \leq \|g\|_{L^2(\Omega)}$.
\end{lemma}

\begin{proof}
$i)$ Let $F \in (\mc{H}^{\beta/2}_{0}(\Omega;f))^{*}$ 
be the bounded conjugate-linear functional
defined by $$F(v) = \int_{\Omega} g(x) \overline{v(x)} \,dx.$$
By the Cauchy-Schwarz inequality and the continuous embedding \eqref{cadena de embbeddings caso acotado}, we have that
\begin{align*}
|F(v)|&\leq \int_{\Omega}|g(x)\overline{v(x)}|\,dx \leq \norm{g}_{L^{2}(\Omega)}\norm{v}_{L^{2}(\Omega)}=\norm{g}_{L^{2}(\Omega)}\norm{\widetilde{v}}_{L^{2}(\mb{R}^{n})}
%=\norm{g}_{L^{2}(\Omega)}\norm{\mc{F}(\widetilde{v})}_{L^{2}(\mb{R}^{n})}
\leq \norm{g}_{L^{2}(\Omega)}\norm{v}_{\mc{H}^{\beta/2}_{0}(\Omega)}
\end{align*}
This shows that $F$ is a bounded linear functional on the Hilbert space $\mathcal{H}_{0}^{\beta/2}(\Omega;f)$, with its operator norm $\|F\|_{*}$ bounded by $\|g\|_{L^2(\Omega)}$.
By the Riesz Representation Theorem \cite[t.3.1]{HH}, there exists a unique element $u_g \in \mathcal{H}_{0}^{\beta/2}(\Omega;f)$ such that
$$
F(v) = \int_{\Omega} g(x) \overline{v(x)} \,dx=\langle u_g, v \rangle_{\mathcal{H}_{0}^{\beta/2}(\Omega)}, \quad \forall v \in \mathcal{H}_{0}^{\beta/2}(\Omega;f).
$$
This proves existence and uniqueness of a weak solution of the linear problem \eqref{ecuacion lineal acotada}. 
If $u_g = 0$, the energy estimate holds trivially. 
Moreover, if $u_{g}\neq0$, we have the following estimate
$$\|u_g\|_{\mathcal{H}_{0}^{\beta/2}(\Omega)}^2 = \langle u_g, u_g \rangle_{\mathcal{H}_{0}^{\beta/2}(\Omega)} = \int_{\Omega} g(x) \overline{u_g(x)} \,dx \leq \|g\|_{L^2(\Omega)} \|u_g\|_{L^2(\Omega)}\leq\|g\|_{L^2(\Omega)} \|u_g\|_{\mathcal{H}_{0}^{\beta/2}(\Omega)}.$$

$ii)$ Note that $C_{0}^{\infty}(\Omega) \subset \mc{H}^{\beta/2}_{0}(\Omega;f)$ holds for any symbol $f \in \mc{G}^{\beta,\gamma}$. Indeed, let $\psi \in C_{0}^{\infty}(\Omega)$. Its zero-extension $\widetilde{\psi}$ is smooth and compactly supported in $\mb{R}^{n}$, hence belongs to the Schwartz class $\mathcal{S}(\mb{R}^n)$. Consequently, its Fourier transform $\mc{F}(\widetilde{\psi})$ is bounded and decays faster than any polynomial.
On any compact set $\{|\xi| \leq R\}$, condition $(I)$ ensures that $1+f(-|\xi|^2)$ is locally integrable, so its product with the bounded function $|\mc{F}(\widetilde{\psi})|^2$ is integrable. Outside a sufficiently large ball $\{|\xi| > R\}$, condition $(C_\gamma)$ gives $1+f(-|\xi|^2) \le \widetilde{C}(1+|\xi|^2)^{\gamma/2}$, so the symbol grows at most polynomially of order $\mc{O}(|\xi|^{\gamma})$ as $|\xi|\to\infty$. Therefore,
\begin{align*}
\int_{\mb{R}^{n}}\bigl(1+f(-|\xi|^{2})\bigr)\bigl|\mc{F}(\widetilde{\psi})(\xi)\bigr|^{2}\,d\xi &\leq \|\mc{F}(\widetilde{\psi})\|_{L^\infty(\mb{R}^n)}^2 \int_{|\xi| \leq R} \bigl(1+f(-|\xi|^2)\bigr)\,d\xi \\
&+ \widetilde{C} \int_{|\xi| > R} (1+|\xi|^2)^{\gamma/2}\bigl|\mc{F}(\widetilde{\psi})(\xi)\bigr|^{2}\,d\xi < \infty,
\end{align*}
which confirms that $\psi \in \mc{H}^{\beta/2}_{0}(\Omega;f)$. Since $C^{\infty}_{0}(\Omega)$ is dense in $L^{2}(\Omega)$, the chain of inclusions $C^{\infty}_{0}(\Omega) \subset \mc{H}^{\beta/2}_{0}(\Omega;f) \subset L^{2}(\Omega)$ implies that $(\mc{H}^{\beta/2}_{0}(\Omega;f), \|\cdot\|_{L^{2}})$ is also dense in $L^{2}(\Omega)$ and by definition, $C^{\infty}_{0}(\Omega)$ is dense in $\mc{H}^{\beta/2}_{0}(\Omega;f)$ with respect to the energy norm $\norm{v}_{\mc{H}^{\beta/2}_{0}(\Omega)}$.
\vspace{0.12cm}

We prove now that the sesquilinear, symmetric, and coercive form {\fontsize{10.3}{14}$t[u,v]=\langle u , v \rangle_{\mc{H}^{\beta/2}_{0}(\Omega)}$} with domain {\fontsize{10.3}{14}$D(t) = \mc{H}^{\beta/2}_{0}(\Omega;f)$} is closed in $\mc{H} = L^{2}(\Omega)$.
\vspace{0.1cm}

Let $(u_n)_{n \in \mb{N}} \subset D(t)$ be a sequence such that $u_n \to u$ in $L^2(\Omega)$ as $n \to \infty$ and $t[u_n - u_m, u_n - u_m] \to 0$ as $n,m \to \infty$. 
Let us prove that $t[u_n - u, u_n - u] \to 0$.
Define the sequence $\psi_n(\xi) := \bigl(1+f(-|\xi|^2)\bigr)^{1/2}\mc{F}(\widetilde{u}_n)(\xi)$. Since
$$\|\psi_n - \psi_m\|_{L^2(\mb{R}^n)}^2 = \int_{\mb{R}^n} \bigl(1+f(-|\xi|^2)\bigr)\bigl|\mc{F}(\widetilde{u}_n)(\xi) - \mc{F}(\widetilde{u}_m)(\xi)\bigr|^2\,d\xi = t[u_n - u_m, u_n - u_m] \longrightarrow 0,$$
$(\psi_{n})_{n \in \mb{N}}$ is a Cauchy sequence in $L^2(\mb{R}^n)$, so by completeness of $L^2(\mb{R}^n)$ there exists $h \in L^2(\mb{R}^n)$ such that $\psi_n \to h$ in $L^2(\mb{R}^n)$. On the other hand, since $u_n \to u$ in $L^2(\Omega)$, the zero-extension yields $\widetilde{u}_n \to \widetilde{u}$ in $L^2(\mb{R}^n)$, and by Plancherel's theorem $\mc{F}(\widetilde{u}_n) \to \mc{F}(\widetilde{u})$ in $L^2(\mb{R}^n)$. Using that $\bigl(1+f(-|\xi|^2)\bigr)^{-1/2} \leq 1$, we have
$$\mc{F}(\widetilde{u}_n)(\xi) = \bigl(1+f(-|\xi|^2)\bigr)^{-1/2}\psi_n(\xi) \longrightarrow \bigl(1+f(-|\xi|^2)\bigr)^{-1/2}h(\xi) \quad \text{in } L^2(\mb{R}^n).$$
By uniqueness of the limit in $L^2(\mb{R}^n)$, it follows that $\bigl(1+f(-|\xi|^2)\bigr)^{1/2}\mc{F}(\widetilde{u})(\xi) = h(\xi) \in L^2(\mb{R}^n)$.
In fact $t[u_n - u, u_n - u] = \|\psi_n - h\|_{L^2(\mb{R}^n)}^2 \to 0$ as $n \to \infty$ and $u \in \mc{H}^{\beta/2}_{0}(\Omega;f)$ because $\mc{H}_{0}^{\beta/2}(\Omega;f)$ is by definition closed on $\mc{H}^{\beta/2}(\Omega)$; thus, $t[u,v]$ is a closed form in $L^2(\Omega)$.

Therefore the form $t[u,v]$ satisfies the conditions of Kato's theorem \cite[t.2.1]{Kato}; then, there exists a unique operator $A$ which is densely defined,
self-adjoint, closed and coercive 
$$A: D(A) \subset L^{2}(\Omega) \to L^{2}(\Omega)$$
such that $D(A) \subset \mc{H}^{\beta/2}_{0}(\Omega;f)$ and satisfies the following 
\begin{equation*}\label{construccion del operador A , caso acotado}
    \langle u, v \rangle_{\mathcal{H}_{0}^{\beta/2}(\Omega)} = \langle A u, v \rangle_{L^{2}(\Omega)}, \quad \forall u\in D(A),\ \forall v\in \mathcal{H}_{0}^{\beta/2}(\Omega;f).
\end{equation*}
By item $(iii)$ of Kato's Theorem \ref{Teorema de Kato}, we have that $u \in D(A)$ if and only if $t[u,v]$ can be represented as the inner product in $\mc{H}=L^{2}(\Omega)$ of some function $w \in L^{2}(\Omega)$ with $v \in D(t)$, which in this case is $D(t)=\mc{H}^{\beta/2}_{0}(\Omega;f)$. 
In fact, $D(A) \subset \mc{H}^{\beta/2}_{0}(\Omega;f)$ is dense and has the form 
\begin{equation*}
    D(A) = \left\{ u \in \mc{H}^{\beta/2}_{0}(\Omega;f) \mid \exists w \in L^2(\Omega), \; t[u,v] = \langle w, v \rangle_{L^2(\Omega)}, \forall v \in \mc{H}^{\beta/2}_{0}(\Omega;f) \right\}.
\end{equation*}
We know that $t[u,v]=\langle u , v  \rangle_{\mc{H}^{\beta/2}_{0}(\Omega)}$. Indeed,
$$   
\int_{\mathbb{R}^n} (1+f(-|\xi|^2)) \mathcal{F}(\widetilde{u})(\xi) \overline{\mathcal{F}(\widetilde{v})(\xi) }\,d\xi = \langle w, v  \rangle_{L^{2}(\Omega)}.
$$
We prove that the condition that $u \in D(A)$ is equivalent to $\mathcal{R}_\Omega \mathcal{F}^{-1}\bigl((1+f(-|\xi|^2)) \mathcal{F}(\widetilde{u})(\xi)\bigr) \in L^{2}(\Omega)$.
We know that for $ u \in \mc{H}^{\beta/2}_{0}(\Omega;f)$, one has $(1+f(-|\xi|^{2}))^{1/2}\mc{F}(\widetilde{u}) \in L^{2}(\mb{R}^{n})$, but this does not imply $(1+f(-|\xi|^{2}))\mc{F}(\widetilde{u}) \in L^{2}(\mb{R}^{n})$,
which implies that $ \mathcal{F}^{-1}((1+f(-|\xi|^2)) \mathcal{F}(\widetilde{u})(\xi) )$ could fail to be a function in $L^{2}(\mb{R}^{n})$, but rather a distribution.
In fact, denoting $$T_{u}=\mathcal{F}^{-1}((1+f(-|\xi|^2)) \mathcal{F}(\widetilde{u})(\xi)),$$
let us consider a test function $\varphi \in C^{\infty}_{0}(\mb{R}^{n})$. Note that $T_{u}$ is linear on $C^{\infty}_{0}(\mb{R}^{n})$; to prove that $T_{u} \in \mc{D}^{\prime}(\mb{R}^{n})$ it suffices to prove that $T_{u}$ is continuous. 
We have that 
the integral
$$\int_{\mathbb{R}^n} (1+f(-|\xi|^2)) \mathcal{F}(\widetilde{u})(\xi) \overline{\mathcal{F}(\varphi)(\xi) }\,d\xi $$
is absolutely convergent, because $(1+f(-|\xi|^{2}))^{1/2}\mc{F}(\widetilde{u}) \in L^{2}(\mb{R}^{n})$ and $\varphi \in \mc{S}(\mb{R}^{n})$, where $\mc{S}(\mb{R}^{n})$ is the Schwartz space, one has $\mc{F}(\varphi) \in \mc{S}(\mb{R}^{n})$, and using that $f \in \mc{G}^{\beta,\gamma}$, we have that $(1+f(-|\xi|^{2}))^{1/2}\mc{F}(\varphi) \in L^{2}(\mb{R}^{n})$. Then
\begin{align*}
\Big|\int_{\mathbb{R}^n} (1+f(-|\xi|^2)) \mathcal{F}(\widetilde{u}) \overline{\mathcal{F}(\varphi) }\,d\xi\Big| &\leq \int_{\mb{R}^n} |(1+f(-|\xi|^{2}))^{1/2}\mc{F}(\widetilde{u})|\cdot |(1+f(-|\xi|^{2}))^{1/2}\overline{\mc{F}(\varphi)}| \, d\xi \\
&\leq \norm{(1+f(-|\xi|^{2}))^{1/2}\mc{F}(\widetilde{u})  }_{L^{2}}\norm{ (1+f(-|\xi|^{2}))^{1/2}\mc{F}(\varphi)}_{L^{2}} < \infty.
\end{align*}
Therefore the integral is absolutely convergent. Moreover the duality pairing $\langle  T_{u} , \varphi \rangle_{\mc{D}', \mc{D}}$ is well defined and $T_{u} \in \mc{D}^{\prime}(\mb{R}^{n})$. We define the restriction of $T_{u}$ over $\Omega$ as 
\begin{equation}
    \langle \mc{R}_{\Omega}T_{u} , \eta   \rangle_{\mc{D}'(\Omega), \mc{D}(\Omega)} = \langle  T_{u} , \widetilde{\eta}  \rangle_{\mc{D}'(\mb{R}^{n}), \mc{D}(\mb{R}^{n})}, \quad \eta \in C_{0}^{\infty}(\Omega).
\end{equation}
Moreover, using that $\mc{F}(\varphi)(-\xi)=\overline{\mc{F}(\varphi)(\xi)}$ for a real function $\varphi$ (or pairing with $\overline{\varphi}$ in the complex case), we have the following
\begin{align*}
    \langle  T_{u} , \overline{\varphi} \rangle_{\mc{D}', \mc{D}} =\langle  \mathcal{F}^{-1}\bigl((1+f(-|\xi|^2)) \mathcal{F}(\widetilde{u})\bigr) , \overline{\varphi} \rangle &= \int_{\mb{R}^{n}}\left(\int_{\mb{R}^{n}}(1+ f(-|\xi|^{2}))\mc{F}(\widetilde{u})(\xi)e^{2 \pi i \langle x , \xi  \rangle} \, d\xi\right) \overline{\varphi(x)}\, dx \\
    &= \int_{\mb{R}^{n}}(1+f(-|\xi|^{2}))\mc{F}(\widetilde{u})(\xi)\left(\int_{\mb{R}^{n}} \overline{\varphi(x)}e^{2 \pi i \langle   x,  \xi \rangle} \, dx\right) d\xi \\
    &= \int_{\mathbb{R}^n} (1+f(-|\xi|^2)) \mathcal{F}(\widetilde{u})(\xi) \overline{\mathcal{F}(\varphi)(\xi) }\,d\xi.
\end{align*}
\justify
Therefore, for every $v \in C^{\infty}_{0}(\Omega)$ we have
$$t[u,v]=\langle T_{u}, \overline{\widetilde{v}} \rangle_{\mc{D}'(\mb{R}^n), \mc{D}(\mb{R}^n)}= \langle \mc{R}_\Omega T_u, \overline{v} \rangle_{\mc{D}'(\Omega), \mc{D}(\Omega)}.$$
Consequently, since $C^{\infty}_{0}(\Omega)$ is dense by definition in $\mc{H}^{\beta/2}_{0}(\Omega;f)$ and in $L^{2}(\Omega)$, by condition $iii)$ of Kato's Theorem \ref{condicion iii del teorema de Kato} we have
$$\langle T_{u}, \overline{\widetilde{v}} \rangle_{\mc{D}'(\mb{R}^n), \mc{D}(\mb{R}^n)}=\int_{\Omega}w(x)\overline{v(x)}\,dx, \quad \forall v \in C_0^\infty(\Omega).
$$
If $u \in D(A)$, then for every $v \in C_0^\infty(\Omega)$,
$$\langle \mc{R}_\Omega T_u, \overline{v} \rangle_{\mc{D}'(\Omega), \mc{D}(\Omega)} = t[u,v] = \langle Au, v \rangle_{L^2(\Omega)},$$
so $\mc{R}_\Omega T_u = Au \in L^2(\Omega)$ in $\mc{D}'(\Omega)$, which implies $u \in \mc{D}(L_\Omega)$ and $L_\Omega u = Au$. Conversely, suppose that $u \in \mc{D}(L_\Omega)$, so that $\mc{R}_\Omega T_u = w \in L^2(\Omega)$. Then for every $v \in C^{\infty}_{0}(\Omega)$,
$$t[u,v] = \langle w, v \rangle_{L^2(\Omega)}.$$
Since both sides are continuous with respect to $v$ in the energy norm $\|\cdot\|_{\mc{H}_0^{\beta/2}(\Omega)}$, namely
$$|t[u,v]| \leq \|u\|_{\mc{H}_0^{\beta/2}(\Omega)}\|v\|_{\mc{H}_0^{\beta/2}(\Omega)} \quad \text{and} \quad |\langle w, v \rangle_{L^2(\Omega)}| \leq \|w\|_{L^2(\Omega)}\|v\|_{L^2(\Omega)} \leq \|w\|_{L^2(\Omega)}\|v\|_{\mc{H}_0^{\beta/2}(\Omega)},$$
and $C_0^\infty(\Omega)$ is dense in $\mc{H}_0^{\beta/2}(\Omega;f)$, the identity $t[u,v] = \langle w, v \rangle_{L^2(\Omega)}$ extends by density to all $v \in \mc{H}_0^{\beta/2}(\Omega;f)$. By item $(iii)$ of Theorem \ref{Teorema de Kato}, we conclude that $u \in D(A)$ and $Au = w = L_\Omega u$.\\
Therefore,
$$
D(A) = \Bigl\{ u \in \mc{H}^{\beta/2}_0(\Omega;f) \Bigm| \mc{R}_\Omega \mc{F}^{-1}\Bigl(\bigl(1+f(-|\xi|^2)\bigr)\mc{F}(\widetilde{u})\Bigr) \in L^2(\Omega) \Bigr\} = \mc{D}(L_\Omega).
$$
Therefore, $A = L_\Omega$ and $D(A) = \mc{D}(L_\Omega)$.
In particular, since $A = L_\Omega$ is a closed operator in $L^2(\Omega)$, its domain $\mc{D}(L_\Omega)$ endowed with the graph norm $\|u\|_{\mc{D}(L_\Omega)} = \bigl(\|u\|_{L^2(\Omega)}^2 + \|L_\Omega u\|_{L^2(\Omega)}^2\bigr)^{1/2}$ is a Hilbert space.
Note that $A$ is self-adjoint and coercive, which implies that $A 
\geq I$, because $\langle Av,v \rangle\geq \norm{v}^{2}$, then $0 \notin \sigma(A)$.
Hence $A$ is invertible. 
%In fact, since that $A$ is coercive, we have that $\text{ker}(A)=\text{ran}(A)^{\perp}$, but $Av=0$ if and only if $v\equiv 0$, this implies that $A$ is inyective; indeed $\text{ran}(A)=L^{2}(\Omega)$, then $A$ is Biyective. 
Moreover, the operator $A$ satisfies $\norm{A\omega}_{L^{2}(\Omega)} \geq \norm{\omega}_{L^{2}(\Omega)}$ for any $\omega \in \mc{D}(L_{\Omega})$ because, using embedding \eqref{cadena de embbeddings caso acotado} we have 
\begin{equation*}
    \norm{\omega}^{2}_{L^{2}(\Omega)} \leq \norm{\omega}^{2}_{\mc{H}^{\beta/2}_{0}(\Omega)}=\langle  \omega , \omega  \rangle_{\mc{H}^{\beta/2}_{0}(\Omega)}=\langle A \omega, \omega \rangle_{L^{2}(\Omega)}\leq \norm{A\omega}_{L^{2}(\Omega)}\norm{\omega}_{L^{2}(\Omega)}.
\end{equation*}
The inverse operator $A^{-1}$ is a linear and bounded operator. 
In fact, for every $g \in L^{2}(\Omega)$ we have that there exists only one $\omega \in \mc{D}(L_{\Omega})$ such that $A\omega=g$ (i.e., $\omega = A^{-1}g$), then we have
\begin{equation*}
\norm{A^{-1}g}_{L^{2}(\Omega)}=\norm{\omega}_{L^{2}(\Omega)} \leq \norm{A\omega}_{L^{2}(\Omega)}=\norm{g}_{L^{2}(\Omega)}. 
\end{equation*}
Then $\|A^{-1}\|_{\mathcal{L}(L^{2}(
\Omega
))} \leq 1$.
 Using the existence of a weak solution by Lemma $(i)$ \eqref{lema solucion debil acotada}, we have by condition $(iii)$ of Kato's Theorem \ref{Teorema de Kato} that $u_g \in D(A) = \mc{D}(L_\Omega)$ and
\begin{equation*}
    \langle u_g, v \rangle_{\mathcal{H}_{0}^{\beta/2}(\Omega)} = \int_{\Omega} g \overline{v} \,dx=\int_\Omega \left( \mathcal{R}_\Omega \mathcal{F}^{-1} \Bigl( (1+f(-|\xi|^2)) \mathcal{F}(\widetilde{u_{g}}) \Bigr) \right)\!(x) \overline{v(x)} \,dx=\int_{\Omega}Au_{g}\overline{v}\,dx.
\end{equation*}
Then there exists a unique $u_{g}=A^{-1}g \in \mc{D}(L_{\Omega})$. 
Thus, $u_g$ is a strong solution satisfying $L_\Omega u_g = g$ directly in $L^2(\Omega)$. Finally,
\begin{equation}\label{desigualdad 2 lema caso acotado}
\|u_g\|_{L^2(\Omega)}=
\|A^{-1}g\|_{L^{2}( 
\Omega)
} \leq \|g\|_{L^2(\Omega)}
\end{equation}
\end{proof}

\begin{theorem}\label{teorema en Rn, para dominios acotados}
Let $\Omega\subset\mathbb{R}^n$ be an open, bounded domain with Lipschitz boundary. Let $f\in\mathcal{G}^{\beta,\gamma}$ with $\gamma\geq\beta>n$, 
$n>1$. Let $U:\Omega\times\mathbb{R}\to\mathbb{R}$ be a Carathéodory non-linearity which is a $C^{1}$ in the second variable, and assume there exist constants $\alpha>1$, $\mu\ge 1$, $C>0$, and a function $h\in L^{2}(\Omega)$ such that
\begin{equation}\label{desigualdad de caso acotado}
|U(x,y)+y| \le C(|h(x)|+|y|^{\alpha}), \qquad \left|\frac{\partial}{\partial y}\bigl(U(x,y)+y\bigr)\right| \le C(1+|y|^{\mu}).
\end{equation}
Then there exist $0 < \tau < 1$ and $0 < \rho_{\tau} < 1$ such that whenever $\norm{h}_{L^{2}(\Omega)} < \rho_{\tau}$, there exists a solution $u \in \mc{D}(L_{\Omega})$ of the non-linear equation $f(\Delta_{\Omega})u-U(\cdot,u)=0$.
\end{theorem}

\begin{proof}
    Using the operator $L_{\Omega}=f(\Delta_{\Omega})+I$ defined in \eqref{operador en Omega}, solving the nonlinear equation $f(\Delta_{\Omega})u-U(x,u)=0$ is equivalent to solving the linear nonhomogeneous problem
\begin{equation}
    L_{\Omega}u=\eta_{u}
\end{equation}
where $\eta_{u}(x)=U(x,u(x)) + u(x)$ for $u \in H^{s}(\Omega)$, $x \in \Omega$. Fix {\fontsize{10.7}{14}$s \in (\frac{n}{2},\frac{\beta}{2})$} and set $s_{\alpha}=\frac{n}{2}\left( \frac{\alpha -1}{\alpha}\right)$. 
Using the Sobolev Theorem \cite[p.26]{Taylor III}, which states that for $2\sigma<n$, $p \in (1,\infty)$ the embedding 
$H^{\sigma}(\mathbb{R}^{n}) \hookrightarrow L^{2n/(n-2\sigma)}(\mathbb{R}^{n})$
is continuous, and for $2\sigma>n$ one has that $H^{\sigma}(\Rn) \hookrightarrow L^{\infty}(\Rn)$ is also continuous, and the fact that for $\Omega$ an  open, bounded domain with Lipschitz boundary, there exists an extension operator $\mathcal{Z}: H^{\sigma}(\Omega) \to H^{\sigma}(\mathbb{R}^{n})$ for every $\sigma \in \mb{R}_{\geq 0}$, which is linear and continuous, and $\mathcal{Z}u|_{\Omega}=u$ for every $u \in H^{\sigma}(\Omega)$. Since $2s>n$, we have
\begin{equation*}
\norm{u}_{L^{\infty}(\Omega)} = \norm{\mathcal{Z}u}_{L^{\infty}(\Omega)} \leq \norm{\mathcal{Z}u}_{L^{\infty}(\Rn)} \leq \norm{\mathcal{Z}u}_{H^{s}(\Rn)} \leq C_{s,p,\Omega}\norm{u}_{H^{s}(\Omega)}.
\end{equation*}
where $C_{s,p,\Omega}>0$ is a constant depending on $s, p$ and $\Omega$.
Then for our case, such that $s>s_{\alpha}$, the following chain of embeddings holds,
\begin{equation}\label{embedding caso dominios acotados}
    H^{s}(\Omega) \hookrightarrow L^{\infty}(\Omega), \quad 
    H^{s}(\Omega) \hookrightarrow H^{s_{\alpha}}(\Omega) \hookrightarrow 
    %L^{2n/(n-2s_{\alpha})}(\Omega)=
    L^{2\alpha}(\Omega).
\end{equation}
Let $u \in H^{s}(\Omega)$. Using the embeddings \eqref{embedding caso dominios acotados}, we want to show that  $\eta_{u} \in L^{2}(\Omega)$. In fact, using that the Lebesgue measure on $\mb{R}^{n}$, denoted by $\lambda_{n}$ satisfies $\lambda_{n}(\Omega)<\infty$, we have that $L^{p}(\Omega)$ spaces are embedded as $L^{p}(\Omega) \subseteq L^{q}(\Omega)$ for any real numbers $1 \leq q<p \leq \infty$. In fact
\begin{align*}
\norm{\eta_{u}}_{L^{2}(\Omega)}&=\left(\int_{\Omega}|U(x,u(x)) + u(x)|^{2}dx   \right)^{\frac{1}{2}} \leq C\left( \int_{\Omega}(|h(x)|^{2} + 2|h(x)||u(x)|^{\alpha} + |u(x)|^{2\alpha})dx  \right)^{\frac{1}{2}}\\
    &\leq \widetilde{C}\left( \int_{\Omega}|h(x)|^{2}+|u(x)|^{2\alpha} dx \right)^{\frac{1}{2}} \leq \widetilde{C}\left( \norm{h}_{L^{2}(\Omega)} + \norm{|u|^{\alpha}}_{L^{2}(\Omega)}  \right) \\ 
    &=\widetilde{C}\left(\norm{h}_{L^{2}(\Omega)} + \norm{u}^{\alpha}_{L^{2\alpha}(\Omega)} \right) \leq \widetilde{C}\left(\norm{h}_{L^{2}(\Omega)} + C^{\alpha}_{s}\norm{u}^{\alpha}_{H^{s}(\Omega)}  \right)
\end{align*}
for constants $\widetilde{C}, C^{\alpha}_{s}>0$. This inequality follows by Young's inequality for $|h|,|u|^{\alpha}\geq 0$, $\||u|^{\alpha}\|_{L^{2}(\Omega)}=\|u\|^{\alpha}_{L^{2\alpha}(\Omega)}$ and $\|u\|^{\alpha}_{L^{2\alpha}(\Omega)} \leq C^{\alpha}_{s}\|u\|^{\alpha}_{H^{s}(\Omega)}$ by embeddings \eqref{embedding caso dominios acotados}. 
Hence
\begin{equation}  \norm{\eta_{u}}_{L^{2}(\Omega)} \leq \widetilde{C}\left(\norm{h}_{L^{2}(\Omega)} + C^{\alpha}_{s}\norm{u}^{\alpha}_{H^{s}(\Omega)} \right) < \infty
\end{equation}
which implies $\eta_{u} \in L^{2}(\Omega)$.
\begin{claim}\label{claim caso acotado}
    For a non-negative $s \in \mathbb{R}$, such that {\fontsize{10.7}{14}$s \in (\frac{n}{2},\frac{\beta}{2})$} with $\beta>n$, the embedding chain
\begin{equation}\label{cadena de embeddings caso acotado 1}
        \mc{D}(L_{\Omega}) \hookrightarrow H^{s}(\Omega) \hookrightarrow L^{q}(\Omega), \quad q \in [1,2\alpha)
    \end{equation}
is continuous. 
\end{claim}

\begin{proofofclaim}
Let us first prove that the embedding $H^{s}(\Omega) \hookrightarrow L^{q}(\Omega)$ for $q \in [1,2\alpha)$ is satisfied; moreover, it is compact. By Theorem \cite[t.1.22]{Rosenberg}, and the fact that $s_{\alpha}<s$, we have that $H^{s}(
\Omega) \hookrightarrow H^{s_{\alpha}}(\Omega)$ is compact. By the Rellich-Kondrachov theorem for bounded domains, we have that $H^{s_{\alpha}}(\Omega) \hookrightarrow L^{q}(\Omega)$ is compact for $q \in [1,2\alpha)$. Therefore, we have that the chain
\begin{equation}\label{cadena de embeddings caso acotado 3}
    H^{s}(\Omega) \hookrightarrow H^{s_{\alpha}}(\Omega) \hookrightarrow L^{q}(\Omega), \quad q \in [1,2\alpha)
\end{equation}
is a compact embedding. Now we prove that the following 
\begin{equation}\label{cadena de embeddings caso acotado 2}
    \mc{D}(L_{\Omega}) \hookrightarrow \mc{H}_{0}^{\beta/2}(\Omega;f) \hookrightarrow \mc{H}^{\beta/2}(f) \hookrightarrow H^{s}(\Omega)
\end{equation}
is a continuous chain of embeddings. 
First of all, we prove that {\fontsize{11}{14}$\mc{D}(L_{\Omega}) \hookrightarrow \mc{H}^{\beta/2}_{0}(\Omega;f)$} is continuous. Note that by construction of Lemma \ref{sistema lineal-solucion fuerte}, there exists a unique densely defined, self-adjoint and coercive operator
$A:D(A) \subset L^{2}(\Omega) \to L^{2}(\Omega)$ which satisfies $D(A) \subset \mc{H}^{\beta/2}_{0}(\Omega;f)$. Moreover  $D(A)=\mc{D}(L_{\Omega})$, then $\mc{D}(L_{\Omega}) \subset \mc{H}^{\beta/2}_{0}(\Omega;f)$.
Using the Lemma \ref{sistema lineal-solucion fuerte} and the Young inequality, we have the following computations
\begin{align*}
\langle u, u \rangle_{\mathcal{H}_{0}^{\beta/2}(\Omega)} &= \langle L_{\Omega} u, u \rangle_{L^{2}(\Omega)} \leq \|L_{\Omega}u\|_{L^{2}(\Omega)}\|u\|_{L^{2}(\Omega)} \leq \tfrac{1}{2} \left( \Vert{}L_{\Omega}u\Vert{}_{L^2(\Omega)}^2 + \Vert{}u\Vert{}_{L^2(\Omega)}^2 \right) \leq  \tfrac{1}{2} \Vert{}u\Vert{}_{\mathcal{D}(L_{\Omega})}^2.
\end{align*}
Therefore, we have the existence of some constant $K>0$ such that $\|u\|_{\mathcal{H}^{\beta/2}_{0}(\Omega)} \leq K\|u \|_{\mc{D}(L_{\Omega})}$.\\ 
%\begin{align*}\int_{\mb{R}^{n} \setminus \Omega}\left| \mathcal{F}^{-1} \Bigl((1+f(-|\xi|^2)) \mathcal{F}(\widetilde{u})(\xi) \Bigr)(x)\right|^{2}dx &= \int_{\mathbb{R}^{n} \setminus \Omega} \left\vert{} \int_{\mathbb{R}^{n}} (1+f(-\vert{}\xi\vert{}^2)) \left( \mc{F}(\widetilde{u})(\xi) \right) e^{2\pi i x \cdot \xi} \, d\xi \right\vert{}^{2} dx.\end{align*}
Using the form of $\mc{H}^{\beta/2}(f)$, we have that {\fontsize{11}{14}$\mc{H}^{\beta/2}_{0}(\Omega;f) \hookrightarrow \mc{H}^{\beta/2}(f)$} is continuous. 
Finally we prove that $\mc{H}^{\beta/2}(f)$ is continuously embedded in $H^{s}(\Omega)$. For this, we prove that $\mc{H}^{\beta/2}(f)$ is continuously embedded in $H^{s}(\mb{R}^{n})$. Let $u \in \mc{H}^{\beta/2}(f)$; using the norm of Sobolev spaces in $\mb{R}^{n}$, 
$$\norm{u}_{H^s(\mb{R}^n)}^2=
\int_{\mb{R}^n} (1+\abs{\xi}^2)^s \abs{\mc{F} (u)(\xi)}^2\,d\xi.
$$
For $\xi \in \mb{R}^{n}, |\xi|>1$, we have that $(1+\abs{\xi}^2)^s \le (1+\abs{\xi}^2)^{\beta/2} \le c_{\beta} \abs{\xi}^{\beta}$ for some constant $c_{\beta}>0$. Furthermore, since $f \in \mc{G}^{\beta,\gamma}$, there exist $r, m_{\beta}>0$ such that
$|\xi|^{\beta} \leq m_{\beta}\bigl(1+f(-\abs{\xi}^2)\bigr), \forall |\xi|>r$, therefore $\forall \xi \in \mb{R}^{n}$ with $|\xi|>\widehat{r}=\max\{1,r\}$, we have that $(1+\abs{\xi}^2)^s \le \widetilde{m}_{\beta}\bigl(1+f(-\abs{\xi}^2)\bigr),$ for some constant $\widetilde{m}_{\beta}>0$. Let us define $m_{s}=\sup_{|\xi|\leq \widehat{r}}(1+|\xi|^{2})^{s}$. Since $1+f(-|\xi|^2) \geq 1$, we obtain
\begin{align*}
  \int_{\mb{R}^n} (1+|\xi|^2)^s |\mc{F}(u)(\xi)|^2\,d\xi &\leq m_{s}\int_{\{|\xi| \leq \widehat{r}\}}\bigl(1+f(-\abs{\xi}^2)\bigr)|\mc{F}(u)(\xi)|^{2}  \,d\xi   \\
  &+ \widetilde{m}_{\beta}\int_{\{|\xi|>\widehat{r}\}}\bigl(1+f(-\abs{\xi}^2)\bigr)|\mc{F}(u)(\xi)|^{2}  \,d\xi.
\end{align*}
Therefore for a constant $C_{s,\beta}=\max\{m_{s},\widetilde{m}_{\beta}\}$ we have that $\mc{H}^{\beta/2}(f) \hookrightarrow H^{s}(\mb{R}^{n})$ is continuous. 
For $s \ge 0$, the Sobolev space $H^s(\Omega)$ on the bounded domain $\Omega$  with Lipschitz boundary, is defined via restriction as
$$
H^{s}(\Omega) = \left\{ v \in L^2(\Omega) \mid \exists\, U \in H^{s}(\mb{R}^n), U|_{\Omega} = v \right\}, \quad \norm{v}_{H^{s}(\Omega)} = \inf_{\substack{U \in H^{s}(\mb{R}^n) \\ U|_{\Omega} = v}} \norm{U}_{H^{s}(\mb{R}^n)}.
$$
Consider the restriction operator {\fontsize{11.4}{14}$\mc{R}_{\Omega}: \mc{H}^{\beta/2}(f) \to H^s(\Omega)$} defined by $\mc{R}_{\Omega}u = u|_{\Omega}$. For any $u \in \mc{H}^{\beta/2}(f)$, the function $u$ itself acts as an admissible global extension in $H^s(\mb{R}^n)$ for $u|_{\Omega}$. Therefore, 
$$
\norm{\mc{R}_{\Omega}u}_{H^{s}(\Omega)} = \norm{u|_{\Omega}}_{H^{s}(\Omega)} \le \norm{u}_{H^{s}(\mb{R}^n)} \le C_{s,\beta}\, \norm{u}_{\mc{H}^{\beta/2}(f)}.
$$
This shows that the restriction mapping $\mc{R}_{\Omega}$ is well-defined, linear, and continuous, establishing the embedding $\mc{H}^{\beta/2}(f) \hookrightarrow H^{s}(\Omega)$. Therefore the embedding chain \eqref{cadena de embeddings caso acotado 1} holds.
\end{proofofclaim}
\justify
Define the subset
\begin{equation}\label{conjutno asociado para dominios acotados}
 \mathcal{M}_{\tau}=\{ u \in H^{s}(\Omega) \mid \norm{u}_{H^{s}(\Omega)}\leq \tau   \},
\end{equation}
with $0< \tau < 1$ chosen below. This is a bounded, closed and convex subset of $H^{s}(\Omega)$.\\ Define the map
\begin{equation}\label{operador asociado para dominios acotados}
    \mathcal{N}: \mathcal{M}_{\tau} \to H^{s}(\Omega), \quad \mathcal{N}(u)=\widetilde{u}
\end{equation}
where $\widetilde{u}$ is the solution of the linear nonhomogeneous system $L_{\Omega}\widetilde{u}=\eta_{u}$. Since $\eta_{u} \in L^{2}(\Omega)$ by Lemma  
\ref{sistema lineal-solucion fuerte}, there exists a unique solution $\widetilde{u}$ such that $\widetilde{u} \in \mc{D}(L_{\Omega})$, then $\mc{N}$ is well defined.
We claim that there exists $\tau>0$ such that $\mathcal{N}: 
\mathcal{M}_{\tau} \to \mathcal{M}_{\tau}$ is well defined. If $\norm{h}_{L^{2}(\Omega)}$ is sufficiently small, one can choose $0< \tau < 1$ and $0 < \rho_{\tau} < 1$ such that $\norm{h}_{L^{2}(\Omega)} < \rho_{\tau}$ and by the chain \eqref{cadena de embeddings caso acotado 1} and estimate \eqref{desigualdad 2 lema caso acotado}, we have the existence of a constant $C_{*}>0$ such that
\begin{align}\label{sucesion acotada caso acotado}
    \norm{\mathcal{N}(u)}_{H^{s}(\Omega)} &= \norm{\widetilde{u}}_{H^{s}(\Omega)}\leq \norm{\widetilde{u}}_{\mc{D}(L_{\Omega})} \leq C_{*}\norm{\eta_{u}}_{L^{2}(\Omega)}
    \leq C_{*}\widetilde{C}\left(\norm{h}_{L^{2}(\Omega)} + C^{\alpha}_{s}\tau^{\alpha}  \right) <  \tau.
\end{align} 
Therefore $\mathcal{N}(u) \in 
\mathcal{M}_{\tau}$ and $\mathcal{N}: \mathcal{M}_{\tau} \to \mathcal{M}_{\tau}$.
Next we prove that $\mc{N}: \mc{M}_{\tau} \to \mc{M}_{\tau}$ is compact.
Let $(u_{n}) \subset \mathcal{M}_{\tau}$ be a bounded 
sequence; by estimate \eqref{sucesion acotada caso acotado}, $\mathcal{N}(u_{n})$ is also 
bounded in $H^{s}(\Omega)$, even more $\mc{N}(u_{n}) \subset \mc{M}_{\tau}$. Using the embedding \eqref{cadena de embeddings caso acotado 3} in the case of $q=2$, we have that there exists a subsequence $\mc{N}(u_{n_{i}}) \subset \mc{N}(u_{n})$ that converges strongly in  $\| \cdot\|_{L^{2}(\Omega)}$ to an $y \in L^{2}(\Omega)$.\\
We will prove that $\mc{N}(u_{n_{i}})$ is a Cauchy sequence in $H^{s}(\Omega)$. Since $H^{s}(\Omega)$ is a Hilbert space, strong convergence will then follow. First of all, note that $H^{s}(\Omega)$ can be expressed by complex interpolation \cite[c.2.8]{Lunardi} as $H^{s}(\Omega)=[L^{2}(\Omega), H^{\beta/2}(\Omega)]_{\theta}$ with $\theta \in (0,1)$ for $\theta=\frac{2s}{\beta}$. 

By the claim \eqref{claim caso acotado}, the chain $\mc{H}^{\beta/2}(f) \hookrightarrow H^{\beta/2}(\Omega) \hookrightarrow H^{s}(\Omega)$ is continuous. Then $\mc{D}(L_{\Omega}) \hookrightarrow H^{\beta/2}(\Omega)$, which implies by \eqref{sistema lineal-solucion fuerte} that $\|\mc{N}(u_{n})\|_{H^{\beta/2}(\Omega)} \leq \|\mc{N}(u_{n})\|_{\mc{D}(L_{\Omega})} \leq C\|\eta_{u_{n}}\|_{L^{2}(\Omega)}$; in fact, since $\|u_n\|_{H^s(\Omega)} \leq \tau$ for all $u_n \in \mc{M}_\tau$, we have the uniform bound
$$
\sup_{n \in \mb{N}} \|\eta_{u_n}\|_{L^2(\Omega)} \leq \widetilde{C}\bigl(\|h\|_{L^2(\Omega)} + C_s^\alpha \tau^\alpha\bigr) < \infty,
$$
so there exists $M>0$ which does not depend on $u_{n}$ such that $\|\mc{N}(u_{n})\|_{H^{\beta/2}(\Omega)} \leq M$.
Applying the interpolation inequality \cite[c.2.1.8]{Lunardi}, there exist constants $C_{I}, K>0$ such that
\begin{align*}
   \norm{\mc{N}(u_{n_{i}})-\mc{N}(u_{n_{j}}) }^{2}_{H^{s}(\Omega)} &\leq C_{I} \|\mc{N}(u_{n_{i}})-\mc{N}(u_{n_{j}})\|_{L^{2}(\Omega)}^{2(1-\theta)} \|\mc{N}(u_{n_{i}})-\mc{N}(u_{n_{j}})\|_{H^{\beta/2}(\Omega)}^{2\theta}\\
   %&\leq C_{I} \|\mc{N}(u_{n_{i}})-\mc{N}(u_{n_{j}})\|_{L^{2}(\Omega)}^{1-\theta} \|\mc{N}(u_{n_{i}})-\mc{N}(u_{n_{j}})\|_{\mc{H}^{\beta/2}(\Omega)}^{\theta}\\
   &\leq C_{I} \|\mc{N}(u_{n_{i}})-\mc{N}(u_{n_{j}})\|_{L^{2}(\Omega)}^{2(1-\theta)} \|\mc{N}(u_{n_{i}})-\mc{N}(u_{n_{j}})\|_{\mc{D}(L_{\Omega})}^{2\theta}\\
   &= C_{I} \|\mc{N}(u_{n_{i}})-\mc{N}(u_{n_{j}})\|_{L^{2}(\Omega)}^{2(1-\theta)}\Big( \norm{\mc{N}(u_{n_{i}})-\mc{N}(u_{n_{j}}) }^{2}_{L^{2}(\Omega)}\\
   &+ \norm{L_{\Omega}\mc{N}(u_{n_{i}})-L_{\Omega}\mc{N}(u_{n_{j}}) }^{2}_{L^{2}(\Omega)} \Big)^{\theta}\\
   &\leq C_{I}\|\mc{N}(u_{n_{i}})-\mc{N}(u_{n_{j}})\|_{L^{2}(\Omega)}^{2(1-\theta)}\left(K\|\eta_{u_{n_{i}}} - \eta_{u_{n_{j}}} \|_{L^{2}(\Omega)}^{2}   \right)^{\theta}\\
   &\leq C_{I}K^{\theta}\|\mc{N}(u_{n_{i}})-\mc{N}(u_{n_{j}})\|_{L^{2}(\Omega)}^{2(1-\theta)}\left(\| \eta_{u_{n_{i}}} \|_{L^{2}(\Omega)} + \| \eta_{u_{n_{j}}} \|_{L^{2}(\Omega)}\right)^{2\theta}.
\end{align*}
We have the uniform estimate
$
\sup_{i \in \mb{N}} \|\eta_{u_{n_{i}}}\|_{L^{2}(\Omega)} \leq \widetilde{C}\bigl(\|h\|_{L^{2}(\Omega)} + C_s^\alpha \tau^{\alpha}\bigr) \leq \widetilde{K}\bigl(\|h\|_{L^{2}(\Omega)} + \tau^{\alpha}\bigr)
$
for the constant $\widetilde{K} := \widetilde{C}\max\{1, C_s^\alpha\} > 0$ independent of $i$ and $j$,
 using the growth conditions \eqref{desigualdad de caso acotado}. Therefore, for some constant $K^{\prime}>0$ we will have that
\begin{equation}
  \norm{\mc{N}(u_{n_{i}})-\mc{N}(u_{n_{j}}) }^{2}_{H^{s}(\Omega)} \leq  C_{I}K^{\theta}K^{\prime}\|\mc{N}(u_{n_{i}})-\mc{N}(u_{n_{j}})\|_{L^{2}(\Omega)}^{2(1-\theta)}(\norm{h}_{L^{2}(\Omega)} + \tau^{\alpha})^{2\theta}.
\end{equation}
Since $\mc{N}(u_{n_{i}})$ is a Cauchy sequence in $L^{2}(\Omega)$, we have that it is also Cauchy in $H^{s}(\Omega)$. Therefore the map $\mc{N}: \mc{M}_{\tau} \to \mc{M}_{\tau}$ is compact.
Now we prove the continuity of $\mathcal{N}:\mathcal{M}_{\tau} 
\to \mathcal{M}_{\tau}$. Let $(u_{k}) \subset \mathcal{M}_{\tau}$ with $u_{k} \to u$ in $H^{s}(\Omega)$ for some $u \in H^{s}(\Omega)$. 
Denote $\mc{G}(x,u)=U(x,u) + u$, using the estimate \eqref{desigualdad 2 lema caso acotado}, and the embedding chain \eqref{cadena de embeddings caso acotado 2},
we have that the following computations
\begin{align*}
   \norm{\mathcal{N}(u_{k}) - \mathcal{N}(u)}^{2}_{H^{s}(\Omega)} 
   &= \norm{\widetilde{u}_{k}-\widetilde{u}}^{2}_{H^{s}(\Omega)} \leq \norm{  \widetilde{u}_{k}-\widetilde{u} }^{2}_{\mc{D}(L_{\Omega})}=\norm{\widetilde{u}_{k}-\widetilde{u}}_{L^{2}(\Omega)}^{2}+\norm{L_{\Omega}\widetilde{u}_{k} - L_{\Omega}\widetilde{u}}_{L^{2}(\Omega)}^{2} \\
&=\norm{\widetilde{u}_{k}-\widetilde{u}}_{L^{2}(\Omega)}^{2} + \norm{\eta_{u_{k}}-\eta_{u}}_{L^{2}(\Omega)}^{2}\leq(E+1) \norm{\eta_{u_{k}} - \eta_{u}  }^{2}_{L^{2}(\Omega)}\\
&= (E+1)\int_{\Omega}\Big|  \int_{0}^{1}\frac{d}{dt}\mc{G}(x,tu_{k}(x) + (1-t)u(x))dt \Big|^{2}dx, \\
&=(E+1)\int_{\Omega}|u_{k}(x)-u(x)|^{2}\Big| \int_{0}^{1}\frac{\partial}{\partial y}\mc{G}(x,tu_{k}(x)+(1-t)u(x) )dt  \Big|^{2}dx
\end{align*} 
with $E>0$ constant. Define the function $g_{x}: [0,1] \to \mathbb{R}, g_{x}(t)=\frac{\partial}{\partial y}\mc{G}(x, tu_{k}(x) 
   +  (1-t)u(x))$ with $x \in 
\Omega$; 
since $U(\cdot, u)$ is $C^{1}$ in the second variable, $g_{x}(t)$ is a well defined Carathéodory function(i.e., measurable with respect to $x$ and continuous with respect to $t$). Note that
\begin{equation*}
    \Big|\int_{0}^{1}g_{x}(t)dt \Big| \leq \int_{0}^{1} |g_{x}(t)|dt \leq \int_{0}^{1}\left( \sup_{t \in [0,1]}|g_{x}(t)|  \right)dt = \sup_{t \in [0,1]}|g_{x}(t)|
\end{equation*}
Since $g_{x}(t)$ is continuous in a compact set, the maximum is attained for each $x \in \Omega$, i.e., for each fixed $x_{0} \in \Omega$, there exists $t_{0} \in [0,1]$ such that $g_{x_{0}}$ attains the supremum in $t_{0}$. 
There exists a Borel measurable function $t(\cdot): \Omega \subset \mb{R}^{n} \to [0,1]$ by \cite[t.18.13]{AliprantisBorder}
such that for each $x^{\prime} \in \Omega$, $t(x^{\prime})$ satisfies $\sup_{t \in [0,1]}|g_{x^{\prime}}(t)|=|g_{x^{\prime}}(t(x^{\prime}))|$. Using (\ref{desigualdad de caso acotado}) there exists a constant $D>0$ such that
\begin{align*}
  \norm{\mathcal{N}(u_{k}) - \mathcal{N}(u)}^{2}_{H^{s}(\Omega)}  &\leq D\int_{\Omega}|u_{k}(x)-u(x)|^{2}\left\{ \sup_{t \in [0,1]}\Big| \frac{\partial}{\partial y}\mc{G}(x,tu_{k}(x)+(1-t)u(x) )   \Big|   \right\}^{2}dx    \\
  &\leq D\int_{\Omega }|u_{k}(x) - u(x)|^{2} \left\{  \sup_{t \in [0,1]}\left|  1 + |tu_{k}(x) + (1-t)u(x)|^{\mu } \right| \right\}^{2}dx \\
  &\leq D\int_{\Omega }|u_{k}(x) - u(x)|^{2} \left\{   1 + |t(x)u_{k}(x) + (1-t(x))u(x)|^{\mu } \right\}^{2}dx
\end{align*}
For the third inequality, we use the fact $0 \leq t(x) \leq 1$ to obtain the following computations
\begin{align*}
1 + |t(x)u_{k}(x) + (1-t(x))u(x)|^{\mu} 
&\leq 1 + \left( |t(x)u_{k}(x)| + |(1-t(x))u(x)|  \right)^{\mu} \\
&\leq 1 + c_{1}\left( |t(x)u_{k}(x)|^{\mu} + |(1-t(x))u(x)|^{\mu}\right)\\
&\leq 1 + c_{2}\left( |u_{k}(x)|^{\mu} + |u(x)|^{\mu}\right),
\end{align*}
with $c_{1},c_{2}>0$ constants. Therefore, there exists a constant $\widetilde{D}>0$ such that
\begin{align*}
  \norm{\mathcal{N}(u_{k}) - \mathcal{N}(u)}^{2}_{H^{s}(\Omega)} &\leq \widetilde{D}\int_{\Omega}|u_{k}(x) - u(x)|^{2}\left(1 + |u_{k}(x)|^{2\mu} + |u(x)|^{2\mu}\right) dx  \\
   &= \widetilde{D}\Bigg(\int_{\Omega}|u_{k}(x) - u(x)|^{2} dx + \int_{\Omega}|u_{k}(x) - u(x)|^{2}|u_{k}(x)|^{2\mu } dx \\
   &+ \int_{\Omega}|u_{k}(x) - u(x)|^{2}|u(x)|^{2\mu} dx \Bigg) \leq \widetilde{D} \norm{u_{k} - u}^{2}_{H^{s}(\Omega)} \Big(   1 + 2(B \tau)^{2\mu}    \Big) .
\end{align*}
This last inequality is satisfied by the chain\,\eqref{embedding caso dominios acotados}, specifically $H^{s}(\Omega) \hookrightarrow L^{\infty}(\Omega)$, so there exists a constant $B_{\tau}>0$ such that
$\|u_{k}\|_{L^{\infty}(\Omega)}, \|u\|_{L^{\infty}(\Omega)} \leq B\tau$.
Therefore $\mathcal{N}(u_{k}) \to \mathcal{N}(u)$ in $H^{s}(\Omega)$. Applying Schauder's fixed point theorem yields a fixed point $u \in \mathcal{M}_{\tau}$ such that $u = \mathcal{N}(u)$. Since $\mathcal{N}(\mathcal{M}_{\tau}) \subset \mc{D}(L_{\Omega})$ by definition of $\mathcal{N}$, we conclude that $u \in \mc{D}(L_{\Omega})$ satisfies $L_{\Omega}u = \eta_{u}$, which gives the desired existence result.
\end{proof}

\vspace{0.12cm}

\section{Nonlocal Equations on Compact Riemannian Manifolds}\label{Nonlocal Equations on Riemannian Manifolds}

In this section we consider nonlocal equations of the form  
\begin{equation}\label{ecuacion no local en Riemannian manifolds}
  f(\Delta_g)\phi - U(x, \phi) = 0  
\end{equation}
on a compact, connected, oriented Riemannian manifold $(M,g)$. The \emph{Laplacian} $\Delta_g$ associated to $g$ is a second-order differential operator acting on smooth functions, i.e., $\Delta_g: C^\infty(M) \to C^\infty(M)$, which is given by
\begin{equation}
 \Delta_g f =- \text{div}_g(\nabla_g f),   
\end{equation}
where $\nabla_g f$ is the gradient of $f$, defined as the unique vector field $X$ such that 
$
g(X, Y) = df(Y),~\forall Y \in \mathfrak{X}(M)
$. The operator $\text{div}_g$ is the divergence of a vector field $X=X^{i}\frac{\partial}{\partial x^{i}}$ in a local coordinates $(U,(x^{i}))$, defined in terms of the volume form $dV_g$ induced by $g$,
\begin{equation}
    \text{div}_g(X) = \frac{1}{\sqrt{|g|}} \partial_i \big(\sqrt{|g|} X^i\big),
\end{equation}
where $\det g$ is the determinant of the matrix $(g_{ij})$ of the metric in local coordinates. 
In local coordinates $(x^i)$ on $M$, the Laplacian $\Delta_g$ of a function $f$ is written as,
\begin{equation}\label{Operador de Laplace-Beltrami}
    \Delta_g f =  -\frac{1}{\sqrt{|g|}} \partial_i \left(\sqrt{|g|} g^{ij} \partial_j f \right),
\end{equation}
where $g^{ij}$ are the coefficients of the inverse of the metric $g_{ij}$.
Another intrinsic definition of $\Delta_g$ is
$
\Delta_{g} u = -\ast d \ast d u,  
$
where $d$ denotes the exterior derivative and $\ast$ is the Hodge star operator. Recall that by definition the Hodge star operator $\star: \Omega^k(M) \to \Omega^{n-k}(M)$ assigns to any $k$-form $\eta \in \Omega^k(M)$ the $(n - k)$-form $\star\eta \in \Omega^{n-k}(M)$ such that 
\begin{equation*}
   \omega \wedge \star\eta = \langle \omega, \eta \rangle \omega_g, \quad  \omega \in \Omega^k(M).
\end{equation*}

\justify
We recall the following definition of spaces $(L^{2}(M,g), dV_{g})$

\begin{definition}
    The space $L^2(M, g)$ is defined as the set of square-integrable functions on the Riemannian manifold $(M, g)$ with respect to the measure induced by the metric $g$, namely
\begin{equation}\label{espacio de funciones cuadrado integrables en una variedad Riemanniana}
    L^2(M, g) = \left\{ f: M \to \mathbb{R}(\mathbb{C}) ~\Big|~ \int_M |f|^2 \, dV_g < \infty \right\}
\end{equation}
Here $dV_g$ is the volume measure on $M$ associated with $g$, defined by $dV_g = \sqrt{\det(g)} \, dx^1 \wedge dx^2 \wedge \dots \wedge dx^n$ (also denoted $\omega_{g}$). This is a Hilbert space equipped with the inner product
$$
\langle f, h \rangle_{L^2(M)} = \int_M f \overline{h} \, dV_g,
$$
and the norm induced by this inner product is
  $$
    \|f\|_{L^2(M)} = \sqrt{\langle f, f \rangle_{L^2(M)}} = \left( \int_M |f|^2 \, dV_g \right)^{1/2}.
    $$
\end{definition}

\justify

The Laplace–Beltrami operator $\Delta_{g}$ is originally defined on $C^{\infty}(M)$, but it extends naturally to $L^{2}(M,g)$ from a dense domain; indeed
\begin{equation*}
\Delta_{g}: \mathcal{D}(\Delta_{g}) \subset L^{2}(M,g) \to L^{2}(M,g)
\end{equation*}
where the maximal domain $\mathcal{D}(\Delta_{g})=\{u \in H^{2}_{\text{loc}}(M) \cap L^{2}(M) \mid \Delta_{g}u \in L^{2}(M) \}$.

\begin{proposition}\cite{Backman,Canzani,Rosenberg}.
    The operator $\Delta_{g}$ has the following properties.\\
    i)~ $\Delta_g$ is a second-order elliptic operator and is essentially self-adjoint with respect to the inner product on $L^2(M, g)$, that is
$
\langle f, \Delta_g h \rangle = \langle \Delta_g f, h \rangle.
$\\
ii)~ If $M$ is compact, the spectrum of $\Delta_g$ is a discrete set of nonnegative eigenvalues $0\leq \lambda_{1}<\lambda_{2} \leq \lambda_{3} \leq \ldots$, with an orthonormal basis of eigenfunctions $\{\omega_i\}$ in $L^2(M, g)$ such that $
\Delta_g \omega_i = \lambda_i \omega_i.$\\
iii)~ The operator $\Delta_g$ commutes with isometries of the manifold $(M, g)$. This means that if $\varphi: M \to M$ is an isometry then
$$
\Delta_g (f \circ \varphi) = (\Delta_g f) \circ \varphi.
$$
iv)~ The operator $\Delta_{g}$ is unbounded and $\Delta_{g}\geq 0$.
\end{proposition}

\justify

The operator $\Delta_{g}: H^{2}(M,g) \to L^{2}(M,g)$ is an unbounded self-adjoint operator on a compact Riemannian manifold, and the resolvent operator $R_{\lambda}(\Delta_{g})$ defined by
\begin{equation*}
R_{\lambda}(\Delta_{g}): L^{2}(M,g) \to H^{2}(M,g), \quad \psi \mapsto  (\Delta_{g} - \lambda I)^{-1}\psi
\end{equation*}
is linear, bounded, self-adjoint and compact \cite{ReedI}.
Therefore, by the Theorem \cite[p.5.12]{Schmudgen} 
%(\ref{Compacidad de resolvente})  
there exist $(\lambda_{n})_{n \in \mathbb{N}} \subset \mathbb{R}$ and an orthonormal basis $\{ \omega_{n} \mid n \in \mathbb{N}\}$ of $L^{2}(M,g)$ such that $0\leq\lambda_{1}<\lambda_{2} \leq \lambda_{3} \leq \ldots$ with finite multiplicity and 
accumulation at infinity, i.e., $|\lambda_{j}| \to \infty$, and $\Delta_{g}\omega_{n}=\lambda_{n} \omega_{n}$ for $n \in \mathbb{N}$. 
Another way to justify the existence of an orthonormal basis of $L^{2}(M,g)$ consisting of eigenfunctions of $\Delta_{g}$
is via the Hodge Theorem for functions \cite[t.1.29]{Rosenberg}. 
Using the spectral theorem for unbounded self-adjoint operators 
%(\ref{Teorema espectral para operadores autoadjuntos no acotados})
\cite[t.5.7]{Schmudgen} there exists a
unique spectral measure $E_{\Delta_{g}}: \mathcal{B}(\mathbb{R}) \to \mathcal{L}(L^{2}(M,g))$ such that
\begin{equation}\label{medida espectral para el laplaciano beltrami}
\hspace{0.3cm}~\Delta_{g}=\int_{\mathbb{R}}\lambda dE_{\Delta_{g}}(\lambda).
\end{equation}
In order to use the functional calculus, by the assignment $f \mapsto f(\Delta_{g})$, one requires that $f \in \mathcal{S}(\mathbb{R}, \mathcal{B}(\mathbb{R}), E_{\Delta_{g}})$, which is the set of all $E_{\Delta_{g}}$-a.e. finite Borel functions $f: \mathbb{R} \to \mathbb{C} \cup \{\infty\}$, that is, $E_{\Delta_{g}}(\{t \in \mathbb{R} \mid f(t)=\infty \})=0$. Since $f \in \mc{G}^{\beta}$, and $f$ is a real-valued measurable function, it follows that $f \in \mathcal{S}(\mathbb{R}, \mathcal{B}(\mathbb{R}), E_{\Delta_{g}})$. Applying that $\sigma(\Delta_{g})$ is purely discrete, we have that
\begin{equation}\label{expresion 2, Gorka, On a class}
    f(\Delta_{g})=f\left( \int_{\mathbb{R}}\lambda dE_{\Delta_{g}}(\lambda)    \right)= \int_{\mathbb{R}}f(\lambda)dE_{\Delta_{g}}(\lambda)=\sum_{i=1}^{\infty}f(\lambda_{i})
E_{\Delta_{g}}(\{\lambda_i\})\footnote{Recall that $E_{\Delta_{g}}(\lambda_{i})=E_{\Delta_{g}}((-\infty, \lambda_{i}])$ is a resolution of the identity.}
\end{equation}
For an operator with a point spectrum, the spectral measure $E_{\Delta_{g}}(\cdot)$ evaluated at the singleton $\{\lambda_{i}\}$ yields $P_{i}$. Specifically, each $P_{j}$ is the orthogonal projection operator onto the closed subspace corresponding to the eigenspace associated with the eigenvalue $\lambda_{j}$. More concretely
\begin{equation}
E_{\Delta_{g}}(\{\lambda_{i}\}) =P_{i}, \quad E_{\Delta_{g}}((-\infty, \lambda_{i}])=\sum_{\lambda_{j} \leq \lambda_{i}}P_{j}.
\end{equation}
Since every closed subspace $\mathcal{H} \subseteq L^{2}(M,g)$ 
can be written as $\mathcal{H}=\overline{\text{span}\{\omega_{n_{i}}\}_{i \in I}}$, it 
follows that for each $C \in \mathcal{B}(\mathbb{R})$ the expression $E_{\Delta_{g}}(C)\psi$ has the form
$$E_{\Delta_{g}}(C)\psi=\sum_{\lambda_{i} \in C}\langle \psi, \omega_{n_{i}}   \rangle \omega_{n_{i}}$$
and therefore (\ref{expresion 2, Gorka, On a class}) is equivalent to
\begin{equation}
 f(\Delta_{g})\psi= \int_{\mathbb{R}}f(\lambda)dE_{\Delta_{g}}(\lambda)\psi = \sum_{i=1}^{\infty}f(\lambda_{i})P_{i}\psi   =\sum_{i=1}^{\infty}f(\lambda_{i})\langle \psi, \omega_{i}  \rangle \omega_{i}  = \sum_{i=1}^{\infty}f(\lambda_{i})\psi_{i}\omega_{i}.
\end{equation}

\justify

The last identity will serve as a tool to study the linear equation
\begin{equation}\label{1.(33)}
    f(\Delta_{g})\phi + \nu \phi=h, \quad  \nu \in \mb{R}_{>0}
\end{equation}
where $f$ is a nonnegative function. If $h \in L^2(M,g)$, using the eigenfunction basis of $\Delta_{g}$ one can write $h = \sum h_{i}\omega_{i}$, and a straightforward calculation yields $\phi_i = \frac{h_i}{f(\lambda_i) + \nu}.$ Hence $\phi$ is
\begin{equation}\label{unica solucion caso lineal caso Riemanniano}
    \phi = \sum_{i=1}^\infty \frac{h_i}{f(\lambda_i) + \nu} \omega_i,
\end{equation}
By its form this function belongs to $L^{2}(M,g)$; however, it need not be smooth for that we introduce the following spaces. 
Consider the following Hilbert space $(l^{2,k}(\mathbb{N}),  \langle \alpha , \beta  \rangle_{l^{2,k}(\mathbb{N})} )$
\begin{equation*}
    l^{2,k}(\mathbb{N}) = \left\{ \alpha = (\alpha_i) ~\Big|~ \alpha_{i} \in \mathbb{R}, \sum_{i=1}^{\infty} |\alpha_i|^2 (1 + \lambda_i^2)^k < \infty \right\}, ~
    \langle \alpha , \beta  \rangle_{l^{2,k}(\mathbb{N})} = \sum_{i=1}^{\infty}\alpha_{i}\beta_{i}(1+\lambda_{i}^{2})^{k}
\end{equation*}
The notation with the eigenvalues $\lambda_{i}$ squared is used in \cite{Gorka 6} to emphasize that the eigenvalues are positive. Also define the space
\begin{align}
\widetilde{H}^k(M,g) &=  \left\{ h=\sum_{i=1}^{\infty}\alpha_{i}\omega_{i} ~\Big|~ \sum_{i=1}^{\infty}|\alpha_{i}|^{2}(1+\lambda_{i}^{2})^{k} < \infty    \right\}
\end{align}
This space is endowed with the following inner product
\begin{equation}
    \langle  h_{1} , h_{2}   \rangle_{\widetilde{H}^k(M,g)} = \langle (\alpha_{i})  , (\beta_{i})  \rangle_{l^{2,k}(\mathbb{N})} = \sum_{i=1}^{\infty}\alpha_{i}\beta_{i}(1+\lambda_{i}^{2})^{k},
\end{equation}

\justify

The space $\widetilde{H}^{k}(M,g)$ is isometrically isomorphic to $l^{2,k}(\mathbb{N})$. Moreover, if $k \in \mathbb{N}$ then $\widetilde{H}^{k}(M,g) = H^{2k}(M,g)$ \cite[l.4.1]{Gorka 6}, where the latter is the Sobolev space on the Riemannian manifold $(M,g)$.\\
We assume that $f$ is a nonnegative real-valued function satisfying the ellipticity condition for the spectrum $(\lambda_{i})$ of the Laplace operator, namely
\begin{equation}\label{codicion de elipticidad de una fucnion real valuada para Riemannian manifolds}
   \forall k>0, \exists\, C_{k}>0 :  C_{k}(1+\lambda_{i}^{2})^{k/2} \leq 1+ f(\lambda_{i}).
\end{equation}
This condition replaces the $(P)$ and $(E_{\beta})$ conditions. Introduce the space
\begin{equation}\label{espacio de Sobolev en verson sumas}
    \widetilde{H}_{f}(M,g) = \left\{\phi=\sum_{i=1}^{\infty} \alpha_{i}\omega_{i} ~\Big|~ \sum_{i=1}^{\infty}|\alpha_{i}|^{2}(1 + f(\lambda_{i}))^{2} <  \infty   \right\}
\end{equation}
Thanks to the ellipticity condition, for every $k$ we have the embedding
\begin{equation}\label{embedding caso riemaniano}
  \widetilde{H}_{f}(M,g) \hookrightarrow H^{2k}(M,g)  
\end{equation}
If $h \in  \widetilde{H}_{f}(M,g)$ then $h$ is a smooth function on $M$ by \cite[t.2.7]{Hebey 2}. 
\vspace{0.12cm}

Define the operator $L_{c}=f(\Delta_{g}) + \mathrm{Id}$ by
\begin{equation}\label{operador caso Riemanniano}
 L_{c}: \widetilde{H}_{f}(M,g) \longrightarrow L^{2}(M,g), \quad \phi \longmapsto \sum (f(\lambda_{i}) + 1)\phi_{i}\omega_{i} 
\end{equation}

\begin{lemma}\label{sistema lineal caso riemanniano}
    For every $h \in L^2(M,g)$, the linear problem
\begin{equation}
f(\Delta_g)\phi + \phi = h
\end{equation}
admits a unique solution $\phi \in \widetilde{H}_f(M,g)$. Moreover,
$\norm{\phi}_{\widetilde{H}_{f}(M,g)}=\norm{h}_{L^{2}(M,g)}$.
\end{lemma}

\vspace{0.12cm}

\begin{theorem}\label{teorema ecuacion no local caso Riemannian manifold}\cite[t.4.1]{Gorka 6}.
      Suppose $f$ satisfies (\ref{codicion de elipticidad de una fucnion real valuada para Riemannian manifolds}), that $U: M \times \mathbb{R} \to \mathbb{R}$, and that there exist constants $\alpha > 1$, $\beta \geq 1$, $C > 0$, and a function $h \in L^2(M,g)$ such that the following inequalities hold:
\begin{equation}
|U(x, y) + y| \leq C \left( |h(x)| + |y|^\alpha \right),
\end{equation}
\begin{equation}
    \left| \frac{\partial}{\partial y} \left( U(x, y) + y \right) \right| \leq C ( 1 + |y|^\beta )
\end{equation}
Then, if $\norm{h}_{L^2(M,g)}$ is sufficiently small, there exists a solution $\phi \in \widetilde{H}_f(M,g)$ to the nonlinear problem
\begin{equation}
f(\Delta_g)\phi - U(x, \phi) = 0.
\end{equation}
\end{theorem}

\vspace{0.12cm}

The following result is a generalization of the localized result (\ref{caso localizado}) on Riemannian manifolds. We seek to prove the existence of solutions to equations of the form (\ref{ecuacion caso perturbado Riemanniano}). Indeed,

\begin{theorem}\label{teorema del caso perturbado para Riemannian manifolds}
Let $(M,g)$ be a compact, connected and oriented Riemannian $n$-manifold without boundary. Let $U : M \times \mathbb{R} \to \mathbb{R}$
be a Carathéodory non-linearity which is a  $C^{1}$ in the second variable, $f$ a 
continuous, nonnegative real-valued function satisfying condition (\ref{codicion de elipticidad de una fucnion real valuada para Riemannian manifolds}). Let $L : L^{2}(M,g) \to L^2(M,g)$ be a linear and
bounded operator.
%be a linearbounded operator and $C^{1}$
Suppose there exist constants $\alpha > 1$, $\beta \ge 1$, $C > 0$, and a $\delta > 0$ sufficiently small, 
a function $\psi \in C^{\infty}(M)$, $\chi \in L^{\infty}(M,g)$ 
and a function $h\in L^2(M,g)$, such that the following holds
\begin{equation}\label{desigualdad numero 1 del caso pertubado}
    |U(x,y) + y| \le C\,\Bigl(|h(x)| + |y|^\alpha\Bigr),
\end{equation}
\begin{equation}\label{desigualdad numero 2 del caso pertubado}
    \left|\frac{\partial}{\partial y}\Bigl(U(x,y) + y\Bigr)\right| \le C\,\Bigl(1+ |\chi(x)| + |y|^\beta\Bigr).
\end{equation}
Then, if $\norm{h}_{L^{2}(M,g)}$ is sufficiently small, there exists a solution $\phi\in \widetilde{H}_f(M,g)$ of the non-linear equation
\begin{equation}\label{ecuacion caso perturbado Riemanniano}
    f(\Delta_{g})\phi - \delta\psi L(\phi) - U(x,\phi)= 0
\end{equation}
\end{theorem}

\begin{proof}
The nonlinear equation (\ref{ecuacion caso perturbado Riemanniano}) is equivalent to
\begin{equation}\label{sistema caso perturbado Riemaniano}
    L_{c}u=Q(\cdot, u)
\end{equation}
where $Q(\cdot,u)=u + \delta \psi L(u) + U(\cdot, u)$.
Let $r,s \in \mb{R}$ such that $r>s>\frac{n}{2}$, without loss of generality suppose that $s=\theta r$ with $\theta \in (0,1)$. By  the Sobolev embedding Theorem \cite{Taylor III, Hebey 2} and the embedding (\ref{embedding caso riemaniano}), we have that the following embedding chain
\begin{equation}\label{embedding caso perturbado 1}
    \widetilde{H}_{f}(M,g) \hookrightarrow  
    H^{r}(M,g)
    \hookrightarrow H^{s}(M,g) \hookrightarrow L^{\infty}(M,g).
\end{equation}
is continuous. Then there exist constants $C_{r}, C_{r,s}, C_{s,0} > 0$ such that $\|v\|_{H^{r}(M)} \leq C_{r}\|v\|_{\widetilde{H}_{f}(M)}$, $\|v\|_{H^{s}(M)} \leq C_{r,s}\|v\|_{H^{r}(M)}$, and $\|v\|_{L^{2}(M)} \leq C_{s,0}\|v\|_{H^{s}(M)}$.
Let $u \in H^{s}(M,g)$; using the inequality (\ref{desigualdad numero 1 del caso pertubado}) and the embedding (\ref{embedding caso perturbado 1}), we show that $Q(\cdot, u) \in L^{2}(M,g)$. 
Denote the norm operator as  $\|\cdot\|_{\mathscr{L}(L^{2}(M), L^{2}(M))}$(or for simplicity $\|\cdot\|_{\mathscr{L}(L^{2}(M))}$).
We have the following computations for some constants $\tilde{C}, C_{\alpha}>0$,
\begin{align*}
\norm{Q(\cdot,u)}_{L^{2}(M)} &= \norm{u + \delta\psi L(u) + U(\cdot,u)}_{L^{2}(M)} \leq \norm{\delta \psi L(u)}_{L^2{(M)}} + \norm{U(\cdot, u) + u   }_{L^{2}(M)} \\
    &\leq \delta\norm{\psi}_{L^{\infty}(M)}\norm{L(u)}_{L^{2}(M)} + \tilde{C}\left(\norm{h}_{L^{2}(M)} +\norm{|u|^{\alpha}}_{L^{2}(M)}   \right) \\
    &\leq \delta C_{s,0}\norm{\psi}_{L^{\infty}(M)}\norm{L}_{\mathscr{L}(L^{2}(M))}\norm{u}_{H^{s}(M)} 
    + C_{\alpha}\left(\norm{h}_{L^{2}(M)} +\norm{u}_{H^{s}(M)}^{\alpha}   \right).
\end{align*}
The last inequality works because, if we take
$s_{\alpha} < s$ such that
$s_{\alpha} \in (\frac{n}{2}(\frac{\alpha - 1}{\alpha}) , s)$ using the Sobolev Theorem \cite[p.26]{Taylor III} and $vol_{g}(M)<\infty$, 
we have $H^{s}(M) \hookrightarrow H^{s_{\alpha}}(M) \hookrightarrow L^{2 \alpha}(M) \hookrightarrow L^{2}(M)$, and by $\norm{|u|^{\alpha}}_{L^{2}(M)}=\norm{u}^{\alpha}_{L^{2\alpha}(M)}$, the former inequality follows. 
Since $\delta>0$ can be taken as small as desired, we conclude $Q(\cdot,u) \in L^{2}(M)$. We will apply Schauder's fixed point theorem. Define
\begin{equation}
\mathcal{K}_{\varepsilon}=\{u \in H^{s}(M) \mid \norm{u}_{H^{s}(M)} \leq \varepsilon \}
\end{equation}
where $\varepsilon>0$ is arbitrary. 
Note that $\mathcal{K}_{\varepsilon}$ is closed, bounded, convex and nonempty subset in the Hilbert space $H^{s}(M)$.
Define the map
\begin{equation}\label{operador caso perturbado riemanniano}
    \mathcal{A}: \mathcal{K}_{\varepsilon} \to H^{s}(M), \quad \mathcal{A}(u)=\widetilde{u}
\end{equation}
where $\widetilde{u}$ is the unique solution of the linear problem $L_{c}\widetilde{u}=Q(\cdot,u)$.
We claim that there exists $\varepsilon >0$ such that $\mathcal{A}: \mathcal{K}_{\varepsilon} \to \mathcal{K}_{\varepsilon}$ is well defined. Indeed, if $u \in \mathcal{K}_{\varepsilon}$, using (\ref{embedding caso perturbado 1}) and (\ref{sistema lineal caso riemanniano}),
\begin{align*}
    \norm{\mathcal{A}(u)}_{H^{s}(M)} &\leq   C_{r,s}\norm{\mathcal{A}(u)}_{H^{r}(M)} = C_{r,s}\norm{\widetilde{u}}_{H^{r}(M)} \leq C_{r,s}C_{r}\norm{\widetilde{u}}_{\widetilde{H}_{f}(M)}= C_{r,s}C_{r}\norm{Q(\cdot, u)}_{L^{2}(M)} \\
    &\leq C_{r,s}C_{r}\left(\delta\norm{\psi}_{L^{\infty}(M)}\norm{L(u)}_{L^{2}(M)} + C_{\alpha}(\norm{h}_{L^{2}(M)} + \norm{u}_{H^{s}(M)}^{\alpha})\right)\\
    &\leq C_{r,s}C_{r}\left(\delta C_{s,0}\norm{\psi}_{L^{\infty}(M)}\norm{L}_{\mathscr{L}(L^{2}(M))}\norm{u}_{H^{s}(M)} + C_{\alpha}(\norm{h}_{L^{2}(M)} + \varepsilon^{\alpha})   \right).
\end{align*}
Since $\norm{h}_{L^{2}(M)}$ and $\delta>0$ are sufficiently small, setting $\widetilde{K}_{r,s} := C_{r,s}C_{r}$, we can choose $0 < \varepsilon < 1$ such that $\widetilde{K}_{r,s}C_{\alpha}\varepsilon^{\alpha} \leq \frac{\varepsilon}{3}$, and then take $\delta > 0$ and $\norm{h}_{L^{2}(M)}$ sufficiently small such that $\widetilde{K}_{r,s}\delta C_{s,0}\norm{\psi}_{L^{\infty}(M)}\norm{L}_{\mathscr{L}(L^{2}(M))} \leq \frac{1}{3}$ and $\widetilde{K}_{r,s}C_{\alpha}\norm{h}_{L^{2}(M)} \leq \frac{\varepsilon}{3}$, so that
\begin{align*}\label{operador bien definido caso perturbado}
    \norm{\mathcal{A}(u)}_{H^{s}(M)} &\leq C_{r,s} \norm{\mathcal{A}(u)}_{H^{r}(M)} \leq  \widetilde{K}_{r,s}\delta C_{s,0}\norm{\psi}_{L^{\infty}(M)}\norm{L}_{\mathscr{L}(L^{2}(M))}\norm{u}_{H^{s}(M)}\\
    &+ \widetilde{K}_{r,s}C_{\alpha}\left(\norm{h}_{L^{2}(M)}+ \varepsilon^{\alpha}   \right) \leq \varepsilon.
\end{align*}
Then $\mathcal{A}: \mathcal{K}_{\varepsilon} \to \mathcal{K}_{\varepsilon}$ is well defined.
Let us prove that $\mc{A}$ is a compact map. By the above estimate, for a bounded sequence $(u_{n}) \subset \mc{K}_{\varepsilon}$, it follows that the sequence $\mc{A}(u_{n}) \subset \mc{K}_{\varepsilon}$ is also bounded in $H^{r}(M)$. 
Using the Rellich-Kondrachov theorem for compact Riemannian manifolds \cite[t.2.9]{Hebey 2}, we have that $H^{s}(M) \hookrightarrow L^{\infty}(M)$ is a compact embedding, therefore there exists a subsequence $\mc{A}(u_{n_{k}}) \subset \mc{A}(u_{n})$ such that for some $z \in L^{\infty}(M)$, $\mc{A}(u_{n_{k}}) \to z$ in the norm $\|\cdot\|_{L^{\infty}}$. 
Since $vol_{g}(M) < \infty$, the embedding $L^{\infty}(M) \hookrightarrow L^{2}(M)$ is continuous, hence there exists $K>0$ such that $\|\mc{A}(u_{n_{k}}) - z\|_{L^{2}(M)} \leq K\|\mc{A}(u_{n_{k}}) - z\|_{L^{\infty}(M)}\to 0$. 
Therefore $\mc{A}(u_{n_{k}})$ is a Cauchy sequence in  $L^{2}(M)$. 
By complex interpolation, $H^{s}(M)=[L^{2}(M),H^{r}(M)]_{\theta}$. Then there exists a constant $C_{I}>0$ such that
{\fontsize{10.5}{14}
\begin{align*}
    \norm{\mc{A}(u_{n_{i}})-\mc{A}(u_{n_{j}})  }_{H^{s}(M)}^{2} &\leq C_{I}\norm{\mc{A}(u_{n_{i}})-\mc{A}(u_{n_{j}})}_{L^{2}(M)}^{2(1-\theta)}\norm{\mc{A}(u_{n_{i}})-\mc{A}(u_{n_{j}})}_{H^{r}(M)}^{2\theta}\\
&\leq C_{I}C_{r}^{2\theta}\norm{\mc{A}(u_{n_{i}})-\mc{A}(u_{n_{j}})}_{L^{2}(M)}^{2(1-\theta)}\norm{\mc{A}(u_{n_{i}})-\mc{A}(u_{n_{j}})}_{\widetilde{H}_{f}(M) }^{2\theta}\\
&=C_{I}C_{r}^{2\theta}\norm{\mc{A}(u_{n_{i}})-\mc{A}(u_{n_{j}})}_{L^{2}(M)}^{2(1-\theta)}\norm{Q(\cdot,u_{n_{i}}) - Q(\cdot,u_{n_{j}})}_{L^{2}(M)}^{2\theta}\\
&\leq C_{I}C_{r}^{2\theta}\norm{\mc{A}(u_{n_{i}})-\mc{A}(u_{n_{j}})}_{L^{2}(M)}^{2(1-\theta)}\left( \norm{Q(\cdot,u_{n_{i}})}_{L^{2}(M)} + \norm{Q(\cdot,u_{n_{j}})}_{L^{2}(M)}  \right)^{2\theta}
\end{align*}}

\justify

Since $(u_n) \subset \mc{K}_{\varepsilon}$ satisfies $\|u_n\|_{H^s(M)} \leq \varepsilon$ for all $n \in \mb{N}$, our previous estimate yields the uniform bound
$$
\sup_{n \in \mb{N}} \norm{Q(\cdot, u_n)}_{L^{2}(M)} \leq \delta C_{s,0}\varepsilon\norm{\psi}_{L^{\infty}(M)}\norm{L}_{\mathscr{L}(L^{2}(M))}  + C_{\alpha}\bigl(\norm{h}_{L^{2}(M)} + \varepsilon^{\alpha}\bigr) < \infty.
$$
Setting 
$$M_0 := 2\Bigl(\delta C_{s,0} \varepsilon\norm{\psi}_{L^{\infty}(M)}\|L\|_{\mathscr{L}(L^{2}(M))}  + C_{\alpha}\bigl(\|h\|_{L^{2}(M)} + \varepsilon^{\alpha}\bigr)\Bigr) > 0$$
which is a constant that does not depend on the indices $i,j$, it follows that $\norm{Q(\cdot,u_{n_{i}})}_{L^{2}(M)} + \norm{Q(\cdot,u_{n_{j}})}_{L^{2}(M)} \leq M_0$.
Therefore 
\begin{equation*}
  \norm{\mc{A}(u_{n_{i}})-\mc{A}(u_{n_{j}})  }_{H^{s}(M)}^{2}\leq C_{I}(C_{r}M_{0})^{2\theta}\norm{\mc{A}(u_{n_{i}})-\mc{A}(u_{n_{j}})}_{L^{2}(M)}^{2(1-\theta)}.  
\end{equation*}
In fact $\mc{A}(u_{n_{k}})$ is a Cauchy sequence in $H^{s}(M)$. Thus $\mc{A}: \mc{K}_{\varepsilon} \to \mc{K}_{\varepsilon}$ is a compact map.
Now we prove continuity of $\mathcal{A}$. Let $(u_{n}) \subset \mathcal{K}_{\varepsilon}$ with $u_{n} \to u$ in $H^{s}(M)$, using the linearity of the system $L_{c}u=Q(\cdot, u)$. Denote $V(\cdot,u)=U(\cdot,u)+u$, taking $\delta$ sufficiently small, we have that
\begin{align*}
 \norm{\mathcal{A}(u_{n}) - \mathcal{A}(u) }^{2}_{H^{s}(M)} &= \norm{\widetilde{u}_{n} - \widetilde{u} }_{H^{s}(M)}^{2} \leq \widetilde{K}_{r,s}^{2}\norm{\widetilde{u}_{n} - \widetilde{u}}_{\widetilde{H}_{f}(M)}^{2} = \widetilde{K}_{r,s}^{2}\norm{Q(\cdot,u_{n}) - Q(\cdot,u)}^{2}_{L^{2}(M)}\\
 &= \widetilde{K}_{r,s}^{2}\norm{U(x,u_{n})+u_{n} - U(x,u) - u + \delta \psi L(u_{n}-u)}_{L^{2}(M)}^{2} \\
 &\leq 2\widetilde{K}_{r,s}^{2}\norm{U(x,u_{n})+u_{n} - U(x,u) - u  }_{L^{2}(M)}^{2} + 2\widetilde{K}_{r,s}^{2}\norm{\delta \psi L(u_{n}-u)}_{L^{2}(M)}^{2}\\
 &\leq 2\widetilde{K}_{r,s}^{2}\norm{V(x,u_{n}) - V(x,u)}_{L^{2}(M)}^{2}\\
 &+ 2\widetilde{K}_{r,s}^{2}C_{s,0}^{2}(\delta\norm{\psi}_{L^{\infty}(M)})^{2}\norm{L}^{2}_{\mathscr{L}(L^{2}(M))}\norm{u_{n}-u}_{H^{s}(M)}^{2}.
\end{align*}
Since the operator $L$ is bounded and $\delta>0$ is a fixed sufficiently small number, as $u_{n} \to u$ in $H^{s}(M)$, it follows that for an arbitrary $\eta>0$, there exists $n_{0} \in \mb{N}$, such that for every $n\geq n_{0}$, 
$2\widetilde{K}_{r,s}^{2}C_{s,0}^{2}(\delta\norm{\psi}_{L^{\infty}})^{2}\norm{L}^{2}_{\mathscr{L}(L^{2}(M))}\norm{u_{n}-u}_{H^{s}(M)}^{2}\leq \frac{\eta}{2}$. Therefore,
\begin{align*}
 \norm{\mathcal{A}(u_{n}) - \mathcal{A}(u) }^{2}_{H^{s}(M)}
 %\leq \norm{V(x,u_{n}) - V(x,u)}_{L^{2}(M)} + \tfrac{\varepsilon}{2}\\
&\leq 2\widetilde{K}_{r,s}^{2}\int_{M}\left| \int_{0}^{1}\frac{d}{dt}V(x,tu_{n} + (1-t)u )dt  \right|^{2}dV_{g} + \tfrac{\eta}{2}\\
 &= 2\widetilde{K}_{r,s}^{2}\int_{M}|u_{n}(x)-u(x)|^{2}\left|  \int_{0}^{1}\frac{\partial}{\partial y}V(x,tu_{n} + (1-t)u )dt    \right|^{2}dV_{g} + \tfrac{\eta}{2}
\end{align*}
We define the function $h_{x}: [0,1] \to \mathbb{R}, h_{x}(t)=\frac{\partial}{\partial y}V(x, tu_{n}(x) 
   +  (1-t)u(x))$ with $x \in M$, since $U(\cdot, u)$ is $C^{1}$ in the second variable, $h_{x}(t)$ is a well defined Carathéodory function(i.e., measurable with respect to $x$ and continuous with respect to $t$). Note that
\begin{equation*}
    \Big|\int_{0}^{1}h_{x}(t)dt \Big| \leq \int_{0}^{1} |h_{x}(t)|dt \leq \int_{0}^{1}\left( \sup_{t \in [0,1]}|h_{x}(t)|  \right)dt = \sup_{t \in [0,1]}|h_{x}(t)|
\end{equation*}
Since $h_{x}(t)$ is continuous in a compact set, the maximum is attained for each $x \in M$, i.e., for each fixed $x_{0} \in M$,
there exists $t_{0} \in [0,1]$ such that $h_{x_{0}}$ attains the supremum in $t_{0}$. 
In fact there exists a Borel measurable function $t(\cdot): M \to [0,1]$ by \cite[t.18.13]{AliprantisBorder} such that for each $x^{\prime} \in M$, $t(x^{\prime})$ 
satisfies $\sup_{t \in [0,1]}|h_{x^{\prime}}(t)|=|h_{x^{\prime}}(t(x^{\prime}))|$. Therefore using (\ref{desigualdad numero 2 del caso pertubado}), for constants $C_{1}, \widetilde{C}> 0$, and that $0 \leq t(x) \leq 1$, we get the following computations
{\fontsize{10.5}{14}\selectfont
\begin{align*}
 \norm{\mathcal{A}(u_{n}) - \mathcal{A}(u)}^{2}_{H^{s}} 
&\leq 2\widetilde{K}_{r,s}^{2} \int_{M}|u_{n}(x) - u(x)|^{2}\left\{\sup_{t \in [0,1]}\left| \frac{\partial}{\partial y}V(x,tu_{n}(x) + (1-t)u(x))\right|     \right\}^{2}dV_{g} +\tfrac{\eta}{2}   \\
&\leq 2\widetilde{K}_{r,s}^{2}\int_{M}|u_{n}(x) - u(x)|^{2}\left\{\left| \frac{\partial}{\partial y}V(x,t(x)u_{n}(x) + (1-t(x))u(x))\right|     \right\}^{2}dV_{g} +\tfrac{\eta}{2} \\ 
&\leq C_{1}\int_{M}|u_{n}(x) - u(x)|^{2} \left( \left| 1 + |\chi(x)| + |t(x)u_{n}(x) + (1-t(x))u(x)|^{\beta} \right| \right)^{2}dV_{g} +\tfrac{\eta}{2}\\
  &\leq \widetilde{C}\int_{M}|u_{n}(x) - u(x)|^{2} \left( 1 +  |\chi(x)|^{2} + |t(x)u_{n}(x)|^{2\beta} + |(1-t(x))u(x)|^{2\beta} \right)dV_{g} +\tfrac{\eta}{2} \\ 
  &\leq \widetilde{C}\int_{M}|u_{n}(x) - u(x)|^{2}\left(  1 + |\chi(x)|^{2} + |u_{n}(x)|^{2 \beta} + |u(x)|^{2 \beta}  \right)dV_{g} +\tfrac{\eta}{2}
\end{align*}}
\justify
By embedding (\ref{embedding caso perturbado 1}), we have $\|u_{n}\|_{L^{\infty}(M)}^{2 \beta} \leq K_{1}$ and $\|u\|_{L^{\infty}(M) }^{2\beta} \leq K_{2}$ for constants $K_{1},K_{2}>0$. Taking $K=2\max\{K_{1},K_{2}\}$, there exists $n_{1} \in \mb{N}$, such that $\widetilde{C}C_{s,0}^{2}((1+K + \norm{\chi }^{2}_{L^{\infty}})\norm{u_{n} - u}_{H^{s}}^{2}) \leq \frac{\eta}{2}$ for every $n \geq n_{1}$. Then taking $n_{2}=\max(n_{0},n_{1})$ we have that for all $n \geq n_{2},$
\begin{align*}
 \norm{\mathcal{A}(u_{n}) - \mathcal{A}(u) }^{2}_{H^{s}(M)} &\leq \widetilde{C} \int_{M}|u_{n}(x) - u(x)|^{2}\left(  1 + |\chi(x)|^{2} + \|u_{n}\|_{L^{\infty}(M)}^{2 \beta} + \|u\|_{L^{\infty}(M) }^{2 \beta}  \right)dV_{g} +\tfrac{\eta}{2} \\
 &\leq \widetilde{C}\left( \left(1+K\right)\norm{u_{n} - u}_{L^{2}(M)}^{2}  + \int_{M}|u_{n}(x) - u(x)|^{2}|\chi(x)|^{2}dV_{g}\right) +\tfrac{\eta}{2} \\
 & \leq \widetilde{C}C_{s,0}^{2}\left( \left(1+K + \norm{\chi }^{2}_{L^{\infty}(M)}\right)\norm{u_{n} - u}_{H^{s}(M)}^{2} \right) +\tfrac{\eta}{2} \leq \eta.
\end{align*}
Hence $\widetilde{u}_n \to \widetilde{u}$ in $H^{s}(M)$, so $\mathcal{A}$ is continuous. By Schauder's fixed point theorem there exists a fixed point $\phi \in \mathcal{K}_{\varepsilon}$ such that $\phi = \mathcal{A}(\phi) \in \widetilde{H}_{f}(M,g)$, which is a solution of (\ref{sistema caso perturbado Riemaniano}).
\end{proof}

\vspace{0.12cm}

\section{Nonlocal Equations on a Cartan-Hadamard Manifolds}\label{seccion de Cartan-Hadamard equations}

We say that $(M,g)$ is a Cartan–Hadamard
$n$-manifold if it is complete, oriented, simply connected and has non-positive sectional curvature $K$. 
Fix $p \in M$. For any $r \geq 0$, we define $K_{p}(r)$ as
\begin{equation}\label{curvatura secional suprema donelly}
K_{p}(r) = \sup_{x \in M} \left\{ K(x, \pi) \mid d_{g}(p,x) \geq r, \, \pi = \text{span}\{u,v\} \subset T_x M \text{ is a 2-plane} \right\}
\end{equation}
Suppose $K_{p}(r) \to -\infty$ as $r \to \infty$. 
Using the Theorem \cite[t.1.1]{Donnelly} and \cite[p.360]{McKean}, the 
spectrum of the Laplace–Beltrami operator $\Delta_{g}$ under these conditions is discrete and positive and $0< \lambda_{1}<\lambda_{2} \leq \lambda_{3} \leq \ldots$, with finite multiplicity and 
accumulation at infinity, i.e., $|\lambda_{j}| \to \infty$. \\
Therefore, by \cite[t.5.7]{Schmudgen}, there exists an orthonormal basis of eigenfunctions
$\{ \omega_{n} \mid n \in \mathbb{N}\}$ of $L^{2}(M,g)$ such that
$\Delta_{g}\omega_{n}=\lambda_{n} \omega_{n}$ for $n \in \mathbb{N}$. 
Under these conditions the Sobolev space $H^{s,p}(M)$ with $s \in \mathbb{R}_{>0}$ and $1 \leq p < \infty$ is defined in (\cite[p.689]{Banica}) as 
\begin{equation}\label{definicion de espacio de sobolev fraccionario general}
    H^{s,p}(M,g)=\{u \in L^{p}(M) \mid \exists h \in L^{p}(M), u=(I+\Delta_{g})^{-s/2}h \}, 
\end{equation}
with norm 
\begin{equation}
  \|u\|_{H^{s,p}(M)}=\| h \|_{L^p(M)}.  
\end{equation}
For the spectral theorem for unbounded self-adjoint operators \cite[t.5.7]{Schmudgen}, there exists a unique spectral measure $E_{\Delta_{g}}: \mathcal{B}(\mathbb{R}) \to \mathcal{L}(L^{2}(M,g))$ such that
\begin{equation}
\Delta_{g}=\int_{\sigma(\Delta_{g})}\lambda dE_{\Delta_{g}}(\lambda).
\end{equation}
Using functional calculus and the function $g_{s}(\lambda)=(1+\lambda)^{-s/2}$ with $s \in \mathbb{R}_{>0}$, which is a Borel-measurable bounded function, and the fact that $\sigma(\Delta_{g})$ is discrete, we have 
\begin{equation}
(I+\Delta_{g})^{-s/2}=g_{s}(\Delta_{g})= \int_{\sigma(\Delta_{g})}(1+\lambda)^{-s/2} dE_{\Delta_{g}}(\lambda)=\sum_{j=1}^{\infty}(1+\lambda_{j})^{-s/2}E_{\Delta_{g}}(\{\lambda_{j}\})
\end{equation}

\justify

The expression $E_{\Delta_{g}}(\{\lambda_{j}\})$ is the orthogonal projection onto the finite dimensional eigenspace $E_{\lambda_{j}}=\ker(\Delta_{g} - \lambda_{j}I)$ which is closed; let us denote $m_{j}=\dim E_{\lambda_{j}}$. The orthogonal projection can be expressed as $E_{\Delta_{g}}(\{\lambda_{j}\})(f)=\sum_{k=1}^{m_{j}}\langle  f , \omega_{j_{k}}  \rangle_{L^{2}(M)}\omega_{j_{k}}$. Indeed,
\begin{align}\label{definicion espectral del operador resolvente}
(I+\Delta_{g})^{-s/2}f &= \sum_{j=1}^{\infty}(1+\lambda_{j})^{-s/2}E_{\Delta_{g}}(\{\lambda_{j}\})=\sum_{j=1}^{\infty}(1+\lambda_{j})^{-s/2}\left( \sum_{k=1}^{m_{j}}\langle  f , \omega_{j_{k}}  \rangle_{L^{2}(M)}\omega_{j_{k}} \right)\nonumber \\ 
&= \sum_{j=1}^{\infty}\sum_{k=1}^{m_{j}}(1+\lambda_{j})^{-s/2}\langle  f , \omega_{j_{k}}  \rangle\omega_{j_{k}}=\sum_{j=1}^{\infty} (1 + \lambda_j)^{-s/2} \langle f, \omega_j \rangle \omega_j.
\end{align}
By definition \eqref{definicion de espacio de sobolev fraccionario general} for $p=2$, a function $f$ belongs to $H^{s}(M,g)$ if and only if it is in the range of the operator $(I+\Delta_{g})^{-s/2}$ which is an injective and bounded self-adjoint operator on $L^2(M)$. Then its range coincides precisely with the domain of its unbounded inverse, $(I+\Delta_{g})^{s/2}$. Therefore, 
$$H^{s}(M,g)=\mc{D}((I+\Delta_{g})^{s/2}).$$
Furthermore, $f \in \mc{D}((I+\Delta_{g})^{s/2})$ if and only if\begin{equation}\int_{\sigma(\Delta_{g})} (1+\lambda)^{s} d\langle E_{\Delta_{g}}(\lambda)f, f \rangle_{L^2(M)} < \infty.\end{equation}
Since the spectrum is discrete, this spectral integral reduces to a series over the eigenvalues. Using the fact that $\|E_{\Delta_{g}}(\{\lambda_{j}\})(f)\|_{L^2(M)}^2 = \|f   \|^{2}_{L^{2}(M)}=\sum_{j}|\alpha_j|^2$ by \cite[l.4.4]{Schmudgen}, the integrability condition is exactly $\sum_{j=1}^{\infty} (1+\lambda_{j})^{s} \vert{}\alpha_{j}\vert{}^{2} < \infty$. Hence, an equivalent way to express the fractional Sobolev space is 
\begin{align}\label{espacio cartan hadamard, en funcion de f}
H^s(M,g) &=  \left\{ f=\sum_{j=1}^{\infty}\alpha_{j}\omega_{j} ~\Big|~ \sum_{j=1}^{\infty}|\alpha_{j}|^{2}(1+\lambda_{j} )^{s} < \infty    \right\}.
\end{align}

\justify

Let $J \subset \mathbb{N}$ be the index set such that $0 < \lambda_{j} < 1$ for $j \in J$, and $\lambda_{i} \geq 1$ for all $i \in (\mathbb{N} \setminus J)$. Suppose that the eigenfunctions $\{\omega_{j}\}_{j \in \mathbb{N}}$ satisfy the following condition for some $s \in \mathbb{R}_{>0}$,
\begin{equation}\label{condicion sombrero}
\mc{S}_{\ell}=\left(\sum_{j \in J}(1+\lambda_{j}^{2})^{-s}\norm{\nabla^{\ell}\omega_{j}}_{L^{\infty}(M)}^{2}    +   \sum_{j \in (\mathbb{N} \setminus J)}(1+\lambda_{j})^{-s}\norm{\nabla^{\ell}\omega_{j}}_{L^{\infty}(M)}^{2}  \right)< \infty
\end{equation}
where $\ell=0,\ldots,m$ with $m \in \mathbb{N}$ fixed.
Under the hypotheses of this section, the eigenfunctions $\omega_{j}$ are smooth 
functions, since $\Delta_{g}$ is an 
elliptic operator \cite[c.7.3]{Grigoryan}. Therefore $\nabla^{\ell}\omega_{j}$ with $j \in \mathbb{N}$ 
and $\ell \in \{0,\dots,m\}$ is a 
smooth $\ell$-covariant tensor field, i.e., $\nabla^{\ell}\omega_{j} \in 
\Gamma(T^{(0,\ell)}TM)$.\\ The notation $\|\nabla^{\ell}\omega\|_{L^{\infty}(M)}$ is the 
$\|\cdot\|_{L^{\infty}(M)}$-norm of the continuous
function $\phi: M \to \mathbb{R}$, determined by $p \overset{\phi}{\mapsto} |(\nabla^{\ell}\omega)(p)|_{g}$ for $p \in M$. That is
\begin{equation}
    \norm{\nabla^{\ell}\omega}_{L^{\infty}(M)}=\text{ess}\sup_{p \in M}|(\nabla^{\ell}\omega)(p)|_{g},
\end{equation}
the 
latter expression is 
$$|(\nabla^{\ell}\omega)(p)|_{g}^{2}=(g^{i_{1}j_{1}} \cdots g^{i_{\ell}j_{\ell}})(\nabla^{\ell}\omega)_{i_{1},...,i_{\ell}}(p)(\nabla^{\ell}\omega)_{j_{1},...,j_{\ell}}(p).$$
An example of a Cartan-Hadamard manifold satisfying the condition (\ref{condicion sombrero}) is below in Example (\ref{ejemplo de variedades Cartan-Hadamard}). 
%$$\norm{\nabla^{k}\omega}_{L^{\infty}(M)}=\text{ess} \sup_{x \in M}|\nabla^{k}\omega(x)|_{g}=\inf \{ C > 0 \mid |\nabla^{k}\omega(x)|_{g} \leq C, dV_{g}-a.e., x \in M \}$$
We denote by $C^{m}_{B}(M)$ the space of functions $u: M \to \mathbb{R}$ of class $C^{m}$\cite[p.63]{Hebey 2}
with the norm $\norm{\cdot}_{C^{m}}$ given by  
\begin{equation}
\norm{u}_{C^{m}}=\sum_{j=0}^{m}\sup_{p \in M} |(\nabla^{j}u)(p)|_{g} < \infty.
\end{equation}

\justify

\begin{definition}\label{definicion de fincion de Caratheodory}
   \cite[d.4.50]{AliprantisBorder}. Let $(S,\Sigma)$ be a measurable space, and let $X$ and $Y$ be topological spaces. A function $f: S \times X \to Y$ is a Carathéodory function if:\\
   $i)$ for each $x \in X$, the function $f^{x}=f(\cdot,x): S \to Y$ is $(\Sigma,\mc{B}(Y))$-measurable \footnote{Note that $\mc{B}(Y)$ is the usual notation to the Borel $\sigma$-algebra of $Y$.}; and \\
   $ii)$ for each $s \in S$, the function $f_{s}=f(s,\cdot): X \to Y$ is continuous.
\end{definition}

\vspace{0.12cm}

\begin{lemma}\cite[d.4.51]{AliprantisBorder}.
Let $(S,\Sigma)$ be a measurable space, $X$ a separable metrizable space, and $Y$ a metrizable space. Then every Carathéodory function $f: S \times X \to Y$is jointly measurable.
\end{lemma}

\begin{definition}\label{definicion del operador de Nemytskii}
    Let $X, Y, Z \neq \varnothing$ be non-empty sets. Let $Y^{X}, Z^{X}$ denote the sets of mappings from $X$ to $Y$ and $Z$ respectively. Let $h \colon X \times Y \to Z$.
Then the Nemytskii superposition operator $H \colon Y^{X} \to Z^{X}$ induced by $h$ is the map taking any map $\varphi \in Y^{X}$ to the map $H\varphi \in Z^{X}$ defined by
$$
(H\varphi)(x) = h(x, \varphi(x)) \in Z, \quad \mbox{for all } x \in X.
$$
The function $h$ is called the generator of the Nemytskii operator $H$.
\end{definition}

\begin{lemma}\label{lema Cartan-Hadamard}
Let $(M,g)$ be a Cartan-Hadamard $n$-manifold satisfying $K_{p}(r) \to - \infty$ as $r \to  \infty$ for every $p \in M$. Assume that condition \eqref{condicion sombrero}
is satisfied for some $s>0$ real and all derivative
orders $\ell=0,\dots,m$. Then $H^{s}(M,g) \hookrightarrow C^{m}_{B}(M)$.
\end{lemma}

\begin{proof}
Let $u \in H^{s}(M,g)$, then $u=\sum_{j=1}^{\infty}\alpha_{j}\omega_{j}$, where $\alpha_{j}=\langle u  , \omega_{j}   \rangle_{L^{2}(M)}$. Fix $\ell \in \{0,...,m\}$; our goal is to bound the expression $\nabla^{\ell}u=\nabla^{\ell}(\sum_{j=1}^{\infty}\alpha_{j}\omega_{j})=\sum_{j=1}^{\infty}\alpha_{j}\nabla^{\ell}\omega_{j}$, and prove that $u \in C^{m}(M)$.
Let us define the finite constant
\begin{equation}\label{eq: lema cartan-hadamard, constante importante}
   \Theta_{\ell}= 2\max\left\{\sum_{j \in J}(1+\lambda_{j}^{2})^{-s}\norm{\nabla^{\ell}\omega_{j}}^{2}_{L^{\infty}(M)}, \sum_{j \in (\mathbb{N}\setminus J)}(1+\lambda_{j})^{-s}\norm{\nabla^{\ell}\omega_{j}}^{2}_{L^{\infty}(M)}   \right\} > 0.
\end{equation}
We denote by $S_{N}:=\sum_{j=1}^{N}\alpha_{j}\omega_{j}$ the partial sums of $u$ and by $
\varepsilon_{N}^{2}:=\Big(\sum_{j>N}(1+\lambda_{j})^{s}|\alpha_{j}|^{2}\Big)
$ the tails of the convergent series given by $\|u\|_{H^{s}(M)}$.
These tails satisfy $\varepsilon_{N}\to 0$ as $N\to\infty$ and $\varepsilon_{0}=\|u\|_{H^{s}(M)}$. Fix an arbitrary $p \in M$, and let $N^{\prime}> N\geq 0$ be natural numbers. Using the Cauchy-Schwarz inequality in $\ell^{2}$, we obtain the following
\begin{align*}
\sum_{j=N+1}^{N'}\big|\alpha_{j}\nabla^{\ell}\omega_{j}(p)\big|_{g}
&=\sum_{j=N+1}^{N'}\Big[(1+\lambda_{j})^{s/2}|\alpha_{j}|\Big]
\Big[(1+\lambda_{j})^{-s/2}\,\big|\nabla^{\ell}\omega_{j}(p)\big|_{g}\Big]\\
&\leq\left(\sum_{j=N+1}^{N'}(1+\lambda_{j})^{s}|\alpha_{j}|^{2}\right)^{1/2}
\left(\sum_{j=N+1}^{N'}(1+\lambda_{j})^{-s}
\big|\nabla^{\ell}\omega_{j}(p)\big|_{g}^{2}\right)^{1/2}\\
&\leq\left(\sum_{j>N}(1+\lambda_{j})^{s}|\alpha_{j}|^{2}\right)^{1/2}
\left(\sum_{j=1}^{\infty}(1+\lambda_{j})^{-s}
\big\|\nabla^{\ell}\omega_{j}\big\|_{L^{\infty}(M)}^{2}\right)^{1/2}\\
&\leq \varepsilon_{N}\left( \sum_{j \in J}(1+\lambda_{j}^{2})^{-s}\norm{\nabla^{\ell}\omega_{j}}^{2}_{L^{\infty}} +  \sum_{j \in (\mathbb{N}\setminus J)}(1+\lambda_{j})^{-s}\norm{\nabla^{\ell}\omega_{j}}^{2}_{L^{\infty}}  \right)\leq \varepsilon_{N}\,\Theta_{\ell}.
\end{align*}
Thus, $(\nabla^{\ell}S_{N}(p))_{N}$ is a Cauchy sequence in the finite-dimensional normed space $(T^{(0,\ell)}T_{p}M, |\cdot|_{g})$, since for $N'> N$, we have 
$\big|\nabla^{\ell}S_{N'}(p)-\nabla^{\ell}S_{N}(p)\big|_{g}\leq
\sum_{j=N+1}^{N'}|\alpha_{j}\nabla^{\ell}\omega_{j}(p)|_{g}\leq
\Theta_{\ell}\varepsilon_{N}$.\\
This space is a finite-dimensional normed vector space, it is a Hilbert space \cite[c.1.1]{HH}; hence, there exists $T_{\ell}(p) \in T^{(0,\ell)}T_{p}M$ such that the series $(\nabla^{\ell}S_{N}(p))_{N}$ converges pointwise to some $T_{\ell}(p)$, i.e., 
$T_{\ell}(p)=\lim_{N \to  \infty}\nabla^{\ell}S_{N}(p)=\nabla^{\ell}u(p)$, in fact, we have uniform convergence, that is
\begin{equation}\label{eq:lema cartan 1}
    \sup_{p\in M}\big|T_{\ell}(p)-\nabla^{\ell}S_{N}(p)\big|_{g}
\leq\Theta_{\ell}\,\varepsilon_{N}\xrightarrow[N\to\infty]{}0.
\end{equation}
Then $T_{\ell}$ is a continuos section in $T^{(0,\ell)}TM$.
%T_{\ell}:  M  \to T^{(0,\ell)}TM 
Taking $N=0$, we obtain
\begin{equation}\label{eq: Lema cartan, ecuacion de desigualdades}
    \big|T_{\ell}(p)\big|_{g}\leq
\sum_{j=1}^{\infty}\big|\alpha_{j}\nabla^{\ell}\omega_{j}(p)\big|_{g}
\leq\Theta_{\ell}\,\|u\|_{H^{s}(M)}, \quad \forall p \in M.
\end{equation}
Let us prove that the component functions $(S_{N})_{i_{1},...,i_{\ell}}$ of the $(0,\ell)$-tensors $\nabla^{\ell}S_{N}$ converge uniformly to the component functions $(T_{\ell})_{i_{1},...,i_{\ell}}$ of the $(0,\ell)$-tensor $T_{\ell}$. 
Let $(U,(x^{i}))$ be a chart of $M$, and consider a compact subset $D \subset U$. 
For $z \in U$ and $T \in T^{(0,\ell)}T_{z}M$, let us denote the Euclidean norm of the component functions in the chart by $\|T\|_{e}:=\big(\sum_{i_{1},\dots,i_{\ell}}T_{i_{1}\cdots i_{\ell}}^{2}\big)^{1/2}$. Let us define the unit sphere of $(0,\ell)$-tensors on $D$ as
\begin{equation}
    %\mc{S}_{D}^{(0,\ell)}(T_{z}M)
    \mc{S}=\{(z,T) \mid z \in D, T \in T^{(0,\ell)}(T_{z}M), \norm{T}_{e}=1\},
\end{equation}
which is compact, since it is homeomorphic to $D \times \mathbb{S}^{n^{\ell}-1}$ via a local trivialization of the form $\psi:\ \mathcal{S}\rightarrow D\times\mathbb{S}^{n^{\ell}-1}, \psi(z,T)=(z,(T_{i_1\cdots i_\ell})_{(i_1,\dots,i_\ell)}),$ 
% La expresión $(T_{i_1\cdots i_\ell})_{(i_1,\dots,i_\ell)}$ representa el multivector, en donde cada i_{k} tiene n valores posibles, lo que implica que al final ladimensio de este vector sera n^{\ell}.
where the expression $(T_{i_1\cdots i_\ell})_{(i_1,\dots,i_\ell)}$ represents the vector of components $T_{i_{1} \cdots i_{\ell}}$;
furthermore, the map $(z,T)\mapsto|T|_{g_{z}}$ is continuous and strictly positive on $\mc{S}$, therefore it reaches its extremes on $\mc{S}$.
Suppose $c_{D}=\min_{\mc{S}}|T|_{g_{z}}>0$, then for $T \in \mc{S}$ we have $
c_{D}\|T\|_{e}\leq |T|_{g_{z}}$, $\forall z\in D, T \in T^{(0,\ell)}T_{z}M.$ Combining this with \eqref{eq:lema cartan 1}, we obtain
\begin{equation}\label{eq:lema cartan, convergencia de componentes ejemplo}
    %\sup_{z\in D}\big\|T_{\ell}(z)-\nabla^{\ell}S_{N}(z)\big\|_{e} = 
    \sup_{z\in D} \left( \sum_{i_{1},\dots,i_{\ell}} | (T_{\ell})_{i_{1}\cdots i_{\ell}}(z) - (\nabla^{\ell}S_{N})_{i_{1}\cdots i_{\ell}}(z) |^{2} \right)^{\frac{1}{2}} 
    \leq c_{D}^{-1}\,\Theta_{\ell}\,\varepsilon_{N}\xrightarrow[N\to\infty]{}0.
\end{equation}
Therefore, each component of $\nabla^{\ell}S_{N}$ converges uniformly to the component of $T_{\ell}$ on $D$. Since
$\nabla^{\ell}S_{N}$ is smooth and 
the uniform limits of continuous functions are continuous, the components of $T_{\ell}$ are continuous on $D$, and thus on $U$.\\
%because every point in $U$ admits a compact neighborhood contained within $U$.
Let us prove that $u=T_{0}  \in C^{m}(M)$, where, by the preceding points, $u$ is a continuous function such that $S_{N} \to u$ uniformly on $M$. To this end, we shall work locally: given a chart $(U,(x^{i}))$, we will prove this claim on an open coordinate ball $B$ such that $\overline{B} \subset U$. 
Note that, by using the proposition \cite[p.4.14]{Lee2} for a $(0,r)$-tensor $T$ with $0\leq r \leq m$, we can obtain the relation
\begin{equation}\label{eq:relacion de componentes covariantes}
(\nabla T)_{i_{1}\cdots i_{r};i}
=\partial_{i}T_{i_{1}\cdots i_{r}}
-\sum_{a=1}^{r}\Gamma^{b}_{i\,i_{a}}\,
T_{i_{1}\cdots i_{a-1}\,b\,i_{a+1}\cdots i_{r}},
\end{equation}
where $\Gamma^{b}_{i\,i_{a}} \in C^{\infty}(U)$. In particular, $\nabla^{\ell}w$ is defined with continuous components for every $w \in C^{\ell}(U)$. 
Using \eqref{eq:relacion de componentes covariantes} and the formula in \cite[pg.6]{Hebey 2}, we can derive an expression relating partial derivatives to covariant derivatives; then for $\ell \in \{0,...,m \}$ we have the following formula
\begin{equation}\label{eq:derivadas covariantes parciales}
\partial_{i_{1}}\partial_{i_{2}}\cdots\partial_{i_{\ell}}w
=(\nabla^{\ell}w)_{i_{1}\cdots i_{\ell}}
+\sum_{r=1}^{\ell-1}
\Lambda^{(r)\,j_{1}\cdots j_{r}}_{i_{1}\cdots i_{\ell}}\,
(\nabla^{r}w)_{j_{1}\cdots j_{r}},
\end{equation}
where the coefficients $\Lambda^{(r)\,j_{1}\cdots j_{r}}_{i_{1}\cdots i_{\ell}}\in C^{\infty}(U)$ are the coefficients $\Lambda^{(r)\,j_{1}\cdots j_{r}}_{i_{1}\cdots i_{\ell}}$ are polynomials in the Christoffel symbols and their derivatives up to order $\ell - 1$.\\
Let us apply formula \eqref{eq:derivadas covariantes parciales} to $w=S_{N}$. 
Using equation \eqref{eq:lema cartan, convergencia de componentes ejemplo} with $D=\overline{B}$, the sequence of component functions $\big((\nabla^{\ell}S_{N})_{i_{1}\cdots i_{\ell}}\big)_{N}$ converges uniformly to $(T_{\ell})_{i_{1},...,i_{\ell}}$ on $\overline{B}$ for $0\leq \ell \leq m$.
Therefore, we have the following calculations
\begin{align}
   \lim_{N \to \infty} (\partial_{i_{1}}\partial_{i_{2}} \cdots \partial_{i_{\ell}}S_{N})&=\lim_{N \to \infty}(\nabla^{\ell}S_{N})_{i_{1}\cdots i_{\ell}} +  \sum_{r=1}^{\ell-1} \Lambda^{(r), j_1 \dots j_r}_{i_1 \dots i_\ell} \lim_{N \to \infty}(\nabla^r S_{N})_{j_1 \dots j_r} \nonumber \\
   &= (T_{\ell})_{i_{1} \cdots i_{\ell}} + \sum_{r=1}^{\ell-1} \Lambda^{(r), j_1 \dots j_r}_{i_1 \dots i_\ell} (T_{r})_{j_1 \dots j_r}=(\partial_{i_{1}}\partial_{i_{2}} \cdots \partial_{i_{\ell}} T_{0})=\partial_{i_{1}}\partial_{i_{2}} \cdots \partial_{i_{\ell}}u \nonumber
\end{align}
Since $(\partial_{i_{1}}\partial_{i_{2}} \cdots \partial_{i_{\ell}}S_{N})_{N}$ converges uniformly on $\overline{B}$, it follows that $\partial_{i_{1}}\partial_{i_{2}} \cdots \partial_{i_{\ell}}u \in C^{0}(\overline{B})$; as the index $(i_{1},i_{2},...,i_{\ell})$ is arbitrary, we conclude that $u \in C^{m}(B)$. 
Since each chart and open coordinate ball were arbitrary, it follows that $u \in C^{m}(M)$. 
Therefore, using equation \eqref{eq: Lema cartan, ecuacion de desigualdades} and summing over $\ell \in \{0,...,m\},$ we have
\begin{equation}\label{desigualdad funciones continuas}
\|u\|_{C^{m}}=\sum_{\ell=0}^{m}\sup_{p\in M}\big|\nabla^{\ell}u(p)\big|_{g}
\leq\left(\sum_{\ell=0}^{m}\Theta_{\ell}\right)\,\|u\|_{H^{s}(M)}<\infty.
\end{equation}
Therefore, $u\in C^{m}_{B}(M)$, and there exists a constant $K>0$ such that $\norm{u}_{C^{m}(M)} \leq K\norm{u}_{H^{s}(M)}$.
\end{proof}

We are now in conditions to state our existence result for regular solutions of the nonlocal equation $
f(\Delta_g)\psi - U(x,\psi) = 0$
on a complete, non-compact Cartan-Hadamard manifold.

\begin{theorem}\label{teorema finnal para el caso no compacto}
      Let $(M,g)$ be a Cartan–Hadamard $n$-manifold 
      satisfying $K_{p}(r) \to -\infty$ as $r \to \infty$ for any $p \in M$, and assume condition \eqref{condicion sombrero} in the case $\ell=0$ for some real $s>0$. 
      Suppose that $f$ satisfies \eqref{codicion de elipticidad de una fucnion real valuada para Riemannian manifolds},
      that $U: M \times \mathbb{R} \to \mathbb{R}$ is a Carathéodory non-linearity which is a $C^{1}$ in the second variable, and that there exist 
      constants $\alpha > 1$, $\beta \geq 1$, $C > 0$, and a 
      function $h \in L^2(M,g)$ such that the following inequalities hold
\begin{equation}\label{desigualdad numero 1 del caso completo}
|U(x, y) + y| \leq C \left(|h(x)| + |y|^\alpha \right), \quad 
    \left| \frac{\partial}{\partial y} \left( U(x, y) + y \right) \right| \leq C ( 1 + |y|^\beta ).
\end{equation}
Then, there exists $\varepsilon \in (0,1]$ and $0 < \rho_{\varepsilon} < 1$ such that whenever $\norm{h}_{L^2(M)} < \rho_{\varepsilon}$, there exists a solution $\psi \in \widetilde{H}_f(M,g)$ to the non-linear equation $f(\Delta_g)\psi - U(x, \psi) = 0.$
\end{theorem}

\begin{proof}   
Using the operator $L_{c}=f(\Delta_{g}) + Id$ defined in (\ref{operador caso Riemanniano}), solving the nonlinear equation $f(\Delta_g)\phi - U(x, \phi) = 0$ is equivalent to solving the linear problem
\begin{equation}\label{sistema lineal caso cartan-hadamard}
    L_{c}u=R(\cdot, u)
\end{equation}
where $R(\cdot,u)=u + U(\cdot, u)$.
Let us consider the following remarks.
First, the ellipticity condition \eqref{codicion de elipticidad de una fucnion real valuada para Riemannian manifolds} implies the existence of a $k_{0} \in \mb{Z}_{>0}$ such that $k_{0}>s$ and $C_{0} \in (0,1]$ such that
\begin{equation}\label{desigualdad elitica original implica desigualdad elitica caso cartan-hadamard}
 1+f(\lambda_{j}) \geq C_{0}(1+\lambda_{j})^{k_{0}/2}, \quad j \in \mb{N}.   
\end{equation}
Indeed, by $\eqref{codicion de elipticidad de una fucnion real valuada para Riemannian manifolds}$ we have that $1+f(\lambda_{j}) \geq C_{k}(1+\lambda_{j}^{2})^{k/2} \geq C_{k}2^{-k/2}(1+\lambda_{j})^{k}$, then, using that $0\leq (1-\lambda)^{2}$ for all $\lambda\geq 0$, that is, $1+\lambda^{2} \geq \frac{1}{2}(1+\lambda)^{2}$. 
Hence \eqref{desigualdad elitica original implica desigualdad elitica caso cartan-hadamard} is satisfied by taking $k_{0}=2k>s$ and $C_{0}=\min\{1,\,C_{k}2^{-k/2}\}$.
In particular, since the ellipticity condition \eqref{codicion de elipticidad de una fucnion real valuada para Riemannian manifolds} is assumed for every integer $k$, then \eqref{desigualdad elitica original implica desigualdad elitica caso cartan-hadamard} holds for every fixed $k_{0}>s$.
 Second, the operator $L_{c}$ defined in \eqref{operador caso Riemanniano} by  $L_{c}=f(\Delta_{g})+Id$ acts as $L_{c}u=\sum_{j}(1+f(\lambda_{j}))u_{j}\omega_{j}$ for
$u\in\widetilde{H}_{f}(M,g)$. This series converges in $L^{2}(M,g)$ by the definition of $\widetilde{H}_{f}(M,g)$. 
By Parseval's identity,
$$\|L_{c}u\|_{L^{2}(M)}^{2}
=\sum_{j=1}^{\infty}\big(1+f(\lambda_{j})\big)^{2}|u_{j}|^{2}
=\|u\|_{\widetilde{H}_{f}(M)}^{2},$$
Consequently, $L_{c}:\widetilde{H}_{f}(M,g)\to L^{2}(M,g)$ is a linear isometry. Moreover, it is bijective. Indeed, $L_{c}u=0$ implies that 
$(1+f(\lambda_{j}))u_{j}=0$ for all $j \in \mb{N}$, but since $1+f(\lambda_{j})\geq1>0$, we have $u_{j}=0$ for all $j \in \mb{N}$, which implies that $u=0$; hence is injective. Given $g \in L^{2}(M,g)$, consider  $\widetilde{u}=\sum_{j}(1+f(\lambda_{j}))^{-1}g_{j}\omega_{j}$; this converges in $L^{2}(M,g)$(its coefficients are dominated by $|g_{j}|$).\\
Note that $$\|\widetilde{u}\|_{\widetilde{H}_{f}(M)}^{2}=\sum_{j}|g_{j}|^{2}
=\|g\|_{L^{2}(M)}^{2}<\infty$$
and $L_{c}\widetilde{u}=\sum_{j}g_{j}\omega_{j}
=g$; therefore $L_{c}$ is surjective, hence it is an isometric isomorphism. This establishes an analogue, in the non-compact case, of the linear existence Theorem \ref{sistema lineal caso riemanniano}; that is, \emph{for every $g\in L^{2}(M,g)$ the linear problem
$L_{c}\widetilde{u}=g$ has a unique solution
$\widetilde{u}=L_{c}^{-1}g\in\widetilde{H}_{f}(M,g)$, and
$\|\widetilde{u}\|_{\widetilde{H}_{f}(M)}=\|g\|_{L^{2}(M)}$.}\label{problema lineal caso no compacto}
\begin{claim}\label{claim teorema caso cartan hadamard}
Let $(M,g)$ be a Cartan–Hadamard $n$-manifold 
      satisfying $K_{p}(r) \to -\infty$ as $r \to \infty$ for any $p \in M$, and assume condition \eqref{condicion sombrero} for some real $s>0$. 
    Suppose that $f$ satisfies the elliptic condition \eqref{codicion de elipticidad de una fucnion real valuada para Riemannian manifolds}. Then the following embedding
\begin{equation}\label{embedding caso Cartan-hadamard}
        \widetilde{H}_{f}(M,g)\hookrightarrow H^{s}(M,g) \hookrightarrow L^{\infty}(M,g)
    \end{equation}
is compact. 
\end{claim}

\begin{proofofclaim}  
We prove the first part of the embedding. 
Set $d_{j}:=(1+\lambda_{j})^{s/2}(1+f(\lambda_{j}))^{-1}>0$. By
\eqref{desigualdad elitica original implica desigualdad elitica caso cartan-hadamard}, we have that $d_{j}\leq C_{0}^{-1}\,(1+\lambda_{j})^{(s-k_{0})/2}$ 
and since $k_{0}>s$, $1+\lambda_{j}\geq1$ and $\lambda_{j}\to\infty$, we obtain $d_{j}\leq C_{0}^{-1}$ and the $d_{j} \to 0$ as
$j\to\infty$. For $u\in\widetilde{H}_{f}(M,g)$,
\[
\|u\|_{H^{s}(M)}^{2}
=\sum_{j=1}^{\infty}(1+\lambda_{j})^{s}|u_{j}|^{2}
=\sum_{j=1}^{\infty}d_{j}^{2}\,\big(1+f(\lambda_{j})\big)^{2}|u_{j}|^{2}
\;\leq\;C_{0}^{-2}\,\|u\|_{\widetilde{H}_{f}(M)}^{2},
\]
so $\iota$ is well defined, linear, and
$\|\iota\|_{\mathscr{L}(\widetilde{H}_{f},H^{s})}\leq C_{0}^{-1}$.
We will use an approximation by finite-rank linear operators to prove that $\iota$ is a compact operator. For each $N \in \mb{N}$ let us define the operators $\iota_{N} : \widetilde{H}_{f}(M,g) \hookrightarrow H^{s}(M,g)$ as $\iota_{N}u:=\sum_{j=1}^{N}u_{j}\omega_{j}$.
Each of these is linear, and of finite rank, since its range lies in $\operatorname{span}\{\omega_{1},\dots,\omega_{N}\}$, and is bounded, by the same previous estimate restricted to $j \leq N$.
Moreover
\begin{align*}
\norm{\iota - \iota_{N}}^{2}_{\mathscr{L}(\widetilde{H}_{f}, H^{s})}&=\sup_{\substack{u \not\equiv 0,~u \in \widetilde{H}_{f}(M)}}\frac{\|(\iota-\iota_{N})u\|_{H^{s}(M)}^{2}}{\|u\|^{2}_{\widetilde{H}_{f}(M)}}=\frac{\sum_{j>N}d_{j}^{2}\,\big(1+f(\lambda_{j})\big)^{2}|u_{j}|^{2}}{\|u\|^{2}_{\widetilde{H}_{f}(M)}} \leq \left(\,\sup_{j>N}d_{j}\,\right)^{2}
\end{align*}
whence
$\|\iota-\iota_{N}\|_{\mathscr{L}(\widetilde{H}_{f},H^{s})}\leq
\sup_{j>N}d_{j}$. Since $d_{j} \to 0$ as $j \to \infty$, for a given $\varepsilon>0$, we can choose $N_{0} \in \mb{N}$ with $d_{j} < \varepsilon, \forall j>N_{0}$; then $\sup_{j>N}d_{j} \leq \varepsilon$ for every $N\geq N_{0}$. In fact $\sup_{j>N}d_{j} \to 0$ as $N \to \infty$.
Then by \cite[c.6.2]{Brezis}, the embedding $\widetilde{H}_{f}(M,g)\hookrightarrow H^{s}(M,g)$ is compact.\\
Following the previous procedure, we prove that the inclusion $\iota: H^{s}(M) \hookrightarrow L^{\infty}(M)$ is a well defined linear operator. 
Let $u \in H^{s}(M)$ be arbitrary, then there exists $\phi \in L^{2}(M)$ such that $u=(I+\Delta_{g})^{-s/2}\phi$. Let us define the mapping $\kappa_{\phi}: M \to \mathbb{R}$ determined by $x \overset{\kappa_{\phi}}{\longmapsto} (I+\Delta_{g})^{-s/2}\phi(x)$, using the constant \eqref{eq: lema cartan-hadamard, constante importante} and Parseval's identity, we seek the behavior of $|\kappa_\phi(x)|$. In fact
\begin{align*}
    | \kappa_{\phi}(x) | 
&\leq \sum_{j=1}^{\infty} |(1+\lambda_{j})^{-s/2}\omega_{j}(x)| \cdot |\langle  \phi,\omega_{j}\rangle_{L^{2}(M)}|\leq \left( \sum_{j=1}^{\infty}(1+\lambda_{j})^{-s}|\omega_{j}(x)|^{2}   \right)^{\frac{1}{2}} \left(\sum_{j=1}^{\infty}|\langle  \phi , \omega_{j}  \rangle|^{2}    \right)^{\frac{1}{2}}\\
&= \norm{\phi}_{L^{2}(M)}\left( \sum_{j \in J}(1+\lambda_{j})^{-s}|\omega_{j}(x)|^{2} + \sum_{j \in (\mathbb{N} \setminus J)}(1+\lambda_{j})^{-s}|\omega_{j}(x)|^{2}  \right)^{\frac{1}{2}} \\
&\leq \norm{\phi}_{L^{2}}\left(\sum_{j \in J}(1+\lambda_{j}^{2})^{-s}\norm{\omega_{j}}_{L^{\infty}(M)}^{2}    +   \sum_{j \in (\mathbb{N} \setminus J)}(1+\lambda_{j})^{-s}\norm{\omega_{j}}_{L^{\infty}(M)}^{2}\right)^{\frac{1}{2}} \leq  \Theta_{0} \|\phi\|_{L^2(M)}.
\end{align*}
Therefore, $\norm{\kappa_{\phi}}_{L^{\infty}(M)} \leq \Theta_{0} \|\phi\|_{L^2(M)} < \infty$; hence, $(I+\Delta_{g})^{-s/2}\phi \in L^{\infty}(M)$.
Indeed, since that the function $u \in H^{s}(M)$ was arbitrary, we have that $H^{s}(M) \subset L^{\infty}(M)$, which implies that $\iota: H^{s}(M) \hookrightarrow L^{\infty}(M)$ is a well-defined and continuous linear map by \eqref{desigualdad funciones continuas}. 
Let us define the operator $\mc{K}: \iota \circ (I+\Delta_{g})^{-s/2}: L^{2}(M) \to L^{\infty}(M)$. By the previous point, this operator is linear and continuous. Let us prove that it is compact; for this, we shall use the finite-rank operator approximation Theorem \cite[c.6.2]{Brezis}. For each $N \in \mathbb{N}$, we define $T_{N}: L^{2}(M) \to L^{\infty}(M)$ by
\begin{equation}
  T_N \phi = \sum_{1\leq j  \leq N} (1 + \lambda_j)^{-s/2} \langle \phi, \omega_j \rangle \omega_j.  
\end{equation}
Regarding the operator norm $\|\cdot\|_{\mathscr{L}(L^{2}(M), L^{\infty}(M))}$, usually denoted as $\|\cdot\|_{\mathscr{L}(L^{2}, L^{\infty})}$, we have the following.
Each $T_{N}$ has finite rank since it satisfies $\text{Im}(T_{N}) \subseteq \text{span}\{\omega_{1},...,\omega_{N}\}$ and is bounded because
\begin{align*}
|T_{N}\phi(x)| &\leq \sum_{1\leq j \leq N}|(1+\lambda_{j})^{-s/2}\omega_{j}(x)||\langle    \phi ,  \omega_{j} \rangle| \leq \Theta_{0}\norm{\phi}_{L^{2}M)} \implies \frac{|T_{N}\phi(x)|}{\|\phi\|_{L^{2}(M)}} \leq \Theta_{0}
\end{align*}
which implies that $\|T_{N}\|_{\mathscr{L}(L^{2},L^{\infty})} \leq \Theta_{0}$. In fact
\begin{align*}
    \|\mc{K} - T_N \|_{\mathscr{L}(L^{2}, L^{\infty})} &= \sup_{\substack{\phi \not\equiv 0\\ \phi \in L^{2}(M)}} \frac{\norm{\mc{K}\phi - T_{N}\phi}_{L^{\infty}(M) }}{ \norm{\phi}_{L^{2}(M) }}\leq \left(\sum_{j>N}(1+\lambda_{j})^{-s}\norm{\omega_{j}}_{L^{\infty}(M)}^{2}   \right)^{1/2}.
\end{align*}
By the condition \eqref{condicion sombrero}, the preceding series converges, which implies that their tail converges to zero, i.e., $\|\mc{K} - T_N \|_{\mathscr{L}(L^{2}, L^{\infty})} \to 0$. 
Since $\mathcal{K}$ is the uniform limit of finite-rank operators, $\mathcal{K}: L^2(M) \to L^\infty(M)$ is a compact operator.
Let $(u_k)$ be a bounded sequence in $H^s(M)$. There exists a sequence $(g_{k}) \subset L^{2}(M)$ such that $u_k = (I + \Delta_g)^{-s/2} g_k$, moreover the sequence $(g_{k})$ is bounded in $L^{2}(M)$ because $\|g_k\|_{L^2} = \|u_k\|_{H^s} \leq C$.
By compactness of $\mc{K}$ there exists  $(g_{k_i}) \subset (g_{k})$ in $L^{2}(M)$ such that $u_{k_{i}}=(I+\Delta_{g})^{-s/2}g_{k_{i}} \subset \mc{K}(g_{k})$. \\
In fact there exists a subsequence $(u_{k_{i}}) \subset (u_{k})$ that converges in $L^{\infty}(M)$. Hence the embedding $H^s(M) \hookrightarrow L^\infty(M)$ is compact for a real $s>0$ which satisfies the condition \eqref{condicion sombrero}.
\end{proofofclaim}
\justify
\vspace{-0.12cm}
Previously we define the function $R(\cdot,u)=u + U(\cdot, u)$, this function generated the Nemytskii operator \eqref{definicion del operador de Nemytskii} by  $\mc{R}: \mb{R}^{M} \to \mb{R}^{\mb{R}}$ defined for $u \in H^{s}(M,g)$ as $\mc{R}(u)(x):=R(x,u(x)), x \in M$. Now we prove $\mc{R}$ is measurable. 
In fact, since the nonlinearity $U$ is a Carathéodory function, it also follows that $R$ is a Carathéodory function, hence for each $y \in \mb{R}$, $R(\cdot,y)=y+U(\cdot,y)$ is measurable, and for each fixed $x \in M$, $R(x,\cdot) \in C^{1}(\mb{R})$.
Let us prove that $R(x,w(x))$ is measurable for any measurable function $w:M \to \mb{R}$; indeed, we know that $x \mapsto R(x,w(x))$ is a $(\mc{B}(M), \mc{B}(\mb{R}))$-measurable function, since $w$ is measurable, there exists a sequence $(w_{k})_{k \in \mb{N}}$ of nonnegative simple measurable functions on $\mc{B}(M)$ such that $w_{k} \to w$ pointwise on $M$ of the form $w_{k}=\sum_{1\leq i \leq p_{k}}c_{i}^{k}\chi_{A_{i}^{k}}$ with $(A_{i}^{k}) \in \mc{B}(M)$ pairwise disjoint for all $1\leq i \leq p_{k}$.
Therefore
\begin{align*}
    x \mapsto R(x,w_{k}(x))&=w_{k}(x) + U(x,w_{k}(x))=\sum_{1\leq i \leq p_{k}}c_{i}^{k}\chi_{A_{i}^{k}}(x) + U\left(x, \sum_{1\leq i \leq p_{k}}c_{i}^{k}\chi_{A_{i}^{k}}(x)\right)\\
    &=\sum_{1\leq i \leq p_{k}}c_{i}^{k}\chi_{A_{i}^{k}}(x) + U(x,c_{i}^{k})\chi_{A_{i}^{k}}(x)=\sum_{1 \leq i \leq p_{k}}R(x, c_{i}^{k})\chi_{A_{i}^{k}}(x)
\end{align*}
Since $R(x,c^{k}_{i})$ is measurable for any fixed $x \in M$, and $\chi_{A^{k}_{i}}$ is measurable for $A_{i}^{k} \in \mc{B}(M)$, the former expression is measurable. 
For any fixed $x \in M$, we have that $R(x,w_{k}(x)) \to R(x,w(x))$ pointwise by continuity. 
Since the limit of measurable functions is measurable, it follows that $R(x,w(x))$ is measurable for every measurable $w: M \to \mb{R}$. Therefore $\mc{R}(u)$ is measurable.\\
By Lemma \ref{lema Cartan-Hadamard}, if $u \in 
\widetilde{H}_{f}(M,g)$ then $u$ is a $C^{m}$-regular function on $M$.
Take $u \in H^{s}(M,g)$. Using inequality (\ref{desigualdad numero 1 del caso completo}) and the embedding (\ref{embedding caso Cartan-hadamard}), we show that $\mc{R}(u) \in L^{2}(M,g)$. \\
Indeed, for a constant $\tilde{C}>0$, we have
\begin{align*}
    \norm{\mc{R}(u)}_{L^{2}(M)}&= \left( \int_{M}|U(x,u(x)) + u(x)|^{2}dV_{g}\right)^{\frac{1}{2}} \leq \sqrt{C}\left(\int_{M} (|h(x)| + |u(x)|^{\alpha}  )^{2} dV_{g}\right)^{\frac{1}{2}}\\
    &\leq \sqrt{C}\left(\int_{M}(|h|^{2} + 2|h||u|^{\alpha} + |u|^{2 \alpha}    ) dV_{g}\right)^{\frac{1}{2}} \leq \tilde{C}\left( \int_{M}(|h(x)|^{2} + |u(x)|^{2\alpha}) dV_{g}  \right)^{\frac{1}{2}}
\end{align*}
 Note that $\| |u|^{\alpha}\|_{L^{2}(M)} = \| u\|^{\alpha}_{L^{2\alpha}(M)}$. 
 %and considering the measure space $(M,\mathcal{B}_{M},dV_g)$, where $\mc{B}_{M}$ is the Borel $\sigma$-algebra generated by the topology $\tau_{M}$ in $M$ (i.e., $\mathcal{B}_{M} = \sigma(\tau_{M})$). 
Using the generalized  H\"older interpolation inequality \cite[p.93]{Brezis}. We take an appropriate real $s>0$, such that the embedding \eqref{embedding caso Cartan-hadamard} is satisfied. 
Taking any real $r$ 
such that $2 \leq r \leq \infty$ and $\theta \in (0,1)$ with $1/r=\theta/2$, we have that $u \in L^{2}(M) \cap L^{\infty}(M)$.\\
Then the following inequality works
\begin{equation}\label{Cartan-hadamard: desigualdad de interpolacion}
    \| u \|_{L^{r}(M)} \leq \|u \|^{\theta}_{L^{2}(M)} \|u\|_{L^{\infty}(M)}^{1-\theta}.
\end{equation} 
Considering $r=2 \alpha$ with $\alpha>1$, with $\theta=1/\alpha \in (0,1)$ and inequality \eqref{Cartan-hadamard: desigualdad de interpolacion}; using the embeddings $H^{s}(M) \hookrightarrow L^{2}(M)$ and the claim \eqref{claim teorema caso cartan hadamard} for a suitable real $s>0$ which satisfies \eqref{condicion sombrero}, there exists a constant  $C_{I} > 0$ such that
\begin{equation}
\| u\|_{L^{2\alpha}(M)} \leq \|u \|^{1/\alpha}_{L^{2}(M)} \|u\|_{L^{\infty}(M)}^{1-1/\alpha} \leq C_{I}\|u \|^{1/\alpha}_{H^{s}(M)} \|u\|_{H^{s}(M)}^{1-1/\alpha}.
\end{equation}
Indeed $H^{s}(M) \hookrightarrow L^{2\alpha}(M)$ with $\alpha>1$, which implies $\| |u|^{\alpha}\|_{L^{2}(M)}=\|u\|^{\alpha}_{L^{2\alpha}(M)} \leq (C_{I})^{\alpha}\|u\|^{\alpha}_{H^{s}(M)}$.\\
Then, for some constant $K>0$, we have the following computations
\begin{align*}
    \norm{\mc{R}(u)}_{L^{2}(M)}
    &\leq K\left( \norm{h}_{L^{2}(M)}  + \norm{|u|^{\alpha}}_{L^{2}(M)}\right) \leq  K\left( \norm{h}_{L^{2}(M)}  + (C_{I})^{\alpha}\norm{u}^{\alpha}_{H^{s}(M)}\right) < \infty
\end{align*}
which implies that $\mc{R}(u) \in L^{2}(M)$. In particular $\mc{R}: H^{s}(M,g) \to L^{2}(M,g)$ is well defined. Define
\begin{equation}
\mathcal{G}_{\varepsilon}=\{u \in H^{s}(M) \mid \norm{u}_{H^{s}(M)} \leq \varepsilon \},
\end{equation}
where $\varepsilon>0$ is arbitrary. Note that $\mathcal{G}_{\varepsilon}$ is closed, bounded, convex and nonempty in $H^{s}(M)$.\\
Since $L_{c}$ is an isometric isomorphism, we can express the solution of the linear problem as $\widetilde{u}=L_{c}^{-1}g$, where the inverse operator is defined as $L_{c}^{-1}u=\sum_{j}(1+f(\lambda_{j}))^{-1}u_{j}\omega_{j}$.
Consider the following composition scheme
\begin{equation*}
    \mc{G}_{\varepsilon} \xrightarrow[]{\mc{R}|_{\mc{G}_{\varepsilon}}} L^{2}(M,g) \xrightarrow[]{L_{c}^{-1}} \widetilde{H}_{f}(M,g) \xrightarrow[]{\iota} H^{s}(M,g)
\end{equation*}
Let us define the following map
\begin{equation}\label{operador caso perturbado riemanniano 1}
    \mathcal{A}: \mathcal{G}_{\varepsilon} \to H^{s}(M), \quad \mathcal{A}(u)= \iota \circ L_{c}^{-1} \circ \mc{R}|_{\mc{G}_{\varepsilon}}(u) =\widetilde{u},
\end{equation}
where $\widetilde{u}$ is the unique solution of the linear problem $L_{c}\widetilde{u}=\mc{R}(u), \forall u \in H^{s}(M,g)$.
Note that $\mc{A}$ is well defined.
Let us prove that there exists $\varepsilon \in (0,1]$ such that $\mc{A}: \mc{G}_{\varepsilon} \to \mc{G}_{\varepsilon}$ is well defined.
Let $u \in \mc{G}_{\varepsilon}$; using the claim \eqref{claim teorema caso cartan hadamard}, we have that $\|\iota\|_{\mathscr{L}(\widetilde{H}_{f},H^{s})}\leq C_{0}^{-1}$.
Consider the values
 \begin{equation*}
    \varepsilon:= \min \left\{1, \left(\frac{C_{0}}{2K(C_{I})^{\alpha}}\right)^{\frac{1}{\alpha-1}} \right\} \in (0,1], \quad \rho_{\varepsilon}:=\frac{\varepsilon C_{0}}{2K}.
 \end{equation*}
using these values, we have the following computations
\begin{align*}\label{caso cartan-hadamard, operador acotado}
    \norm{\mathcal{A}(u)}_{H^{s}(M)} &=\norm{\iota \circ L_{c}^{-1} \circ \mc{R}|_{\mc{G}_{\varepsilon}}(u)}_{H^{s}(M)} \leq \|\iota\|_{\mathscr{L}(\widetilde{H}_{f},H^{s})}\norm{L_{c}^{-1} \circ \mc{R}|_{\mc{G}_{\varepsilon}}(u)}_{H^{s}(M)}\\
    &\leq C_{0}^{-1}\norm{L_{c}^{-1} \circ \mc{R}|_{\mc{G}_{\varepsilon}}(u)}_{\widetilde{H}_{f}(M)} = C_{0}^{-1}\norm{\mc{R}(u)}_{L^{2}(M)}\\
    &\leq C_{0}^{-1}K\left( \norm{h}_{L^{2}(M)}  + (C_{I})^{\alpha}\norm{u}^{\alpha}_{H^{s}(M)}\right) \leq C_{0}^{-1}K\norm{h}_{L^{2}(M)} + C_{0}^{-1}K(C_{I})^{\alpha}\varepsilon^{\alpha}\\
&=\underbrace{C_{0}^{-1}K\norm{h}_{L^{2}(M)}}_{\varepsilon/2} + \underbrace{C_{0}^{-1}K(C_{I})^{\alpha}\varepsilon^{\alpha-1}}_{\varepsilon/2}\varepsilon \leq \varepsilon
\end{align*}
where $\|h\|_{L^{2}(M)} < \rho_{\varepsilon}$. Hence $\mc{A}(\mc{G}_{\varepsilon}) \subseteq \mc{G}_{\varepsilon}$.
We now prove that $\mathcal{A}$ is continuous.
Let $(u_{n}) 
\subset \mathcal{G}_{\varepsilon}$ be bounded with $u_{n} \to u$ in $H^{s}(M)$. Using the linearity of the system $L_{c}u=\mc{R}(u)$ we have the following computations
\begin{align*}
 \norm{\mathcal{A}(u_{n}) - \mathcal{A}(u) }^{2}_{H^{s}(M)} &= \norm{\widetilde{u}_{n} - \widetilde{u} }_{H^{s}(M)}^{2} \leq \norm{\widetilde{u}_{n} - \widetilde{u}}_{\widetilde{H}_{f}(M)}^{2} = \norm{\mc{R}(u_{n}) - \mc{R}(u)}^{2}_{L^{2}(M)}\\
 &=\int_{M}|R(x,u_{n})  - R(x,u)|^{2}dV_{g} = \int_{M}\left| \int_{0}^{1}\frac{d}{dt}R(x,tu_{n} + (1-t)u )dt  \right|^{2}dV_{g} \\
 &= \int_{M}|u_{n}(x)-u(x)|^{2}\left|  \int_{0}^{1}\frac{\partial}{\partial y}R(x,tu_{n} + (1-t)u )dt    \right|^{2}dV_{g}
\end{align*}
% Notese que el operador Q, tendra esta forma: u \mapsto Q(\cdot, u(\cdot)), y es u operador de H^{s}(M) \to L^{2}(M).
We define the function $i_{x}: [0,1] \to \mathbb{R}, i_{x}(t)=\frac{\partial}{\partial y}R(x, tu_{n}(x) 
   +  (1-t)u(x))$ with $x \in M$, since $U(\cdot, u)$ is $C^{1}$ in the second variable, $i_{x}(t)$ is a well defined Carathéodory function(i.e., measurable with respect to $x$ and continuous with respect to $t$).
%\begin{definition}A function \(W:\Omega\times\mathbb{R}^{N}\to\mathbb{R}\cup\{+\infty\}\) is \emph{Carathéodory} if \(x\mapsto W(x,\xi)\) is Lebesgue-measurable for every \(\xi\in\mathbb{R}^{N}\) and \(\xi\mapsto W(x,\xi)\) is continuous for almost every \(x\in\Omega\). In this case, \(x\mapsto W(x,u(x))\) is Lebesgue-measurable whenever \(u:\Omega\to\mathbb{R}^{N}\) is Lebesgue-measurable.\end{definition}
   Note that
\begin{equation*}
    \Big|\int_{0}^{1}i_{x}(t)dt \Big| \leq \int_{0}^{1} |i_{x}(t)|dt \leq \int_{0}^{1}\left( \sup_{t \in [0,1]}|i_{x}(t)|  \right)dt = \sup_{t \in [0,1]}|i_{x}(t)|
\end{equation*}
Since $i_{x}(t)$ is continuous in a compact set, the maximum is attained for each $x \in M$, i.e., for each fixed $x_{0} \in M$,
there exists $t_{0} \in [0,1]$ such that $i_{x_{0}}$ attains the supremum at $t_{0}$. 
In fact there exists a Borel measurable function $t(\cdot): M \to [0,1]$ by \cite[t.18.13]{AliprantisBorder}
such that for each $x^{\prime} \in M$, $t(x^{\prime})$ 
satisfies $\sup_{t \in [0,1]}|i_{x^{\prime}}(t)|=|i_{x^{\prime}}(t(x^{\prime}))|$. Therefore using (\ref{desigualdad numero 1 del caso completo}), for constants $ \widetilde{C}, \widehat{C}> 0$, and that $0 \leq t(x) \leq 1$, we get 
\begin{align*}
  \norm{\mathcal{A}(u_{n}) - \mathcal{A}(u)}^{2}_{H^{s}(M)}   
  &\leq C\int_{M}|u_{n}(x) - u(x)|^{2} \left( \left| 1 + |t(x)u_{n}(x) + (1-t(x))u(x)|^{\beta} \right| \right)^{2}dV_{g}\\
  &\leq \widetilde{C}\int_{M}|u_{n}(x) - u(x)|^{2} \left( 1 + |t(x)u_{n}(x) + (1-t(x))u(x)|^{2\beta } \right)dV_{g} \\
  &\leq \widehat{C}\int_{M}|u_{n}(x) - u(x)|^{2} \left( 1 + |t(x)u_{n}(x)|^{2\beta} + |(1-t(x))u(x)|^{2\beta} \right)dV_{g} \\ 
  &\leq \widehat{C}\int_{M}|u_{n}(x) - u(x)|^{2}\left(  1 + |u_{n}(x)|^{2 \beta} + |u(x)|^{2 \beta}  \right)dV_{g}
\end{align*}
Using the embedding (\ref{embedding caso Cartan-hadamard}), there exists a constant 
$K>2\max\left\{ \|u_n\|_{L^{\infty}(M)}^{2\beta}, \|u\|_{L^{\infty}(M)}^{2\beta} \right\}$.\\
Then we have the following computations
\begin{align*}
 \norm{\mathcal{A}(u_{n}) - \mathcal{A}(u) }^{2}_{H^{s}(M)} &\leq \widehat{C} \int_{M}|u_{n}(x) - u(x)|^{2}\left(  1
 + |u_{n}(x)|^{2 \beta} + |u(x)|^{2 \beta}  \right)dV_{g}  \\
 &\leq \widehat{C} \left( \left(1+K\right)\norm{u_{n} - u}_{L^{2}(M)}^{2}\right) 
 \leq \widehat{C}\left( \left(1+K \right)\norm{u_{n} - u}_{H^{s}(M)}^{2} \right).
\end{align*}
Hence $\widetilde{u}_n \to \widetilde{u}$ in $H^{s}(M)$, so $\mathcal{A}$ is a continuous map. \\
Set $B:=\{L_{c}^{-1}\mathcal{R}(u)\ :\ u\in\mathcal{G}_{\varepsilon}\}$, this is a bounded subset 
in $\widetilde{H}_{f}(M,g)$; in fact
\begin{equation*}
    \norm{L_{c}^{-1}\mc{R}(u)}_{\widetilde{H}_{f}(M)} = \norm{\mc{R}(u)}_{L^{2}(M)} \leq K\left( \norm{h}_{L^{2}(M)}  + (C_{I})^{\alpha}\norm{u}^{\alpha}_{H^{s}(M)}\right) < \infty 
\end{equation*}
Since $\iota: \widetilde{H}_{f}(M,g) \hookrightarrow H^{s}(M,g)$ is compact by claim \eqref{claim teorema caso cartan hadamard}, we have that 
\begin{equation}
    \mc{A}(\mc{G}_{\varepsilon})=\iota \circ L_{c}^{-1} \circ \mc{R}(\mc{G}_{\varepsilon})=\iota(B)
\end{equation}
then $\mathcal{A}(\mc{G}_{\varepsilon})$ is relatively compact in $H^{s}(M,g)$. 
Define 
\begin{equation*}
\mc{C}:=\overline{\text{conv}}(\,\overline{\mc{A}(\mc{G}_{\varepsilon})}\,), \quad \text{conv}(\, \overline{\mc{A}(\mc{G}_{\varepsilon})}\,)=\bigcap_{ \{D \subseteq H^{s}(M) \, \mid\, \, \overline{\mc{A}(\mc{G}_{\varepsilon})} \,\subseteq D, D~\text{convex}\} }D,
\end{equation*}
in fact $\mc{C}$ is the closed convex hull of the closure of the image. 
Since
$\mathcal{A}(\mathcal{G}_{\varepsilon})\subseteq\mathcal{G}_{\varepsilon}$ and $\mathcal{G}_{\varepsilon}$ is closed and convex in $H^{s}(M,g)$, we have
$\mc{C}\subseteq\mathcal{G}_{\varepsilon}$, and consequently
$\mathcal{A}(\mc{C})\subseteq\mathcal{A}(\mathcal{G}_{\varepsilon})\subseteq \mc{C}$. Moreover since that $\mc{A}(\mc{G}_{\varepsilon})$ is relatively compact in $H^{s}(M,g)$ which is a Fréchet space, the closed convex hull of a relatively compact set is compact by Theorem \cite[t.3.20(c)]{Rudin1}. 
Then $\mc{C}$ is a nonempty(it contains $\mc{A}(0)$), compact, convex subset of $H^{s}(M,g)$, and $\mc{A}|_{\mc{C}}: \mc{C} \to \mc{C}$ is continuous. By Schauder's fixed point Theorem \cite[s.9.2.2, t.3]{Evans}, there exists
$\psi\in \mc{C}\subseteq\mathcal{G}_{\varepsilon}$ such that
$\mathcal{A}(\psi)=\psi$.
\end{proof}

\newpage

\subsection{Cartan-Hadamard Manifold Example}

\justify
Consider now an example of a Cartan–Hadamard Riemannian manifold $(M,g)$ satisfying that the sectional curvature $K_{p}(r) \to -\infty$ as $r \to \infty$, and condition (\ref{condicion sombrero}).

\begin{tcolorbox}[
    colback=gray!10,       % Color del fondo del contenido
    coltext=black,         % Color del texto del contenido
    colframe=red!8,      % Color del marco (igual al fondo para que no se note)
    boxrule=0pt,           % Grosor del borde (0pt para que no haya borde visible)
    sharp corners,         % Esquinas rectas
    title=,     % Título de la caja
    coltitle=black,        % Color del texto del título
    colbacktitle= gray!20,  % Color de fondo del título
    fonttitle=\mdseries   % Fuente en negrita para el título
]
\begin{example}\label{ejemplo de variedades Cartan-Hadamard}
      Given the manifold $(\mathbb{R}^{2}, g)$ where $g$ is defined as $g=dr^{2} + \varphi(r)^{2}d\Omega_{1}^{2}$ with $ \varphi: [0,\infty)\to [0,\infty), \varphi(r)=re^{r^{2}/2}$ and $d\Omega_{1}^{2}$ the measure on $\mathbb{S}^{1}$. The conditions of theorem \eqref{teorema finnal para el caso no compacto} are satisfied.
\end{example}
\end{tcolorbox}

\begin{proof}
Note that by definition $\mb{R}^{2}$ is a contractible space, hence it is simply connected. 
Since $\varphi(r)$ satisfies $\varphi(0)=0, \varphi^{\prime}(0)=1$, and $\varphi^{(2k)}(0)=0$ for all $k \in \mathbb{Z}_{\geq 0}$, then the metric $g$ is smooth \cite[p.20]{Petersen} avoiding any kind of singularity on the manifold $(\mb{R}^{2}, g)$, therefore the manifold is simply connected.
To save notation let us consider $M=\mathbb{R}^{2}$. Let us prove that this manifold is complete, for which we use the Hopf-Rinow Theorem \cite[t.6.19, c.6.20]{Lee2}. If we prove that there exists a $p \in M$ such that the exponential map $\exp_{p}=\exp|_{\mc{E}_{p}}$ is defined on all of $T_{p}M$ then $M$ is complete, where 
$\mc{E}_{p}=\mc{E}\cap T_{p}M$ and $\mc{E}=\{v \in TM \mid \gamma_{v} ~\text{is defined on an interval containing}~[0,1]\}$, with $\exp:\mc{E} \to M, \exp(v)=\gamma_{v}(1)$ and $\gamma_{v}$ is the maximal geodesic such that $\gamma(0)=p, \gamma^{\prime}(0)=v$.\\
Let us consider the curves that start from the origin and travel in a straight line outward, keeping their angles constant. In polar coordinates $(r,\theta)$, these smooth curves are parametrized by the radius $r$ as 
\begin{equation}\label{rayo geodesico origen}
    \gamma(t)=(t,\theta_{0}), \quad t\geq 0.
\end{equation}
Note that $\gamma^{\prime}(t)=\partial_{r}$, which is the radial coordinate vector, which for this manifold satisfies being a normal vector field, that is $\norm{\gamma^{\prime}(t)}_{g}^{2}=\langle \gamma^{\prime}(t) , \gamma^{\prime}(t)   \rangle_{g}=\langle \partial_{r}, \partial_{r}    \rangle_{g}=g(\partial_{r}, \partial_{r})=1$.\\
Let us prove that $\gamma$ is a geodesic with respect to the Levi-Civita connection $\nabla$, for this we must prove that $\nabla_{\gamma^{\prime}}\gamma^{\prime}=0, \forall t \in \text{dom}(\gamma)$, but that is equivalent to proving that in smooth coordinates $(x^{i})$, if we write $\gamma$ as $\gamma(t)=(x^{1}(t),...,x^{n}(t))$ then $\gamma$ is a geodesic if it satisfies the geodesic equation \cite[p.103]{Lee} given by 
\begin{equation}
    \ddot{\gamma}^{k} + \sum_{i,j} \Gamma^{k}_{ij} \dot{\gamma}^{i} \dot{\gamma}^{j}=0
\end{equation}
Since in our case $\dot{\gamma}(t)=(1,0)$, the previous equation reduces to verifying whether $\Gamma_{rr}^{r}(\dot{\gamma}^{r})^{2}=0$ holds, since $\nabla$ is the Levi-Civita connection,  the Christoffel symbol $\Gamma_{rr}^{r}$ is given by
\begin{equation}
\Gamma_{rr}^{r}=\frac{1}{2}g^{rr}\left( 2\partial_{r}g_{rr} - \partial_{r}g_{rr} \right)=\frac{1}{2}g^{rr}(\partial_{r}g_{rr})
\end{equation}
Using that $g_{rr}=1$, and $g^{rr}=1$, we have that $\Gamma_{rr}^{r}=0$, therefore $\gamma$ is a geodesic.
Note that each geodesic $\gamma$ satisfying $\gamma(0)=0$ is of the form \eqref{rayo geodesico origen}, hence for each $v \in T_{0}M$, there exists a maximal geodesic $\gamma_{v}$ such that $\gamma_{v}(0)=0$ and $\gamma^{\prime}_{v}(0)=v$ defined for all $t\geq0$ \cite[c.4.28]{Lee2}, which implies $[0,1] \subset \text{dom}(\gamma_{v})$. Therefore $T_{0}M \subset \mc{E}$, hence $T_{0}M \subset \mc{E} \cap T_{0}M=\mc{E}_{0} \subset T_{0}M$, and $\exp|_{\mc{E}_{0}}$ is defined on all of $T_{0}M$, indeed the manifold is complete.\\
Let us prove that the sectional curvature is negative at every point, and that for a fixed point $p \in M$ the expression $K_{p}(r)$ satisfies $K_{p}(r) \to -\infty, r \to \infty$. 
Let us consider a fixed $r_{0} \in \mb{R}^{+}$, such that $K_{p}(r_{0})=K(x_{0},\pi)$. 
Given a chart $(U,(x^{i}))$ with $(x^{1},x^{2})=(r,\theta)$ and $x_{0} \in U$, we have that $\mathcal{B}=\{\partial_{r}, \partial_{\theta}\}$ forms a basis of $T_{x_{0}}M$. Using \cite[p.121]{Petersen} we have
\begin{equation}
K(x_{0},\pi)=K(\partial_{r}, \partial_{\theta})=-\frac{\varphi^{\prime\prime}(r_{0})}{\varphi(r_{0})}
\end{equation}
therefore for each point $r \in \mb{R}^{+}$, since $\varphi(r)=re^{r^{2}/2}$, we have that 
\begin{equation}
    K_{p}(r)=-(3+r^{2}).
\end{equation}
This implies that $K\leq -3$ at every point, hence the manifold $(\mb{R}^{2},g)$ is a Cartan-Hadamard manifold that moreover satisfies $K_{p}(r) \to -\infty, r \to \infty$.\\
Considering the metric $g$, we have that $g_{rr}=1, g_{\theta\theta}=(\varphi(r))^{2}$, hence since its representative matrix is diagonal, its inverse elements are $g^{rr}=1, g^{\theta\theta}=(\varphi(r))^{-2}$, therefore the representative expression of the Laplacian will be 
\begin{equation}\label{ecuacion radial original}
    \Delta_{g}u=-\frac{1}{\varphi(r)}\frac{\partial}{\partial r}\left(   \varphi(r)\frac{\partial u}{\partial r} \right) - \frac{1}{\varphi(r)}\frac{\partial}{\partial \theta}\left( \frac{1}{\varphi(r)}\frac{\partial u}{\partial \theta}   \right).
\end{equation}

\begin{claim}\label{claim del ejemplo de cartan-hadamard}
 Let $\{\lambda_j\}_{j \in \mathbb{N}}$ be the spectrum of the Laplace--Beltrami operator $\Delta_g$ and  $J = \{ j \in \mb{N} \mid \lambda_j < 1 \}$. There exists an orthonormal basis $\{\omega_j\}$ of eigenfunctions of $\Delta_{g}$ of $L^2(M,g)$ such that $\| \omega_j \|_{L^\infty(M)} < \infty$ for all $j \in\mb{N}$ and, the following is satisfied
\begin{equation}\label{ecuacion sombrero ejemplo}
    \sum_{j \in J} (1 + \lambda_j^2)^{-s} \| \omega_j \|_{L^\infty(M)}^2
+ \sum_{j \in (\mathbb{N} \setminus J)} (1 + \lambda_j)^{-s} \| \omega_j \|_{L^\infty(M)}^2 < \infty, \quad s>3.
\end{equation} 
\end{claim}

\begin{proofofclaim}
 We first note that, although condition \eqref{condicion sombrero} is stated for all derivative orders $\ell=0,\dots,m$, the existence result we invoke (Theorem \ref{teorema finnal para el caso no compacto}) only requires the case $\ell=0$. Indeed, its proof relies only on the $L^\infty$ bounds of the eigenfunctions, without involving their covariant derivatives. Consequently, to verify that the present example satisfies the hypotheses of that Theorem \ref{teorema finnal para el caso no compacto}, it suffices to prove the convergence of the series in \eqref{ecuacion sombrero ejemplo}.\\
Using the Theorem \cite[p.360]{McKean}, we have that $\lambda_j \geq \frac{(2-1)^2 \cdot 3}{4} = \frac{3}{4}$ for all $j \in \mathbb{N}$, hence $ 0 \notin \sigma(\Delta_{g})$.
  Indeed for every constant $C>0$, the set $\{j \in \mb{N} \mid \lambda_{j} \leq C\}$ is finite, therefore $J$ is finite. Using the Theorem \cite[t.II.10]{ReedI} and the example \cite[p.160]{ReedII} we have that 
\begin{equation}\label{descomposicion del espacio L2}
L^{2}(\mb{R}^{2})=L^{2}(\mb{R}_{>0}, \varphi(r)dr) \otimes L^{2}(\mb{S}^{1}, d\Omega_{1}^{2})=\bigoplus_{\ell \geq 0}\mc{L}_{\ell}.
\end{equation}
Here $\mc{L}_{\ell}$ is defined as $\mc{L}_{\ell}=L^{2}(\mb{R}_{>0}, \varphi(r)dr) \otimes K_{\ell}$ with $K_{\ell}$ the $\ell$-th eigenspace of $\Delta_{\mb{S}^{1}}=-\frac{\partial^{2}}{\partial \theta^{2}}$ where $\sigma(\Delta_{\mb{S}^{1}})=\{\ell^{2} \mid \ell \in \mb{N}_{0}\}$, then $K_0 = \text{span}\{1\}$ and $K_\ell = \text{span}\{\cos(\ell\theta), \sin(\ell\theta)\}$. 
We know that the resolvent operator of $\Delta_{g}$ is compact by \cite[p.5.12]{Schmudgen}, denoting the self-adjoint operator $L_{\ell}:=\Delta_{g}\big|_{\mc{L}_{\ell}}$, we have that this inherits the compactness of the resolvent operator, hence $L_{\ell}$ has purely discrete spectrum $\sigma(L_\ell) = \{\lambda_{\ell, k} \mid k \in \mathbb{N}_0\}$ with $0 < \lambda_{\ell,0} \leq \lambda_{\ell,1} \leq \lambda_{\ell,2} \leq \cdots \to \infty$.\\
%$$\mathcal{L}_{\ell} =\begin{cases}\bigl\{\, f\in L^{2}(M) : f(r,\theta)=f_{0}(r),\ f_{0}\in L^{2}(\mathbb{R}^{+},\varphi\,dr) \,\bigr\}, & \ell = 0,\\[6pt]\bigl\{\, u(r)\cos(\ell\theta) + v(r)\sin(\ell\theta) : u,v\in L^{2}(\mathbb{R}^{+},\varphi\,dr) \,\bigr\}, & \ell \ge 1 .\end{cases}$$
We seek solutions to the radial equation \eqref{ecuacion radial original} of the form
\begin{equation}
    \omega_{\ell, k}(r,\theta)=u_{\ell,k}(r)Ce^{i\ell\theta}
\end{equation}
where $\ell=|m|$ with $m \in \mb{Z}$, which ensures that $\psi(\theta)=Ce^{i m \theta}$ has unique values in $[0,2\pi]$ and $C >0$ is a normalization constant.
%which ensures that $\psi(\theta)=Ce^{im \theta}$ has single-values on $[0,2\pi]$ \cite[p.4]{Gang} with $C>0$ a normalizing constant.
In fact, the solutions $\omega_{\ell,k}$ must satisfy the equation
\begin{equation}\label{ecuacion radial original2}
    -u_{\ell,k}'' - \left( \frac{1}{r} + r \right) u_{\ell,k}' + \left(\frac{\ell^2}{r^{2}e^{r^{2}}} - \lambda\right)u_{\ell,k} = 0.
\end{equation}
Since each $\omega_{\ell,k}$ must be orthonormal, the functions $u_{\ell,k}$ must satisfy the following condition 
\begin{align*}\label{condicion de normalizacion de soluciones}
       1 &= \left( \int_0^\infty |u_{\ell,k}(r)|^2 \, r e^{r^2/2} \, dr \right) \underbrace{\left( \int_0^{2\pi} |Ce^{i\ell \theta} |^2 \, d\theta \right)}_{=\,1} = \int_0^\infty |u_{\ell,k}(r)|^2 \, r e^{r^2/2} \, dr.
\end{align*}
Since the Laplace-Beltrami operator eigenfunctions $\omega_{\ell,k}$ in each sector are smooth \cite[c.7.3]{Grigoryan}, which implies that $u_{\ell,k}$ are also smooth. 
The equation \eqref{ecuacion radial original2}, is an equation of the form $u^{''}+p(r)u^{'}+q(r)u=0$ 
\cite{Teschl}
%\cite[(4.20)]{Teschl} 
with $p(r)=\frac{1}{r}+r$ and $q(r)=\lambda - \frac{\ell^{2}}{\varphi^{2}(r)}$. 
In \cite[p.117]{Teschl} these coefficients are assumed to be of the form $p(r)=\tfrac{1}{r}\sum_{i \geq 0}p_{i}r^{i}$, and $q(r)=\tfrac{1}{r^{2}}\sum_{i \geq 0}q_{i}r^{i}$, hence we can say that the coefficients $p_{0}$ and $q_{0}$ respectively satisfy the following limits
\begin{equation}
    rp(r) \xrightarrow[r \to 0]{}p_{0}, \quad r^{2}q(r) \xrightarrow[r \to 0]{}q_{0}
\end{equation}
hence $p_{0}=1$ and $q_{0}=-\ell^{2}$, therefore the \emph{indicial equation} \cite[(4.36)]{Teschl} given by $\alpha^{2}+(p_{0}-1)\alpha + q_{0}=0$ has the following \emph{characteristic exponents} given by $\alpha_{1,2}=\pm \ell$, hence $\Re(\alpha_{1}) \geq \Re(\alpha_{2})$. Therefore applying the theorem \cite[t.4.5]{Teschl} we have that the solutions $u$ of the radial equation \eqref{ecuacion radial original2} are of the form $u_{1}(r)=r^{\alpha_{1}}h_{1}(r)$, $u_{2}(r)=r^{\alpha_{2}}h_{2}(r) + c\ln(r)u_{1}(r)$
where the functions $h_{j}(r)$ are analytic near $r = 0$ and satisfy $h_{j}(0) = 1$. The constant $c \in \mb{C}$ might be zero unless $\alpha_{1}-\alpha_{2}= 0$. Therefore the asymptotic behavior of $u_{\ell,k}$ satisfies
\begin{equation}\label{desarrollo polinomico de las autofunciones de la ecuacion radial}
    u_{\ell,k}(r)=\mc{O}(r^{\ell}), \quad u^{\prime}_{\ell,k}(r)=\mc{O}(r^{\ell-1}), \quad u_{0,k}'(0) = 0.
\end{equation}
% Esto pues se hace un desarrollo de taylor de la funcion u_{0,k} y al final es una cuestion usual de la forma 
%$$u_{0,k}(r) = c_0 + c_2 r^2 + c_4 r^4 + \dots$$
% luego su derivada es $$u_{0,k}'(r) = 2c_2 r + 4c_4 r^3 + \dots$$.

\justify

For $\ell \geq 1$, let us apply the following change of variable $v = \varphi^{1/2} u$, which will allow us to transform the radial equation \eqref{ecuacion radial original2} into a \emph{Schrödinger operator} of the form 
\begin{equation}\label{Operador de Schrodinger}
-v'' + V_\ell v = \lambda v,
\qquad
V_\ell(r) = 1 + \scalebox{1.2}{$\frac{r^2}{4}$} - \scalebox{1.2}{$\frac{1}{4r^2}$} + \scalebox{1.2}{$\frac{\ell^2}{r^2 e^{r^2}}$},
\end{equation}
%Para eliminar el término asimétrico, definimos el cambio de variable $v = \varphi^{1/2} u$, lo que implica que $u = \varphi^{-1/2} v$. Aplicando la regla del producto y de la cadena, calculamos las derivadas de $u$ en términos de la nueva función $v$. La primera derivada resulta ser $u' = -\frac{1}{2} \varphi^{-3/2} \varphi' v + \varphi^{-1/2} v'$. Al derivar esta expresión nuevamente para obtener $u''$ y agrupar los términos resultantes, obtenemos:
%$$ u'' = \varphi^{-1/2} v'' - \varphi^{-3/2} \varphi' v' + \left( \frac{3}{4} \varphi^{-5/2} (\varphi')^2 - \frac{1}{2} \varphi^{-3/2} \varphi'' \right) v. $$

\begin{lemma}\label{lema potencial}
 For $\ell \geq 1$ and $r > 0$ it holds that $V_\ell(r) \geq \frac{r^2}{4}$. If $\ell \geq 2$, $\inf_{r > 0} V_\ell \geq \frac{\ln \ell}{4}$.
\end{lemma}

\begin{proofoflemma}
Let us consider the case $\ell \geq 1$, for $r \geq \frac{1}{2}$ we have $\frac{1}{4r^2} \leq 1$, hence $V_\ell \geq 1 + \frac{r^2}{4} - 1 = \frac{r^2}{4}$. For $0 < r < \frac{1}{2}$, since $\ell^2 e^{-r^2} \geq e^{-1/4} > \frac{1}{4}$, we have $(\ell^2 e^{-r^2} - 1/4)/r^2 > 0$ and therefore $V_\ell > 1 + \frac{r^2}{4}$. 
For the case $\ell \geq 2$, suppose $r^2 \geq \ln \ell$, hence $V_\ell \geq 1 + \frac{\ln \ell}{4} - \frac{1}{4 \ln \ell} \geq \frac{\ln \ell}{4}$, since $\ln \ell \geq \ln 2 > \frac{1}{4}$.

If $r^2 < \ln \ell$ we have $e^{r^2} < \ell$, which implies
$$
V_\ell \geq \frac{\ell^2 e^{-r^2} - \frac{1}{4}}{r^2} \geq \frac{\ell - \frac{1}{4}}{\ln \ell} \geq \frac{\ln \ell}{4},
$$
where the last inequality is equivalent to $4\ell - 1 \geq (\ln \ell)^2$, valid for all $\ell \geq 1$ because $h(x) = 4x - 1 - (\ln x)^2$ satisfies $h(1) = 3 > 0$ and $h'(x) = 4 - \frac{2 \ln x}{x} \geq 4 - \frac{2}{e} > 0$.
\end{proofoflemma}

\begin{lemma}\label{lema de decamiento}
For $\ell \geq 1$, $v$ and $v'$ decay at $+\infty$ faster than any exponential, $\int_0^\infty r^2 v^2 \, dr \leq 4\lambda$, and
\begin{equation}\label{ecuacion de energia}
\int_0^\infty (v')^2 \, dr + \int_0^\infty V_\ell v^2 \, dr = \lambda.
\end{equation}
\end{lemma}

\begin{proofoflemma}
Since $V_{\ell}(r) \to \infty$ as $r \to \infty$, for an arbitrary $K>0$, there exists an $r_{K} \in \mb{R}_{>0}$, such that $V_{\ell}(r)-\lambda\geq K^{2}$ on $[r_{K},\infty)$. 
Note that $vv^{\prime\prime}=(V_{\ell}-\lambda)v^{2} \geq K^{2}v^{2} \geq 0$ on $[r_{K},\infty)$
which is strictly positive wherever $v(r) \neq 0$. Let us first prove that the non-trivial solution $v$ has no zeros on $[r_K,\infty)$. Suppose by contradiction that there exists $r_0 \in [r_K,\infty)$ such that $v(r_0) = 0$. If $v'(r_0) = 0$, then by uniqueness of the Cauchy problem for the linear second-order differential equation $v'' = (V_\ell - \lambda)v$, we would have $v \equiv 0$, contradicting $\|v\|_{L^2(\mb{R}^+,dr)} = 1$. If $v'(r_0) > 0$, then $v(r) > 0$ immediately to the right of $r_0$, 
then $v$ grows linearly on the interval $[r_{0},\infty)$, which contradicts the square integrability of $v$. For the case $v'(r_{0})<0$, the procedure is analogous, considering $-v$. Therefore $v$ has no zeros in $[r_{K}, \infty)$.
Hence $\tfrac{v^{\prime\prime}}{v} \geq K^{2}>0$. 
Therefore $v$ and $v^{\prime\prime}$ have the same sign; suppose without loss of generality, that $v>0$, hence $v^{\prime\prime}>0$.
Therefore $v^{\prime}$ is increasing, and $v$ does not vanish on $[r_{K}, \infty)$. We prove that $v^{\prime}(r)<0$ for all $r 
\in [r_{K}, \infty)$. Suppose that there exists $r_{0} \in [r_{K},\infty)$ such that $v'(r_{0})\geq 0$; then, since $v'$ is increasing, we have $v'(r)\geq v'(r_{0})\geq 0$ for all $r \in [r_{0},\infty)$, so $v(r)\geq v(r_0) > 0$ is non-decreasing on the $[r_{0},\infty)$, which contradicts the fact that $v$ is square integrable. 
Therefore $v'(r)<0$ for all $r \in [r_{K}, \infty)$. Indeed, $v$ is decreasing, convex, square integrable,  $v^{\prime}$ is increasing and bounded above by $0$, hence $v^{\prime} \nearrow 0$.
Note that $\frac{d}{dr} \left( v'^2 - K^2 v^2 \right) = 2 v' \left( v'' - K^2 v \right) \leq 0.$
\vspace{0.12cm}

Hence $\psi(r)=(v^{\prime}(r))^{2}-K^{2}v^{2}(r)$ is a non-increasing function, with $\psi(r) \to 0$ as $r \to \infty$, because $v^{\prime} \to 0$, and $v \to 0$ as $r \to \infty$. 
Therefore, $\psi(r)\geq \psi(\infty)=0$ for all $r \in [r_{K}, \infty)$ which implies that $v'^2 \geq K^{2}v^{2}$, and since $v^{\prime}<0$ this implies that $v^{\prime} \leq -Kv$. Solving the differential equation and integrating on $[r_{K},r)$ we have that $v(r) \leq v(r_{K}) e^{-K(r - r_{K})}$.
To establish the super-exponential decay of $v'$, we use $v'(r) \to 0$ and integrate $v''(t) = (V_\ell(t) - \lambda)v(t)$ on $[r,\infty)$ for $r \geq r_K$, obtaining
$$
v'(r) = -\int_r^\infty \bigl(V_\ell(t) - \lambda\bigr)v(t)\,dt.
$$
Since $V_\ell(t) = \mc{O}(t^2)$ as $t \to \infty$ and $v(t)$ decays faster than $e^{-Kt}$ for every $K>0$, it follows that $v'(r)$ also has super-exponential decay as $r \to \infty$. 
Multiplying \eqref{Operador de Schrodinger} by $v$ and integrating by parts on $(\varepsilon, R)$, we obtain that the term with $v^{\prime\prime}v$ satisfies the following
\begin{equation}\label{ecuacion del lema 2 ejemplo tesis}
-\int_{\varepsilon}^{R}v^{\prime\prime}vdr=-(vv^{\prime})\Big|^{R}_{\varepsilon} + \int_{\varepsilon}^{R}v^{\prime2}dr,
\end{equation}
using the super-exponential decay and the behavior \eqref{desarrollo polinomico de las autofunciones de la ecuacion radial}, $vv^{\prime}=\mc{O}(r^{2\ell}) \to 0$ as $r \to 0^{+}$, therefore the term $vv^{\prime}$ of \eqref{ecuacion del lema 2 ejemplo tesis} vanishes, and since $\|v\|_{L^{2}(\mb{R}^{+},dr)}=1$, taking $\varepsilon \to 0$ and $R \to \infty$ 
\begin{equation}
    \int_{0}^{\infty}(v^{\prime})^{2}dr + \int_{0}^{\infty}V_{\ell}v^{2}dr=\lambda.
\end{equation}
Since $\ell \geq 1$ and $r>0$, $V_{\ell}(r)\geq \tfrac{r^{2}}{4}$ is satisfied \eqref{lema potencial} hence $\int_{0}^{\infty} r^2 v^2 \,dr \leq 4 \int_{0}^{\infty} V_\ell v^2 \, dr \leq 4\lambda$.
\end{proofoflemma}

\begin{lemma}\label{valores aproximados del espectro del ejemplo}
For $\ell \geq 1$ and $k \geq 0$, $\lambda_{\ell,k} \geq 2k + \frac{1}{2}$. Moreover, $\lambda_{\ell,k} \geq \frac{\ln \ell}{4}$ if $\ell \geq 2$.
\end{lemma}

\begin{proofoflemma}
Let us denote $\lambda=\lambda_{\ell,k}$. Suppose $\ell \geq 2$, by the Lemma \ref{lema potencial} we have that $\inf_{r>0}V_{\ell}(r) \geq \tfrac{\ln \ell}{4}$, hence using equation \eqref{ecuacion de energia}, we have $\lambda \geq \tfrac{\ln \ell}{4}$. \\
Let us fix $\ell\geq 1$. Consider the subspace $E_k = \text{span}(v_{\ell,0}, \ldots, v_{\ell,k})$, note that this is an orthonormal system in $L^{2}(\mb{R}^{+},dr)$. 
%puesto que $v_{\ell,j}=\varphi^{1/2}u_{\ell,j}$ para cada $j \in \mb{N}_{0}$, 
Therefore, using equation \eqref{ecuacion de energia} we have that 
\begin{equation*}
    \int_0^\infty v_{\ell,i}'(r) v_{\ell,j}'(r) \, dr + \int_0^\infty V_\ell(r) v_{\ell,i}(r) v_{\ell,j}(r) \, dr = \lambda_{\ell,j} \, \delta_{ij},
\end{equation*}
so the quadratic form $q(v) = \int v'^2 + \int V_\ell v^2$ satisfies $q(v) \leq \lambda_{\ell,k}$ on $E_k$.
The quadratic form $q$ dominates the quadratic form of the \emph{harmonic oscillator} $\mc{H}_N = -\frac{d^2}{dr^2} + \frac{r^2}{4}$ on $L^2(\mb{R}_{>0})$ with Neumann condition, whose operator domain is $\mc{D}(\mc{H}_{N})=\{\omega \in H^{2}(\mb{R}_{>0}) \mid  r^{2}\omega \in L^{2}(\mb{R}_{>0}), \omega^{\prime}(0)=0\}$ (\cite{Bonnaille}). 
We obtain the quadratic form associated with this oscillator, by multiplying $\mc{H}_{N}v$ by $v$ and integrating over its domain, obtaining
\begin{equation}
    \langle v, \mc{H}_{N}v \rangle=\int_{0}^{\infty}(v^{\prime 2})dr + \int_{0}^{\infty}\frac{r^{2}}{4}v^{2}dr
\end{equation}
with domain $\mc{Q}(\mc{H}_{N})=\{\omega \in H^{1}(\mb{R}_{>0}) \mid r\omega \in L^{2}(\mb{R}_{>0})\}$, this domain contains $E_{k}$, since each generator $v_{\ell,j} \in E_{k}$ is in $
\mc{Q}(\mc{H}_{N})$ thanks to the bounds $\int v'^2 \leq \lambda$ and $\int r^2 v^2 \leq 4\lambda$ of equation \eqref{ecuacion de energia}. 
The spectrum of $\mc{H}_{N}$ on the whole line is $\sigma(\mc{H}_{N})=\{(n+\tfrac{1}{2}) \mid n \in \mb{N}_{0} \}$ \cite[t.8.5]{Teschl1}. However since the operator must satisfy the Neumann conditions on $[0,\infty)$ it is restricted to the functions whose derivative is zero at the origin, that is, to the even functions. Since the Hermite functions $\psi_{n}$ defined in \cite[(8.41)]{Teschl1} are even if and only if the index $n$ is even, the odd eigenvalues are discarded, hence $\sigma(\mc{H}_{N})=\{ (2k+\tfrac{1}{2}) \mid k \in \mb{N}_{0}\}$. 
Using the min-max principle \cite[t.XIII.1]{ReedIV}
and \cite[t.XIII.2]{ReedIV} the following equality holds  
where $\mu_{k}(\mc{H}_{N})$ is the $k$-th eigenvalue of $\mc{H}_{N}$. Therefore we have
\begin{equation}
    2k + \tfrac{1}{2}=\mu_{k}(\mc{H}_{N}) \leq \max_{v \in E_k, \, \|v\| = 1} \left( \int v'^2 + \int \tfrac{r^2}{4} v^2 \right) \leq \max_{v \in E_k, \, \|v\| = 1} q(v) \leq \lambda_{\ell,k}.
\end{equation}
\end{proofoflemma}

\begin{lemma}\label{lemma control sobre las eigenfunciones ejemplo tesis}
For $\ell \geq 1$ and $k \geq 0$, with $\lambda = \lambda_{\ell,k}$, we have $ \| u_{\ell,k} \|_{L^\infty(\mb{R}^{+})}^2 \leq \frac{8 (1 + \lambda)}{\ell}.$
\end{lemma}

\begin{proofoflemma}
We will denote the functions $u_{\ell,k}$ and $v_{\ell,k}$ respectively as $u$ and $v$ to save notation. Note that $v(0^+) = 0$, since $v(r)=\mc{O}(r^{\ell+1/2})$ by \eqref{desarrollo polinomico de las autofunciones de la ecuacion radial} using the fundamental theorem of calculus, we have that $v(r)^{2}=2\int_{0}^{r}v(x)v^{\prime}(x)dx$, hence $v(r)^{2} \leq 2 \|v\|_2 \|v'\|_2$ then using equation \eqref{ecuacion de energia} we have that $\|v'\|_2^2 \leq \lambda$, hence $\|v\|_\infty^2 \leq 2\sqrt{\lambda}$.
Let $b > 0$ be the unique point where $\varphi(b) = \ell/(4\sqrt{1+\lambda})$, which exists since $\varphi$ is an increasing bijection of $(0,\infty)$.\\
For $r \geq b$ we have 
$$u(r)^2 = \frac{v(r)^2}{\varphi(r)} \leq \frac{2\sqrt{\lambda}}{\varphi(b)} = \frac{8 \sqrt{\lambda(1+\lambda)}}{\ell}\leq \frac{8 \sqrt{\lambda^{2}+2\lambda + 1}}{\ell}  \leq \frac{8(1+\lambda)}{\ell}.$$
On $(0, b]$ we have $\frac{\ell^2}{\varphi^2(r)} \geq \frac{\ell^2}{\varphi(b)^2} = 16(1+\lambda) > \lambda$, where $0<r\leq b$ so that the coefficient $Q = \frac{\ell^2}{\varphi^2(r)} - \lambda$ is strictly positive.
Note that the radial equation \eqref{ecuacion radial original2} can be written as 
$$\varphi^{\prime}u^{\prime} + \varphi u^{\prime\prime}=\frac{\ell^{2}u}{\varphi} - \lambda \varphi u$$
equivalently it could be written as $(\varphi u')' = Q \varphi u$. Let us prove that $u$ does not vanish on $(0,b]$ and $\sup_{(0,b]} u = u(b)$. Suppose that $u$ vanishes on $(0,b]$, hence $\exists z \in (0,b]$ such that $u(z)=0$, if we integrate by parts on $(0,z)$ the expression $u (\varphi u')'$ we obtain the following
\begin{equation*}
    \int_{0}^{z} Q \varphi u^2 dr=\int_{0}^{z}u(\varphi u^{\prime})^{\prime}dr = u(z)\varphi(z)u^{\prime}(z) - u(0)\varphi(0)u^{\prime}(0) - \int_{0}^{z}\varphi(u^{\prime})^{2}dr=- \int_{0}^{z}\varphi(u^{\prime})^{2}dr.
\end{equation*}
Since the left-hand side is non-negative and the right-hand side non-positive, both vanish, hence $u \equiv u' \equiv 0$ on $(0,z)$, and by uniqueness of the initial value problem \cite[t.3.2.1]{Boyce} for \eqref{ecuacion radial original2} we would have $u \equiv 0$ on $(0,\infty)$, which contradicts the normality, hence $u$ does not vanish on $(0,b]$, without loss of generality $u > 0$ on $(0,b]$. 
Then $(\varphi u')' = Q \varphi u > 0$, so $\varphi u'$ is strictly increasing, and since $\varphi u' \to 0$ as $r \to 0^+$ by \eqref{desarrollo polinomico de las autofunciones de la ecuacion radial}, we conclude $\varphi u' > 0$ on $(0,b]$. Consequently $u$ is strictly increasing on $(0,b]$ and $\sup_{(0,b]} u = u(b)$.
\end{proofoflemma}

%\newpage

\begin{lemma}\label{lema sector 0 tesis}
It holds that $\| u_{0,k} \|_{L^\infty(\mb{R}^{+})}^2 \leq 32 \, \lambda_{0,k}$, and there exists a constant $C_0 > 0$ such that 
%$N_0(\lambda) = \# \{ k : \lambda_{0,k} \leq \lambda \} \leq C_0 (1 + \lambda)$; en particular
$1 + \lambda_{0,k} \geq \frac{k}{C_0}$ for all $k \in \mb{N}_{0}$.
\end{lemma}

\begin{proofoflemma}
Let us denote $u=u_{0,k}$. This function solves $u'' + \left(\frac{1}{r} + r\right) u' + \lambda u = 0$ with $u'(0) = 0$ by \eqref{desarrollo polinomico de las autofunciones de la ecuacion radial}. Defining the expression $E = u'^2 + \lambda u^2$ known as energy, we have
$$
E' = 2 u' (u'' + \lambda u) = -2 \left( \frac{1}{r} + r \right) u'^2 \leq 0,
$$
therefore the function $E$ is non-increasing, hence $E(r) \leq E(0)$ for all $r \geq 0$, which implies $u'(r)^2+\lambda\,u(r)^2\le \lambda\,u(0)^2$, which implies that $\|u\|_\infty = |u(0)|$ and $|u'| \leq \sqrt{\lambda} \, |u(0)|$. Using the fundamental theorem of calculus $u(r)-u(0)=\int_{0}^{r}u^{\prime}(t)dt$ then we have
\begin{equation*}
    |u(r)-u(0)| \leq \int_{0}^{r}|u^{\prime}(t)|dt \leq \sqrt{\lambda}|u(0)|r
\end{equation*}
using the reverse triangle inequality, we have that $|u(0)|-|u(r)| \leq |u(r) - u(0)|$, so $|u(r)|\geq |u(0)|-\sqrt{\lambda}|u(0)|r=|u(0)|(1-\sqrt{\lambda}r)$,
for this inequality to make sense it must hold that $0\leq (1-\sqrt{\lambda}r) < 1$, without loss of generality let us consider $\sqrt{\lambda}r \leq \tfrac{1}{2}$, that is $0 \leq r \leq \frac{1}{2\sqrt{\lambda}}$, with this restriction, we have $|u(r)| \geq \tfrac{|u(0)|}{2}$, therefore
{\fontsize{10.5}{14}
\begin{equation}
    1 = \int_0^\infty u^2 \varphi \, dr \geq \frac{u(0)^2}{4} \int_0^{1/(2\sqrt{\lambda})} r \, dr = \frac{u(0)^2}{32 \lambda},
\end{equation}}
\justify
which proves $\| u_{0,k} \|_{L^\infty(\mb{R}^{+})}^2 \leq 32 \, \lambda_{0,k}$.
When $\ell=0$, the equation that is solved is
$-u'' - \left( \frac{1}{r} + r \right) u' = \lambda u.$ Let apply the  variable change $x=\tfrac{r^{2}}{2}$, indeed suppose $u(r)=\tilde{u}(x)$ then  $u' = r \frac{d\tilde{u}}{dx}$ and $u'' = \frac{d\tilde{u}}{dx} + 2x \frac{d^2\tilde{u}}{dx^2}$ and $\varphi dr=e^{x}dx$.
Substituting in \eqref{ecuacion radial original2} in the case $\ell=0$ we obtain
$$
x\frac{d^2\tilde{u}}{dx^2} + (1+x)\frac{d\tilde{u}}{dx} + \frac{\lambda}{2}\tilde{u} = 0. 
$$
Now let us make the change $\tilde{u}(x)=e^{-x}L(x)$ where $L(x)=e^{x}\tilde{u}(x)$, where applying it to the previous equation, we obtain the Laguerre equation \cite[p.504]{Abramowitz} given by 
\begin{equation}\label{ecuacion de laguerre ejemplo tesis}
    x L'' + (1-x)L' + nL = 0, \quad n=\left(\tfrac{\lambda}{2}-1\right). 
\end{equation}
When $n=k \geq 0$ is an integer, the solution of \eqref{ecuacion de laguerre ejemplo tesis} is the Laguerre polynomial $L_{k}(x)$ \cite[p.780]{Abramowitz}, given by $L_{k}(x)=\frac{e^{x}}{k!}\frac{d^k}{dx^k}(e^{-x}x^k)=\sum_{j=0}^{k}\binom{k}{j} \frac{(-1)^{j}}{j!}x^{j}$ which satisfy the orthonormality relation respect to Laguerre weight \cite[p.773]{Abramowitz}. In fact
$$\int_{0}^{\infty}e^{-x}L_{k}(x)L_{m}(x)dx=\delta_{km}.$$ 
Moreover $u_{0,k}(r)=e^{-r^{2}/2}L_{k}(\tfrac{r^{2}}{2})$ also satisfies being orthonormal in the space $L^2(\mathbb{R}_+,\varphi\,dr)$, that is, $\int_{0}^{\infty}u_{0,k}(r)u_{0,j}(r)\varphi dr=\delta_{kj}$. 
Let us define the \emph{generalized hypergeometric function} also known as \emph{confluent hypergeometric functions of the first kind} by the power series
\begin{equation}
    {}_{1}F_{1}(a;b;x)=\sum_{n\in \mathbb{N}}\frac{(a)_{n}}{(b)_{n}}\frac{x^{n}}{n!}, \quad (a)_{n}=a(a+1) \cdots (a+n-1), (a)_{0}=1
\end{equation}
this is the regular solution of the Kummer equation, given by $xy^{\prime\prime} + (b-x)y^{\prime}-ay=0$. Then in our case for equation \eqref{ecuacion de laguerre ejemplo tesis} its regular solution is ${}_{1}F_{1}(-n;1;x)$ \cite[p.504]{Abramowitz} where $n=\tfrac{\lambda}{2}-1$.\\
If $1-\tfrac{\lambda}{2}=-k\in\mathbb{N}_0$, that is with $\lambda=2k+2$, then $(-k)_{n}=0$ for $n>k$, which means that the series shortens and exactly the polynomial $L_{k}$ remains.\\
%${}_{1}F_{1}(-1;1;x)=1-x=L_{1}(x)$.\\
If $1-\tfrac{\lambda}{2}=a$, with $a \notin \mathbb{Z}_{<0}$, the series is infinite and for $x \to \infty$, ${}_{1}F_{1}(-1;1;x) \sim \tfrac{1}{\Gamma(a)}e^{x}x^{a-1}$ \cite[p.504]{Abramowitz} hence $u=e^{-x}{}_{1}F_{1} \sim x^{-\lambda/2}$ decays, but $\int_0^\infty u^2\,\varphi\,dr=\int_0^\infty u^2 e^{x}\,dx=\infty$, hence $u$ would not be square-integrable. Therefore for the sector $\ell=0$, the orthonormal basis of eigenfunctions and eigenvalues associated with the Laplace-Beltrami operator are of the form 
$$\lambda_{0,k}=2k+2,\qquad u_{0,k}(r)=e^{-r^2/2}\,L_k\!\bigl(\tfrac{r^2}{2}\bigr), \quad k \in \mb{N}_{0}.$$
Therefore $1+\lambda_{0,k} \geq k$, $\forall k \in \mb{N}_{0}$, and there exists a constant $C_{0}>0$ such that $1+\lambda_{0,k}\geq \frac{k}{C_{0}}$.~\\
\end{proofoflemma}
\justify
We analyze the behavior of the summation in each sector $\ell\geq 0$. Using the decomposition of the space $L^{2}(\mb{R}^{2})$ in \eqref{descomposicion del espacio L2}, we take as orthonormal basis of each sector $\mc{L}_{\ell}$ the following 
\begin{equation}\label{base ortonormal de L2 ejemplo tesis}
    \mc{B}_{\ell}=\left\{ \omega_{0,k}=\frac{u_{0,k}}{\sqrt{2\pi}}    \right\} \cup \left\{ \omega_{\ell,k}^{(1)}=\frac{u_{\ell,k}\cos(\ell\theta)}{\sqrt{\pi}}, \omega_{\ell,k}^{(2)}=\frac{u_{\ell,k}\sin(\ell\theta)}{\sqrt{\pi}}     \right\}_{\ell\geq 1}
\end{equation}
these functions satisfy the following equalities
\begin{equation}\label{normas de la base}
\| \omega_{\ell,k}^{(1)} \|_{L^\infty(M)}^2 =\| \omega_{\ell,k}^{(2)} \|_{L^\infty(M)}^2 = \frac{1}{\pi} \| u_{\ell,k} \|_{L^\infty(\mb{R}^{+})}^2,
\qquad
\| \omega_{0,k} \|_{L^\infty(M)}^2 = \frac{1}{2\pi} \| u_{0,k} \|_{L^\infty(\mb{R}^{+})}^2.
\end{equation}
By the previous lemmas we have that $\|\omega_{j}\|_{L^{\infty}(M)} < \infty$ for $j \in J$, therefore, the sum over the index $J$ is finite.
Let us study how the series behaves over $j \in (\mathbb{N} \setminus J)$. In fact
\justify
\underline{\textsc{Case} ($\ell=0$)}: Using the Lemma \ref{lema sector 0 tesis} we have the following
\begin{align*}
    \sum_{k \in (
    \mb{N} \setminus J
    )}(1+\lambda_{0,k})^{-s}\norm{\omega_{0,k}}^{2}_{L^{\infty}(\mb{R}^{2})}&=\frac{1}{2\pi}\sum_{k \in ( \mb{N} \setminus J)}(1+\lambda_{0,k})^{-s}\norm{u_{0,k}}^{2}_{L^{\infty}(\mb{R}^{+})}
    \leq \frac{16}{\pi}\sum_{k \in ( \mb{N} \setminus J)}(1+\lambda_{0,k})^{1-s} \\
    &\leq \frac{16}{\pi}\sum_{k \in ( \mb{N} \setminus J)}\left( \frac{k}{C_{0}}  \right)^{1-s}
    \leq \frac{16}{\pi}\left(\frac{1}{C_{0}}\right)^{1-s}\sum_{\substack{k > C_{0} \\ k \in (\mb{N} \setminus J)}}\left(\frac{1}{k}\right)^{s-1}
\end{align*}
by the $p$-series test this last series is convergent for $s>2$.
\justify
\underline{\textsc{Case} ($\ell=1$)}: In this case we will use the Lemma \ref{lemma control sobre las eigenfunciones ejemplo tesis} and the Lemma \ref{valores aproximados del espectro del ejemplo}, note also that due to the fact that $\ell=|m|$ with $m\in \mb{Z}$ there exist $2$ solutions described in the basis $\mc{B}_{\ell}$ \eqref{base ortonormal de L2 ejemplo tesis} given in this case by $\omega_{1,k}^{(1)}$ and $\omega_{1,k}^{(2)}$, therefore there exists a factor $2$ multiplying the series, indeed
\begin{align*}
    \sum_{k \in (
    \mb{N} \setminus J
    )}(1+\lambda_{1,k})^{-s}\norm{\omega_{1,k}}^{2}_{L^{\infty}(\mb{R}^{2})}&=\frac{2}{\pi}\sum_{k \in ( \mb{N} \setminus J)}(1+\lambda_{1,k})^{-s}\norm{u_{1,k}}^{2}_{L^{\infty}(\mb{R}^{+})}
    \leq \frac{16}{\pi}\sum_{k \in ( \mb{N} \setminus J)}(1+\lambda_{1,k})^{1-s}\\
    &\leq \frac{16}{\pi}\sum_{k \in ( \mb{N} \setminus J)}(2k+\tfrac{3}{2})^{1-s}\leq \frac{16}{2^{s-1}\pi}\sum_{k \in ( \mb{N} \setminus J)}\frac{1}{k^{s-1}}
\end{align*}
The convergence of the last inequality is obtained by the comparison test, for any real $s>2$.
\justify
\underline{\textsc{Case} ($\ell\geq 2$)}: Using the Lemma \ref{valores aproximados del espectro del ejemplo} we can distribute the exponent symmetrically:
$$
(1 + \lambda_{\ell,k})^{1-s} \leq \left(2k + \tfrac{3}{2}\right)^{\frac{1-s}{2}} \left( 1 + \tfrac{\ln \ell}{4} \right)^{\frac{1-s}{2}}.
$$
then applying the Lemma \ref{lemma control sobre las eigenfunciones ejemplo tesis}
\begin{align*}
    \sum_{k \in (
    \mb{N} \setminus J
    )}(1+\lambda_{\ell,k})^{-s}\norm{\omega_{\ell,k}}^{2}_{L^{\infty}(\mb{R}^{2})}&\leq \sum_{\ell \geq 2} \frac{2}{\pi} \cdot \frac{8}{\ell} \sum_{{k \in (
    \mb{N} \setminus J
    )}} (1 + \lambda_{\ell,k})^{1-s} \\
    &\leq \frac{16}{\pi} \underbrace{\left[ \sum_{{k \in (\mb{N} \setminus J)}} \left(2k + \tfrac{3}{2} \right)^{\frac{1-s}{2}} \right]}_{(a)}
    \underbrace{\left[ \sum_{\ell \geq 2} \frac{1}{\ell} \left( 1 + \tfrac{\ln \ell}{4} \right)^{\frac{1-s}{2}} \right]}_{(b)}
\end{align*}
The expression $(a)$ converges by the comparison test for $\tfrac{s-1}{2}>1$,  that is for any real $s>3$. For analyze the convergence of $(b)$ we use the integral test, indeed 
{\fontsize{10.5}{14}
\begin{equation*}
    f(x)=\frac{1}{x}\left( 1 + \frac{\ln x}{4}  \right)^{\frac{1-s}{2}}=\frac{1}{x}\left( \frac{1}{1+\frac{\ln x}{4}}  \right)^{\frac{s-1}{2}}
\end{equation*}}
this function is decreasing, moreover under the change of variable $t=\ln x$ we have that 
{\fontsize{10.5}{14}\selectfont
\begin{equation}
\int_{2}^{\infty}f(x)dx=\int_{\ln 2}^{\infty}\left( 1 + \frac{t}{4}  \right)^{\frac{1-s}{2}}dt
\end{equation}}
\justify
which converges by the integral test for any real $s >3$. Therefore taking a real number $s > 3$ we have that the series over $(\mathbb{N} \setminus J)$ converges.
\end{proofofclaim}
\justify
Thus, the constructed Cartan-Hadamard manifold example \eqref{ejemplo de variedades Cartan-Hadamard} fully satisfies all conditions of Theorem \ref{teorema finnal para el caso no compacto}.

\end{proof}

\newpage

\section*{Acknowledgements} I would like to express my gratitude to my advisor, Dr. Enrique G. Reyes, and to Universidad de Santiago de Chile for their financial support and ANID through Fondecyt Grant No. 1241719.


\begin{thebibliography}{99}


%\vspace{-0.7cm}

%%%%%%%%%%%%% A %%%%%%%%%%%
\bibitem{Abramowitz}
\textsc{Abramowitz, M., Stegun, I.}, \emph{Handbook of Mathematical Functions with Formulas, Graphs, and Mathematical Tables}. U.S. Department of Commerce, National Bureau of Standards, applied mathematics, 1972.

\bibitem{AliprantisBorder}
\textsc{Aliprantis, C., Border, K.},
\emph{Infinite Dimensional Analysis: A Hitchhiker's Guide},
3rd ed.,
Springer,
2006.



%%%%%%%%%%%%%%%%%%%%%%%%%%%%

%%%%%%%%%%%%% B %%%%%%%%%%%%

\bibitem{Backman}
\textsc{Backman, E.}, \emph{The Spectrum of the Laplace-Beltrami Operator on Noncompact Manifolds}, Degree Project in Mathematics, 
KTH Royal Institute of Technology, 2024.
 


\bibitem{Banica}
\textsc{Banica, V., Gonzalez, M., Sáez, M.}, \emph{Some constructions for the fractional Laplacian on noncompact manifolds}, Rev. Mat. Iberoam. \textbf{31} (2015), no. 2, 681–712, European Mathematical Society.

\bibitem{Bonnaille}
\textsc{Bonnaille, V.}, \emph{Harmonic Oscillators with Neumann condition on the half-line}, AIMS’ Journals, 2000.


\bibitem{Boyce}
\textsc{Boyce, W., Diprima, R.}, \emph{Boyce, DiPrima. Elementary Differential Equations and Boundary Value Problems,} 7th edition, John Wiley \& Sons, Inc, 2001.


\bibitem{Brezis}
\textsc{Brezis, H.},\emph{ Functional Analysis, Sobolev Spaces and Partial Differential Equations}, Mathematics Subject Classification (2010), Springer, 2010.


%%%%%%%%%%%%%%%%%%%%%%%%%%

%%%%%%%%%%%% C %%%%%%%%%%%%




\bibitem{Canzani}
\textsc{Canzani, Y.}, \emph{Notes of Analysis on Manifolds via the Laplacian}, Harvart University, Math 253, 2013.


\bibitem{Castellanos 1}
\textsc{Castellanos, S.}, \emph{On Non-Local Equations and Symmetries}, PhD. Thesis, Dep. Matemática, University of Santiago de Chile, 2026.


%%%%%%%%%%%%%%%%%%%%%%%%%%%

%%%%%%%%%%%% D %%%%%%%%%%%%
\bibitem{DoCarmo}
\textsc{Do Carmo, M.P.}, \emph{Riemannian Geometry}. Mathematics: Theory \& Applications,  Birkhäuser, Boston, 1992.

\bibitem{Donnelly}
\textsc{Donnelly, H. and Li, P.}, \emph{Pure point spectrum and negative curvature for noncompact manifolds}, Duke mathematical journal vol 46, No.3, 1979.


%%%%%%%%%%%%%%%%%%%%%%%%%%%

%%%%%%%%%%%% E %%%%%%%%%%%%


\bibitem{Evans}
\textsc{Evans, L.}, \emph{Partial differential equations   }, 2nd edition, American Mathematical Society, Graduate text in Mathematics, Vol.19, 2010.

%%%%%%%%%%%%%%%%%%%%%%%%%%%

%%%%%%%%%%%% F %%%%%%%%%%%%



%%%%%%%%%%%%%%%%%%%%%%%%%%%

%%%%%%%%%%%% G %%%%%%%%%%%%



\bibitem{Gorka 6}
\textsc{Górka, P., Prado, H., Reyes, E.G.}, \emph{On a General Class of Nonlocal Equations}, Springer Basel AG, Annales Henri Poincare 14 (2013), 947–966.

\bibitem{Grigoryan}
\textsc{Grigor'yan, A.},
\emph{Heat Kernel and Analysis on Manifolds},
AMS/IP Studies in Advanced Mathematics, vol. 47, American Mathematical Society, Providence, RI, 2009.


%%%%%%%%%%%%%%%%%%%%%%%%%%%

%%%%%%%%%%%% H %%%%%%%%%%%%



\bibitem{Hebey 2}
\textsc{Hebey, E.}, \emph{Nonlinear Analysis on Manifolds:
Sobolev Spaces and Inequalities}, Courant Institute of Mathematical Sciences and American Mathematical Society, 2000.


\bibitem{HH}
\textsc{Henriquez, H.}, \emph{Fundamentos de Analisis Funcional}, Academia Española, 2012.


\bibitem{Hormander}
\textsc{Hormander, L.}, \emph{The Analysis of Linear Partial Differential Operators III}, Springer, 1985.


%%%%%%%%%%%%%%%%%%%%%%%%%%%

%%%%%%%%%%%% I %%%%%%%%%%%%


%%%%%%%%%%%%%%%%%%%%%%%%%%%

%%%%%%%%%%%% J %%%%%%%%%%%%


%%%%%%%%%%%%%%%%%%%%%%%%%%%

%%%%%%%%%%%% K %%%%%%%%%%%%

\bibitem{Kato}
\textsc{Kato, T.},
\emph{Perturbation Theory for Linear Operators}, ed.2nd,
Springer-Verlag,
1995.

\bibitem{Kostelecky}
\textsc{Kostelecký, V.A., Samuel, S.}, \emph{On a nonperturbative vacuum for the open bosonic string}, 1990.
%%%%%%%%%%%%%%%%%%%%%%%%%%%

%%%%%%%%%%%% L %%%%%%%%%%%%

\bibitem{Lee}
\textsc{Lee, J.}, \emph{Introduction to Smooth Manifolds}, Graduate Texts in Mathematics, Second Edition, Springer 2012.

\bibitem{Lee2}
\textsc{Lee, J.},  \emph{Introduction to Riemannian Manifolds}, 2nd ed. Graduate Texts in Mathematics 176, Springer 2018.



\bibitem{Lunardi}
\textsc{Lunardi, A.}, \emph{Interpolation Theory}, 3th edition, Edizioni della Normale Pisa, Publications of the Scuola Normale Superiore,  2018.
%%%%%%%%%%%%%%%%%%%%%%%%%%%

%%%%%%%%%%%% M %%%%%%%%%%%%


\bibitem{McKean}
\textsc{McKean, H.},
\emph{An upper bound to the spectrum of $\Delta$ on a manifold of negative curvature},
J. Differential Geom. \textbf{4}, 359--366, (1970).

%%%%%%%%%%%%%%%%%%%%%%%%%%%

%%%%%%%%%%%% N %%%%%%%%%%%%




%%%%%%%%%%%%%%%%%%%%%%%%%%%

%%%%%%%%%%%% O %%%%%%%%%%%%

%%%%%%%%%%%%%%%%%%%%%%%%%%%

%%%%%%%%%%%% P %%%%%%%%%%%%


\bibitem{Petersen}
\textsc{Petersen, P.}, \emph{Riemannian Geometry}, 3th Edition, Graduate Texts in Mathematics, Springer, 2016.



%%%%%%%%%%%%%%%%%%%%%%%%%%%

%%%%%%%%%%%% Q %%%%%%%%%%%%

%%%%%%%%%%%%%%%%%%%%%%%%%%%

%%%%%%%%%%%% R %%%%%%%%%%%%


\bibitem{ReedI}
\textsc{Reed, M., Simon, B.}, \emph{Methods of Modern Mathematical Physics I, Functional Analysis}, Academic Press, INC., 1980.


\bibitem{ReedII}
\textsc{Reed, M., Simon, B.}, \emph{Methods of Modern Mathematical Physics II, Fourier Analysis, Self-adjointness}, Academic Press, INC., 1975.


\bibitem{ReedIV}
\textsc{Reed, M., Simon, B.}, \emph{Methods of Modern Mathematical Physics IV, Analysis of Operators}, Academic Press, INC., 1980.


\bibitem{Rosenberg}
\textsc{Rosenberg, S.}, \emph{The Laplacian on a Riemannian Manifold}, London Mathematical Society, 1997.

\bibitem{Rudin1}
\textsc{Rudin, W.}, \emph{Functional Analysis}, Second Edition, Mc-Graw-Hill, 1991.
%%%%%%%%%%%%%%%%%%%%%%%%%%%

%%%%%%%%%%%% S %%%%%%%%%%%%
\bibitem{Schmudgen}
\textsc{Schm\"udgen, K.}, \emph{Unbounded
Self-adjoint Operators on Hilbert Space}, Springer, 2012.



%%%%%%%%%%%%%%%%%%%%%%%%%%%
%%%%%%%%%%%% T %%%%%%%%%%%%


\bibitem{Taylor III}
\textsc{Taylor, M.}, \emph{Partial Differential Equations III Nonlinear Equations}, Vol. III, Second Edition, Springer, 2011. 



\bibitem{Teschl1}
\textsc{Teschl, G.}, \emph{Mathematical Methods in Quantum Mechanics: With Applications to Schrödinger Operators}, Second edition, Graduate Studies in Mathematics 140, AMS, 2009.

\bibitem{Teschl}
\textsc{Teschl, G.}, \emph{Ordinary Differential Equations and Dynamical Systems}, Graduate Studies in Mathematics 140, AMS, 2012.



%%%%%%%%%%%%%%%%%%%%%%%%%%%
%%%%%%%%%%%% U %%%%%%%%%%%%


%%%%%%%%%%%% V %%%%%%%%%%%%

\bibitem{Vladimirov 2}
\textsc{Vladimirov, V.S., Volovich, Ya.I.}, \emph{On the Nonlinear Dynamics equation in p-adic string theory}. Teoret. Mat. Fiz. \textbf{138}, 355–368 (2004).

%%%%%%%%%%%%%%%%%%%%%%%%%%%
%%%%%%%%%%%% W %%%%%%%%%%%%


%%%%%%%%%%%%%%%%%%%%%%%%%%%
%%%%%%%%%%%% X %%%%%%%%%%%%



%%%%%%%%%%%% Y %%%%%%%%%%%%


%%%%%%%%%%%%%%%%%%%%%%%%%%%
%%%%%%%%%%%% Z %%%%%%%%%%%%
%%%%%%%%%%%%%%%%%%%%%%%%%%%




\end{thebibliography}
\end{document}